\documentclass[10pt,reqno, fullpage]{amsart}

\usepackage{amsmath,amssymb,amsfonts,mathrsfs,verbatim,enumitem,pstricks,amsthm, graphicx}
\usepackage[all]{xy}
\usepackage{centernot, xr, accents}
\usepackage{hyperref}

\usepackage{amsmath,amssymb}
\usepackage[all]{xy}
\usepackage{hyperref}
\usepackage{verbatim}

\usepackage[bottom]{footmisc}
\usepackage{lipsum,cleveref}

\usepackage{hyperref}
\hypersetup{
	colorlinks,
	citecolor=red,
	filecolor=black,
	linkcolor=blue,
	urlcolor=black
}

\theoremstyle{plain}
\newtheorem{thm}{Theorem}[section]

\newtheorem{prp}[thm]{Proposition}

\newtheorem{claim}[thm]{Claim}

\theoremstyle{definition}
\newtheorem{rem}{Remark} 

\newtheorem{defn}[thm]{Definition}

\def \hf{\hspace*{0.5cm}}                      
\def\bge{\begin{equation}}                
\def\ede{\end{equation}}                
\def\bgd{\begin{displaymath}}         
\def\edd{\end{displaymath}}            
\def\bgee{\begin{equation*}}           
\def\edee{\end{equation*}}           

\def\lra{\longrightarrow}

\def\lan{\langle}
\def\ran{\rangle}

\def\BA{\begin{eqnarray}}
\def\EA{\end{eqnarray}}
\def\BAA{\begin{eqnarray*}}
\def\EAA{\end{eqnarray*}}
\def\Bal{\begin{align*}}
\def\Eal{\end{align*}}

\def \P{\mathbb{P}}
\def \A{\mathcal{A}}

\def \Num{N}
\def \PP{\mathcal{P}}

\def \B{\mathcal{B}}

\def \25node{A_5} 
\def \62node{A_6}

\def \D{\mathcal{D}}
\def \DD{\mathcal{D}}

\def \X{\mathfrak{X}}

\def \DD{\hat{\D}}

\def \N{\nabla}

\def \hxt2{\hat{x}_{t_2}} 
\def \hyt2{\hat{y}_{t_2}}

\def \hf{\hspace*{0.5cm}}                      
\def\bge{\begin{equation}}                
\def\ede{\end{equation}}                
\def\bgd{\begin{displaymath}}         
\def\edd{\end{displaymath}}            
\def\bgee{\begin{equation*}}           
\def\edee{\end{equation*}}           

\def \P{\mathbb{P}}
\def \PP{\mathcal{P}}
\def \X{X} 
\def \A{A}
\def \D{D}

\def \N{N}
\def \lra{\longrightarrow}

\def \DD{\mathcal{D}}
\def \lan{\langle} 
\def \ran{\rangle} 
\def \B{\mathcal{B}}
\def \hf{\hspace*{0.5cm}}

\begin{document}


\title[Counting planar curves in $\mathbb{P}^3$ with degenerate singularities]{Counting planar curves in $\mathbb{P}^3$ with degenerate singularities}

\author[Nilkantha Das]{Nilkantha Das}
\address{School of Mathematical Sciences, National Institute of Science Education and Research, HBNI, Bhubaneswar, Odisha- 752 050, India.}

\email{nilkantha.das@niser.ac.in }
\author[Ritwik Mukherjee]{Ritwik Mukherjee}

\address{School of Mathematical Sciences, National Institute of Science Education and Research, HBNI, Bhubaneswar, Odisha- 752 050, India.}
\email{ritwikm@niser.ac.in}


\date{}


\begin{abstract} 
In this paper, we consider the following question: how many degree $d$ curves are there in $\mathbb{P}^3$ 
(passing through the right number of generic lines and points), 
whose image lies inside a $\mathbb{P}^2$, having 
$\delta$ nodes and one singularity of codimension $k$.
We obtain an explicit formula for this number 
when $\delta+k \leq 4$  
(i.e. the total codimension of the singularities is not more than four). 
We use a topological method to compute the degenerate contribution to the Euler class; it is an 
extension of the method that originates in the paper by A.~Zinger (\cite{Zin}) 
and which is further pursued by S.~Basu and the second author in \cite{R.M}, \cite{BM13_2pt_published} 
and \cite{BM8}. Using this method, 
we have obtained formulas when the 
singularities present are more degenerate than nodes (such as cusps, tacnodes and triple points). 
When the singularities are only nodes, we have verified that our answers are consistent with those  obtained by 
by S.~Kleiman and R.~Piene (in \cite{KP2}) and by 
T.~Laarakker (in \cite{TL}). 
We also verify that our answer for the characteristic number of planar cubics with a cusp 
and the number of planar quartics with two nodes and one cusp is consistent with the answer obtained by  
R.~Singh and the second author (in \cite{RS}), where they compute the characteristic 
number of rational planar curves in $\mathbb{P}^3$ with a cusp.
We also verify some of the numbers predicted by the conjecture made by Pandharipande in \cite{RPDeg}, 
regarding the enumerativity of BPS numbers for $\mathbb{P}^3$.  
\end{abstract}

\maketitle

\tableofcontents

\section{Introduction}
One of the most fundamental and studied problems in enumerative geometry is the following: what is 
the number of 
degree $d$ curves in $\mathbb{P}^2$ that have $\delta$ distinct nodes and that pass through $\frac{d(d+3)}{2} - \delta$ generic points? 
A more general question is to enumerate the characteristic number of curves that have more degenerate singularities.
To make this precise, let us make the following definition. 
\begin{defn}
Let $f:\mathbb{P}^2  \longrightarrow \mathcal{O}(d)$ be a holomorphic section. 
A point $q \in f^{-1}(0)$ is said to have a singularity of type $A_k$ or $D_k$
if there exists a coordinate system
$(x,y) :(U,q) \longrightarrow (\mathbb{C}^2,0)$ such that $f^{-1}(0) \cap U$ is given by 
\begin{align*}
A_{k \geq 1}: y^2 + x^{k+1}   &=0  \qquad \textnormal{and} \qquad D_{k \geq 4}: y^2 x + x^{k-1} =0.  
\end{align*}
\end{defn}
\hf \hf In more common terminology, $q$ is a 
{\it simple node} (or just node) if its singularity type is $A_1$; 
a {\it cusp} if its type is $A_2$; a {\it tacnode} if its type is $A_3$ 
and an {\it ordinary triple point} if its type is $D_4$. 

\begin{rem}
We will frequently use the phrase ``a singularity of codimension $k$''. Roughly speaking, 
this refers to the number of conditions having that singularity imposes on the space of curves. 
More precisely, it is the expected codimension of the equisingular strata.  
Hence, a singularity of type $A_k$ or $D_k$ 
is a singularity of codimension $k$. 
\end{rem}
\hf \hf A classical question in enumerative geometry is this:  
what is $N_d(A_1^{\delta})$, the number of 
degree $d$ curves in $\mathbb{P}^2$, that have 
$\delta$ distinct (ordered) nodes, that pass through $\frac{d(d+3)}{2}-\delta$ generic points? 
More generally, one can ask what is $N_d(A_1^{\delta} \mathfrak{X})$, the number of 
degree $d$ curves in $\mathbb{P}^2$, that have 
$\delta$ distinct (ordered) nodes and one singularity of type $\mathfrak{X}$, that pass through $\frac{d(d+3)}{2}-\delta-c_{\mathfrak{X}}$ generic points, 
where $c_{\mathfrak{X}}$ is the codimension of the singularity $\mathfrak{X}$? \\ 
\hf \hf The question of computing $N_d(A_1^{\delta})$ and $N_d(A_1^{\delta} \mathfrak{X})$ 
has been studied for a very long time starting with Zeuthen (\cite{Zu}) more than a hundred years ago. 
It has been studied extensively in the last thirty years from various perspectives by numerous mathematicians including amongst others,  
Z.~Ran (\cite{Ran1}, \cite{Ran}), I.~Vainsencher (\cite{Van}), L.~Caporaso and J.~Harris (\cite{CH}), M.~Kazarian (\cite{Kaz}), S.~Kleiman and R.~Piene (\cite{KP1}), 
D.~Kerner (\cite{Ker1} and \cite{Ker2}), 
F.~Block (\cite{FB}),
Y.~Tzeng and J.~Li (\cite{Tz}, \cite{Tzeng_Li}), M.~Kool, V.~Shende and R.~Thomas (\cite{KST}), S.~Fomin and G.~Mikhalkin (\cite{FoMi}), 
G.~Berczi (\cite{Berczi}) and S.~Basu and R.~Mukherjee (\cite{R.M}, \cite{BM13_2pt_published} and \cite{BM8}).\\
\hf \hf This problem  
motivates a natural generalization considered by Kleiman and Piene in \cite{KP2}, where they study the 
enumerative geometry of nodal curves in a moving family of 
surfaces (i.e. a fiber bundle version of the earlier question). 
More recently, this question has been studied further by  T.Laarakker in \cite{TL}. \\ 
\hf \hf Let us now state the question more precisely. 
We define a \textbf{planar curve} in $\mathbb{P}^3$ to be a curve in $\mathbb{P}^3$, whose image lies inside some $\mathbb{P}^2$. 
Let us define 
\begin{align*}
N_d^{\textnormal{Planar}, \mathbb{P}^3 }(A_1^{\delta} \mathfrak{X}; r,s) 
\end{align*}
to be the number of planar degree $d$ curves in $\mathbb{P}^3$, intersecting $r$ lines and passing through $s$ points, 
and having $\delta$ distinct nodes and one singularity of type $\mathfrak{X}$, where  
$r+2s = \frac{d(d+3)}{2}+3-(\delta +c_{\mathfrak{X}})$ and 
$c_{\mathfrak{X}}$ is the codimension of the singularity $\mathfrak{X}$.  
The result of S. Kleiman and R. Piene (\cite{KP2}) can be used to obtain a formula for  
$N_d^{\textnormal{Planar}, \mathbb{P}^3 }(A_1^{\delta}; r,s)$, if $\delta \leq 8$ (see section \ref{KP_check} for details). 
In \cite{TL}, T.Laarakker
obtains a formula for $N_d^{\textnormal{Planar}, \mathbb{P}^3 }(A_1^{\delta}; r,s)$, for all $\delta$. \\ 
\hf \hf The main result of this paper is as follows: 
\begin{thm}
\label{main_thm}
Let $\mathfrak{X}$ be a singularity of codimension $c_{\mathfrak{X}}$ 
and $\delta$ a  non negative integer. 
We obtain an explicit formula for $N_d^{\textnormal{Planar}, \mathbb{P}^3}(A_1^{\delta} \mathfrak{X}, r,s)$, when $\delta + c_{\mathfrak{X}} \leq 4$, 
provided $d \geq d_{\textnormal{min}}$, where $d_{\textnormal{min}}:= c_{\mathfrak{X}} + 2\delta.$
\end{thm}
In section \ref{low_degree_checks}, we verify that when the singularities present are only nodes, 
our answers agree with the answers obtained by S.~Kleiman and R.~Piene (in \cite{KP2}) and by T.~Laarakker (in \cite{TL}). 
We also verify some of the numbers predicted by the conjecture made by R.~Pandharipande in \cite{RPDeg}, 
regarding the enumerativity of the BPS numbers for $\mathbb{P}^3$.  \\
\hf \hf Very recently, a stable map version of this question has been studied by the second author, A.~Paul and R.~Singh (in \cite{MPS}). 
In that paper, the authors find a formula for the characteristic number of planar genus zero (rational) degree $d$-curves in $\mathbb{P}^3$. 
Building up on the results of that 
paper, the second author of this paper and R.~Singh
obtain a formula for  the characteristic number of planar genus zero (rational) degree $d$-curves in $\mathbb{P}^3$ 
having a cusp (in \cite{RS}). In section \ref{low_degree_checks}, we also verify that our formula for 
$N_d^{\textnormal{Planar}, \mathbb{P}^3 }(A_2; r,s)$ and $N_d^{\textnormal{Planar}, \mathbb{P}^3 }(A_1^2 A_2; r,s)$ is logically 
consistent with the formula obtained in \cite{RS},
when $d=3$ and $d=4$ respectively.
\begin{rem}
\label{KP_remark}
In \cite[Theorem 1.2]{KP1}, the authors compute the corresponding numbers $N(A_1^{\delta} \mathfrak{X})$ for a fixed surface, 
while  in \cite{KP2} an algorithm is developed to compute $N(A_1^{\delta})$ for a family of surfaces. 
It ought to be possible to
generalize the algorithm developed in \cite{KP2} to higher singularities
and  
compute all the numbers obtained by 
\Cref{main_thm} (this has been point out to us by S. Kleiman \cite{KP3}).
\end{rem}


\section{Setup and Notation} 
\label{notation}
\noindent Let us now describe the setup develop some notation to obtain a formula for the numbers stated in \Cref{main_thm}. 
Our basic objects are planar degree degree $d$ curves in $\mathbb{P}^3$, i.e. degree $d$ curves in $\mathbb{P}^3$ 
whose image lies inside a $\mathbb{P}^2$. 
Let us denote the dual of $\mathbb{P}^3$ by 
$\widehat{\mathbb{P}}^3$; this is the space of $\mathbb{P}^2$ inside $\mathbb{P}^3$. 
An element of $\widehat{\mathbb{P}}^3$ can be thought of as 
a nonzero linear functional $\eta: \mathbb{C}^4 \longrightarrow \mathbb{C}$ upto scaling (i.e., it is the projectivization of the dual of $\mathbb{C}^4$).  
Given such an $\eta$, we  define the projectivization of its zero set as $\mathbb{P}^2_{\eta}$. In other words, 
\begin{align*}
\mathbb{P}^2_{\eta} &:= \mathbb{P}(\eta^{-1}(0)). 
\end{align*}
Note that this $\mathbb{P}^2_{\eta}$ is a subset of $\mathbb{P}^3$.\\ 
\hf\hf Next, given a positive integer $\delta$,  let us define 
\begin{align*}
\mathcal{S}_{\delta} &:= \{ ([\eta], q_1, \ldots, q_{\delta}) \in \widehat{\mathbb{P}}^3 \times (\mathbb{P}^3)^{\delta}: \eta(q_1) =0, \ldots, \eta(q_{\delta}) =0\}. 
\end{align*}
Clearly $\mathcal{S}_{\delta}$ is a fiber bundle over $\widehat{\mathbb{P}}^3$ with fiber $(\P^2)^{\delta}$. This is a plane in $\mathbb{P}^3$ and a collection of $\delta$ points that lie on that plane.
We  will often abbreviate $\mathcal{S}_1$ as $\mathcal{S}$. 
Let us consider the section of the following line bundle 
induced by the evaluation map, i.e.  
\begin{align*}
\textnormal{ev}:\widehat{\mathbb{P}}^3 \times \mathbb{P}^3 \longrightarrow \gamma_{\widehat{\mathbb{P}}^3}^{*}\otimes \gamma_{\mathbb{P}^3}^*, \qquad 
\textnormal{given by} \qquad \{\textnormal{ev}([\eta], [q])\}(\eta \otimes q):= \eta(q), 
\end{align*}
where $\gamma_{\widehat{\mathbb{P}}^3}^{*}$ and $\gamma_{\mathbb{P}^3}^*$ are dual of the tautological line bundles over $\widehat{\mathbb{P}}^3$ 
and $\mathbb{P}^3$ respectively (or equivalently $\mathcal{O}_{\widehat{\mathbb{P}}^3}(1)$ and $\mathcal{O}_{\mathbb{P}^3}(1)$ respectively).  
Note that 
\begin{align}
\mathcal{S}&= \textnormal{ev}^{-1}(0). \label{s_ev} 
\end{align}
Next, let us denote $\mathcal{D} \longrightarrow \widehat{\mathbb{P}}^3$ to be 
a fiber bundle over $\widehat{\mathbb{P}}^3$, such that the fiber over each $[\eta] \in \widehat{\mathbb{P}}^3$ 
is the space of degree $d$ curves in $\mathbb{P}^2_{\eta}$. Next, we note that 
$\widehat{\mathbb{P}}^3$ is naturally isomorphic to $\mathbb{G}(3,4)$. Let us denote $\gamma_{3,4} \longrightarrow \mathbb{G}(3,4)$ 
to be the tautological three plane bundle over the Grassmannian. Hence, via this isomorphism we note that   
\begin{align*}
\mathcal{D} & \approx \mathbb{P}(\textnormal{Sym}^{d} \gamma_{3,4}^{*}) \longrightarrow \widehat{\mathbb{P}}^3.
\end{align*}
Hence, $\mathcal{D}$ is a fiber bundle over $\widehat{\mathbb{P}}^3$, whose fibers are isomorphic to $\mathbb{P}^{\frac{d(d+3)}{2}}$.
An element of $\mathcal{D}$ will be denoted by $([f],  [\eta])$; this means that $f$ is a homogeneous degree $d$-polynomial defined on $\mathbb{P}^2_{\eta}$. \\
\hf \hf Next, given a positive integer $\delta$, let us define 
\begin{align*}
\mathcal{S}_{\mathcal{D}_{\delta}} := \{ ([f], [\eta], q_1, \ldots, q_{\delta}) 
\in \mathcal{D} \times (\mathbb{P}^3)^{\delta}&: ([\eta], q_1, \ldots, q_{\delta}) \in \mathcal{S}_{\delta}\}.
\end{align*}
Note that $\mathcal{S}_{\mathcal{D}_{\delta}}$ can be considered as pull back bundle of $\mathcal{D}$ via the fiber bundle map $\pi: \mathcal{S}_{\delta} \rightarrow \widehat{\mathbb{P}}^3$, i.e. the following diagram 
\begin{displaymath}
\xymatrix{
\mathcal{S}_{\mathcal{D}_{\delta}} \ar[d]_-{\pi_{\mathcal{D}}^*} \ar[r] & \mathcal{D} \ar[d]^-{\pi_{\mathcal{D}}} \\
\mathcal{S}_{\delta} \ar[r]_-{} & \widehat{\mathbb{P}}^3
}
\end{displaymath}
commutes. We will  abbreviate $\mathcal{S}_{\mathcal{D}_{1}}$ as $\mathcal{S}_{\mathcal{D}}$. 
Next, let $X_1, X_2, \ldots, X_{\delta}$ be subsets of $\mathcal{S}_{\mathcal{D}}$. 
We define 
\begin{align*}
X_1 \circ X_2 \circ \ldots \circ X_{\delta} := \{ ([f], [\eta], q_1, \ldots, q_{\delta}) \in \mathcal{S}_{\mathcal{D}_{\delta}}& : ([f],  [\eta], q_i) \in X_i ~~\forall 
~i = 1 ~\textnormal{to} ~\delta \qquad \textnormal{and} \\ 
                                        & \qquad \qquad \qquad \qquad \quad q_i \neq q_j \qquad \textnormal{if} ~~i \neq j \}.    
\end{align*}
We will make the following abbreviation
\begin{align*}
X_1^{\delta_1} \circ X_2^{\delta_2} \ldots X_m^{\delta_m} &:= 
\underbrace{X_1 \circ \ldots \circ X_1}_{\textnormal{$\delta_1$ times}} \circ 
\underbrace{X_2 \circ \ldots \circ X_2}_{\textnormal{$\delta_2$ times}} \circ  \ldots  
\circ \underbrace{X_m \circ \ldots \circ X_m}_{\textnormal{$\delta_m$ times}}.
\end{align*}
When $\delta_i =1$, we will omit writing the superscript. For example, 
\begin{align*}
X_1 \circ X_2^{3} \circ X_3 &= X_1^1 \circ X_2^{3} \circ X_3^1 = X_1 \circ X_2 \circ X_2 \circ X_2 \circ X_3. 
\end{align*}
Next, let $\mathfrak{X}$ be a singularity of a given type. 
We will also denote $\mathfrak{X}$ to be the space of curves and a marked point
such that the curve has a singularity of type $\mathfrak{X}$ at the marked point. 
More precisely, 
\begin{align*}
\mathfrak{X} &:= \{ ([f], [\eta], q) \in \mathcal{S}: ~f ~~\textnormal{has a signularity of type $\mathfrak{X}$ at $q$} \}. 
\end{align*}
For example, 
\begin{align*}
A_2 &:= \{ ([f], [\eta], q) \in \mathcal{S}: ~f ~~\textnormal{has a signularity of type $A_2$ at $q$} \}. 
\end{align*}
For example, $A_1^{2} \circ A_2$ is the space of curves with 
three ordered points, where the curve has a simple node at the first two points and a cusp at the 
last point and all the three points are distinct. 
Similarly, $A_1^{2} \circ \overline{A}_2$
is the space of curves with 
three distinct ordered points, where the curve has a simple node at the first two points and a 
singularity at least as degenerate as a cusp at the 
last point; the curve could have a tacnode at the last marked point 
(here  $\overline{X}$ indicates the closure of $X$). \\ 
\hf \hf Next, consider the following rank two vector bundle  $\pi:W \longrightarrow \mathcal{S}$, where the fiber over each point 
$([\eta], q)$ is the tangent space of $\mathbb{P}^2_{\eta}$ at the point $q$, i.e. 
\begin{align}
\pi^{-1}([\eta],  q) &:= T\mathbb{P}^2_{\eta}|_{q}. \label{W_defn}
\end{align}
Let $W_{\DD}\longrightarrow \mathcal{S}_{\DD}$ denote 
the pullback of $W$ to $\mathcal{S}_{\DD}$ and let $\mathbb{P}W_{\DD} \longrightarrow \mathcal{S}_{\DD}$ 
denote the projectivization of $W_{\DD}$.
We can now define the space of curves having a singularity singularity of certain type together with a direction, i.e. if $\mathfrak{X}$ be singularity of a given type, then define 
\begin{align*}
\widehat{\mathfrak{X}} & :=  \{ ([f], [\eta], l_q) \in \mathbb{P} W_{\mathcal{D}} : f \text{ has a singularity of type } \mathfrak{X} \text{ at } q \}.
\end{align*} 
We can also define the space of curves with a singularity and a specific direction along which certain directional derivatives vanish, i.e. 
\begin{align*}
\mathcal{P}\A_k &:= \{ ([f], [\eta], l_q) \in \mathbb{P} W_{\mathcal{D}}: ([f], [\eta], q) \in \A_k, ~~\nabla^2 f|_{q}(v, \cdot) =0 
\qquad \forall v \in l_q\} \qquad \textnormal{if ~~$k \geq 2$}. 
\end{align*}
For example, $\PP A_2$ is the space of curves with a marked point and a marked direction, such that 
the curve has a cusp at the marked point and the marked direction belongs to the kernel of the Hessian. 
Note that  
the projection map $\pi : \PP A_k \lra A_k$ 
is one to one. 
Next, let us define 
\begin{align*}
\mathcal{P}\A_1 &:= \{ ([f], [\eta], l_q) \in \mathbb{P} W_{\mathcal{D}}: ([f], [\eta], q) \in \A_1, ~~\nabla^2 f|_q(v, v) =0 
\qquad \forall v \in l_q\} \qquad \textnormal{and} \\
\mathcal{P}\D_4 &:= \{ ([f], [\eta], l_q) \in \mathbb{P} W_{\mathcal{D}}: ([f], [\eta], q) \in \D_4, ~~\nabla^3 f|_q(v, v, v) =0 
\qquad \forall v \in l_q\}. 
\end{align*}
In other words, $\PP A_1$ is the space of curves with a marked point and a marked direction, such that 
the curve has a node at the marked point and the second derivative along the marked direction vanishes. 
Note that there are two such distinguished directions. Hence, the projection map  
$\pi : \PP A_1 \lra A_1$ 
is two to one. Similarly, the projection map $\pi : \PP D_4 \lra D_4$ is three to one.\\ 
\hf \hf Next, let 
$\mathcal{S}_{\mathcal{D}_{\delta}}\times_{\mathcal{D}} \mathbb{P} W_{\mathcal{D}}$ 
denote 
the fibered product of $\mathcal{S}_{\mathcal{D}_{\delta}}$ and $\mathbb{P} W_{\mathcal{D}}$ over $\mathcal{D}$ via the natural forgetful map. It can be considered as a fiber bundle over $\widehat{\mathbb{P}}^3$ whose fiber over each point $[\eta] \in \widehat{\mathbb{P}}^3$ is 
\begin{align*}
\mathbb{P}(H^0(\mathcal{O}(d), \mathbb{P}^2_{\eta})) \times (\mathbb{P}^2_{\eta})^{\delta} \times \mathbb{P}(T \mathbb{P}^2_{\eta}). 
\end{align*}
Let 
$\pi:\mathcal{S}_{\mathcal{D}_{\delta}}\times_{\mathcal{D}} \mathbb{P} W_{\mathcal{D}} \longrightarrow 
\mathcal{S}_{\mathcal{D}_{\delta+1}}$ denote the projection map. 
If $S$ is a subset of $\mathcal{S}_{\mathcal{D}_{\delta+1}}$, then we define 
\begin{align}
\widehat{S} &:= \{ ([f], [\eta], q_1, \ldots, q_{\delta}, l_{q_{\delta+1}}) \in 
\mathcal{S}_{\mathcal{D}_{\delta}}\times_{\mathcal{D}} \mathbb{P} W_{\mathcal{D}}: 
([f], [\eta], q_1, \ldots, q_{\delta+1}) \in S \} = \pi^{-1}(S). \label{hat_S_defn}
\end{align}
Finally, if $S_1, \ldots S_n$ are subsets of $\mathcal{S}_{\mathcal{D}}$ and 
$T$ is a subset of $\mathbb{P}W_{\mathcal{D}}$, then we define 
\begin{align*}
S_1 \circ S_2 \circ \ldots \circ S_{\delta}\circ T &:= \{ ([f], [\eta], q_1, \ldots, q_{\delta}, l_{q_{\delta+1}}) \in 
\mathcal{S}_{\mathcal{D}_{\delta}}\times_{\mathcal{D}} \mathbb{P} W_{\mathcal{D}}:  
([f], q_1) \in S_1, \ldots, ([f], q_{\delta}) \in S_{\delta}, \\ 
& \qquad \qquad \qquad \qquad \qquad \qquad \qquad \qquad \qquad \qquad ([f], l_{q_{\delta+1}}) \in T \qquad \textnormal{and}\\ 
& \qquad \qquad \qquad \qquad \qquad \qquad  \qquad \qquad \qquad \qquad q_1, \ldots, q_{\delta}, q_{\delta+1} 
\qquad \textnormal{are all distinct}\}.
\end{align*}
As an example, $A_1^{2} \circ \PP A_2$ is the space of curves with 
three distinct ordered points, where the curve has a simple node at the first two points and a cusp at the 
last point and a distinguished direction at the last marked point, 
such that the Hessian vanishes along that direction. 

\section{Cohomology ring structure of projective fiber bundles} 
\label{coh_ring}
We now recapitulate some basic facts about the cohomology ring of the various spaces we will encounter. 
First, we recall that via the annihilation map, $\widehat{\mathbb{P}}^3$ is isomorphic to $\mathbb{G}(3,4)$. 
Via this isomorphism, we can think of $a$ (which is actually a generator of $H^*(\widehat{\mathbb{P}}^3)$) 
as a generator of $H^*(\mathbb{G}(3,4))$. We note that 
\begin{align*}
c(\gamma_{3,4}^*) & = 1 + a  + a^2 + a^3. 
\end{align*}
Next, using the splitting principle, we conclude that 
\begin{align}
c(\textnormal{Sym}^d \gamma_{3,4}^*) & = 1 + s_1 a + s_2 a^2 + s_3 a^3, \qquad \textnormal{where} \\ 
s_1 & := \frac{d(d+1)(d+2)}{6}, \nonumber \qquad 
s_2 := \frac{d(d+1)(d+2)(d+3)(d^2+2)}{6} \qquad \textnormal{and} \nonumber \\ 
s_3 &:= \frac{d(d+1)(d+2)(d+3)(d^2+2)(d^3+3d^2+12d+12)}{1296}. \label{si}
\end{align}
Notice that $\mathcal{D}= \mathbb{P}(\textnormal{Sym}^d \gamma_{3,4}^*)$, is a $\mathbb{P}^{n-1}$ 
bundle, where 
\begin{align}
n:= 1 + \frac{d(d+3)}{2}. \label{n_dim_vec_sp} 
\end{align}
Hence, we conclude (by the Leray Hirsch Theorem) that 
the cohomology ring structure of $\mathcal{D}$ is given by 
\begin{align}
H^*(\mathcal{D}) & \approx \frac{\mathbb{Z}[a, \lambda]}{\langle a^4, ~~\lambda^n + s_1 a\lambda^{n-1} + s_2 a^2\lambda^{n-2} + s_3 a^3\lambda^{n-3}\rangle}, 
\label{ring_str}
\end{align}
where $\gamma_{\DD} \longrightarrow \mathbb{P}(\textnormal{Sym}^d \gamma_{3,4}^*)$ denotes the tautological line bundle 
and $\lambda:= c_1(\gamma_{\DD}^*)$.

\section{Intersection Numbers}
Let $\gamma_{W}\longrightarrow \mathbb{P} W$ denote the tautological line bundle over 
the projective bundle $\mathbb{P}W \longrightarrow \mathcal{S}$. We denote 
$\lambda_{W} := c_1(\gamma_{W}^*)$ 
and $H$ to be the standard generator of $H^*(\mathbb{P}^3)$
(i.e. the class of a hyperplane in $\mathbb{P}^3$).  \\ 
\hf \hf We are now in a position to define a few numbers. 
Since we will primarily be dealing with planar degree $d$-curves in $\mathbb{P}^3$, we will usually use the prefix 
$N$ as opposed to the more elaborate $N_d^{\textnormal{Planar},\mathbb{P}^3}$. If there is a chance for confusion, we will 
use the latter notation.\\ 
\hf \hf We will occasionally be dealing with curves in $\mathbb{P}^2$. In such a case we will use the notation 
$N_d^{\mathbb{P}^2}$; we will never use $N$ in such a case. Let us now define 
\begin{align}
\Num(A_1^{\delta} \mathfrak{X}, r, s, n_1, n_2, n_3) &:= \langle a^{n_1} \lambda^{n_2} \pi_{\delta+1}^*H^{n_3}, 
~~[\overline{A_1^{\delta} \circ \mathfrak{X}}]\cap \mathcal{H}_L^r \cap \mathcal{H}_p^s \rangle,\\
\Num(A_1^{\delta} \PP \mathfrak{X}, r, s, n_1, n_2, n_3, \theta) &:= 
\langle a^{n_1} \lambda^{n_2} \pi_{\delta+1}^*H^{n_3} \lambda_{W}^{\theta}, 
~~[\overline{A_1^{\delta} \circ \PP \mathfrak{X}}]\cap \mathcal{H}_L^r \cap \mathcal{H}_p^s \rangle \qquad \textnormal{and}  \\ 
\Num(A_1^{\delta} \widehat{\mathfrak{X}}, r, s, n_1, n_2, n_3, \theta) &:= 
\langle a^{n_1} \lambda^{n_2} \pi_{\delta+1}^*H^{n_3} \lambda_W^{\theta}, 
~~[\overline{A_1^{\delta} \circ \widehat{\mathfrak{X}}}]\cap \mathcal{H}_L^r \cap \mathcal{H}_p^s \rangle. 
\end{align}
Here $\pi_i$
denotes the projection onto the 
$i^{\textnormal{th}}$-point.\\
\hf \hf Next, we note that if $\theta \geq 2$, then 
\begin{align}
\Num(A_1^{\delta} \PP \mathfrak{\X}, r, s, n_1, n_2, n_3, \theta) & = -3 \Num(A_1^{\delta} \PP \mathfrak{\X}, r, s, n_1, n_2, n_3+1, \theta-1)  \nonumber \\ 
                                                                  & + \Num(A_1^{\delta} \PP \mathfrak{\X}, r, s, n_1+1, n_2, n_3, \theta-1)\nonumber \\
                                                                  & -\Num(A_1^{\delta} \PP \mathfrak{\X}, r, s, n_1+2, n_2, n_3, \theta-2) \nonumber \\ 
                                                                  & + 2\Num(A_1^{\delta} \PP \mathfrak{\X}, r, s, n_1+1, n_2, n_3+1, \theta-2) \nonumber \\ 
                                                                  & -3 \Num(A_1^{\delta} \PP \mathfrak{\X}, r, s, n_1, n_2, n_3+2, \theta-2). \label{lm_reduce}
\end{align}
This is because 
\begin{align*}
\lambda_{W}^2 & = -c_1(W) \lambda_W - c_2(W) \implies \lambda_{W}^2 =  -(3H-a) \lambda_{W} -(a^2-2aH+ 3 H^2).
\end{align*}
\noindent The Chern classes $c_1(W)$ and $c_2(W)$ are given by \cref{ciW}. Next, we note that 
\begin{align}
\Num(\A_1^{\delta} \mathfrak{X}, r, s, n_1, n_2, n_3) & = \frac{1}{\textnormal{deg}(\pi)} 
\Num(\A_1^{\delta} \PP \mathfrak{X}, r, s, n_1, n_2, n_3, 0),   \label{deg_to_one_up_to_down}
\end{align}
where $\textnormal{deg}(\pi)$ is the degree of the projection map $\pi:\PP \mathfrak{X}\lra\mathfrak{X}$. 
We remind the reader that 
the degree is one when $\mathfrak{X} = A_{k \geq 2}$, it is two when 
$\mathfrak{X} = A_1$ and it is three when $\mathfrak{X} = D_4$. \\
We also note that
\begin{align}
\Num(A_1^{\delta} \widehat{\mathfrak{X}}, n_1, n_2, n_3, \theta) & = 0 \qquad \textnormal{if} \qquad \theta =0, \nonumber \\ 
\Num(A_1^{\delta} \widehat{\mathfrak{X}}, n_1, n_2, n_3, \theta) & = \Num(A_1^{\delta} \mathfrak{X}, n_1, n_2, n_3) 
\qquad \textnormal{if} \qquad \theta =1 \qquad 
\textnormal{and} \nonumber \\ 
\Num(A_1^{\delta} \widehat{\mathfrak{X}}, n_1, n_2, n_3, \theta) & = 
\Num(A_1^{\delta} \widehat{\mathfrak{X}}, n_1, n_2+1, n_3, \theta-1) 
-\Num(A_1^{\delta} \widehat{\mathfrak{X}}, n_1, n_2, n_3+1, \theta-2) \nonumber \\ 
& \qquad \qquad \textnormal{if} \qquad \theta>1. \label{deg_to_one_up_to_down_hat}
\end{align}
Finally, let us define 
 \begin{align}
N(r,s,n_1, n_2) &:= \langle a^{n_1} \lambda^{n_2}, ~[\mathcal{D}]\cap \mathcal{H}_L^{r} \cap \mathcal{H}_p^s \rangle.  
\end{align}
We now note that 
\begin{align}
\mathcal{H}_L  & = \lambda + d a \qquad \textnormal{and} \qquad \mathcal{H}_p = \lambda a. \label{HL_Hp_class}  
\end{align}
The reason why this is true is explained in \cite[Pages 18 and 19]{Zing_Notes}. 
Now, using the ring structure of $\mathcal{D}$ (as given by \cref{ring_str}), 
we can compute $N(r,s,n_1, n_2)$ 
by extracting the coefficient of $a^3 \lambda^{n-1}$ from  
\begin{align*}
(\lambda + d a)^r (\lambda a)^{s} a^{n_1} \lambda^{n_2}. 
\end{align*}
Hence, $N(r,s,n_1, n_2)$ can be computed for any $r, s,  n_1$ and $n_2$.  


\section{Recursive Formulas} 
\label{recursive_formulas}
We are now ready to state the recursive formulas. 
We have written a mathematica program to implement these recursive formulas and 
obtain the final formulas. The program is available on the second author's homepage 
\[ \textnormal{\url{https://www.sites.google.com/site/ritwik371/home}} \]
For the convenience of the reader, we have explicitly written down the formulas 
for $N(r,s,0,0)$ and $N(A_1^{\delta} \mathfrak{X}, r, s, 0,0)$ 
in \Cref{expfor}. 
Note that $N(r,s,0,0)$ is the number of planar degree-$d$ curves intersecting $r$ lines and passing through $s$ points.
Our formulas for $N(A_1^{\delta},r, s, 0, 0, 0)$ agree with those obtained by 
Kleiman and Piene in \cite{KP2} and by Ties  Laarakker in \cite{TL}. \\

\begin{thm}
\label{na1_again}
Consider the ring 
\begin{align*}
\mathcal{R}&= \frac{\mathbb{Z}[a, H, \lambda]}{\langle a^4, ~~H^4, ~~\lambda^n + s_1 a\lambda^{n-1} + s_2 a^2\lambda^{n-2} + s_3 a^3\lambda^{n-3}\rangle}, 
\end{align*}
where $s_1$, $s_2$, $s_3$ and $n$ are as defined in \cref{si} and \cref{n_dim_vec_sp}. Let 
\begin{align*}
e & := (\lambda + H) (\lambda + d a)^r (\lambda a)^s a^{n_1} \lambda^{n_2} H^{n_3}(\lambda + dH)\Big((\lambda + dH)^2 -(3H-a)(\lambda + dH) + a^2 -2aH + 3H^2\Big).
\end{align*}
Then $N(A_1, r, s, n_1, n_2, n_3)$ is the coefficient of $\lambda^{n-1} a^3 H^3$ in the polynomial $e$, seen as an element of the ring $\mathcal{R}$.
\end{thm}

\begin{rem}
\Cref{na1_again} is true for all $d \geq 1$.  
\end{rem}
Next, we will give a formula for $N(A_1^{\delta}A_1,r, s, n_1, n_2, n_3)$, when $1 \leq \delta \leq 3$. 
First let us make a couple of definitions.  
\begin{align}
\textnormal{Eul}(\delta, r, s, n_1, n_2, 0) & := (d-2d^2 + d^3) N(A_1^{\delta-1} A_1, r,s,n_1+1, n_2, 0)  \nonumber \\ 
                                            & +  (3-6d + 3d^2) N(A_1^{\delta-1} A_1, r,s,n_1, n_2+1, 0) \nonumber \\ 
\textnormal{Eul}(\delta, r, s, n_1, n_2, 1) & := (d^2-d)N(A_1^{\delta-1} A_1, r,s,n_1+2, n_2, 0) \nonumber \\ 
                                            & + (3d^2-4d+1)N(A_1^{\delta-1} A_1, r,s,n_1+1, n_2+1, 0) \nonumber \\ 
                                            & + (3d-3) N(A_1^{\delta-1} A_1, r,s,n_1, n_2+2, 0), \nonumber \\ 
\textnormal{Eul}(\delta, r, s, n_1, n_2, 2) & := d N(A_1^{\delta-1} A_1, r,s,n_1+3, n_2, 0) \nonumber \\ 
                                            & + (2d-1)N(A_1^{\delta-1} A_1, r,s,n_1+2, n_2+1, 0) \nonumber \\ 
                                            & + (3d-2)N(A_1^{\delta-1} A_1, r,s,n_1+1, n_2+2, 0) \nonumber \\
                                            & +  N(A_1^{\delta-1} A_1, r,s,n_1, n_2+3, 0) \nonumber \\ 
\textnormal{Eul}(\delta, r, s, n_1, n_2, 3) & := N(A_1^{\delta-1} A_1, r,s,n_1+3, n_2+1, 0)\nonumber \\ 
                                            & + N(A_1^{\delta-1} A_1, r,s,n_1+2, n_2+2, 0) \nonumber \\ 
                                            & + N(A_1^{\delta-1} A_1, r,s,n_1+1, n_2+3, 0) \nonumber \\
\textnormal{Eul}(\delta, r, s, n_1, n_2, n_3) &=  0 \qquad \textnormal{if} ~~n_3 >3. \label{eul_num_na1}
\end{align}
\noindent We also define 
\begin{align}
\textnormal{B}(\delta, r, s, n_1, n_2, n_3) &:= \binom{\delta}{1}B_1 + \binom{\delta}{2}B_2 + \binom{\delta}{3}B_3, \qquad \textnormal{where} \nonumber \\
B_1  &:= \Big(N(A_1^{\delta-1}A_1, r, s, n_1, n_2+1, n_3)  
+ dN(A_1^{\delta-1}A_1, r, s, n_1, n_2, n_3+1) \nonumber \\ 
& \qquad  \qquad + 3 N(A_1^{\delta-1}\PP A_2, r, s, n_1, n_2, n_3, 0) \Big) \nonumber \\ 
B_2 &:= 4 \Big( N(A_1^{\delta-2}\PP A_3, r, s, n_1, n_2, n_3, 0) \Big) \nonumber \\ 
B_3 &:= \frac{18}{3}\Big( N(A_1^{\delta-3}\PP D_4, r, s, n_1, n_2, n_3, 0) \Big). \label{bd_na1}
\end{align}

We are now ready to state the formula for $N(A_1^{\delta}A_1,r, s, n_1, n_2, n_3)$.

\begin{thm}
\label{na1_delta}
Let $\textnormal{Eul}(\delta, r, s, n_1, n_2, n_3)$ and $\textnormal{B}(\delta, r, s, n_1, n_2, n_3)$ 
be defined as in \cref{eul_num_na1} and \cref{bd_na1} respectively.  If $1 \leq \delta \leq 3$, then 
\begin{align*}
N(A_1^{\delta}A_1,r, s, n_1, n_2, n_3) & = \textnormal{Eul}(\delta, r, s, n_1, n_2, n_3)-\textnormal{B}(\delta, r, s, n_1, n_2, n_3),   
\end{align*}
provided $d\geq 2 \delta +1$.
\end{thm} 

We now state the remaining formulas. 


\begin{thm}
\label{npa1}
If $0 \leq \delta \leq 2$, then 
\begin{align*}
N(A_1^{\delta} \mathcal{P} A_1, r, s, n_1, n_2, n_3,  0) & = 2 N(A_1^{\delta} A_1, r, s, n_1, n_2, n_3), \\ 
N(A_1^{\delta} \mathcal{P} A_1, r, s, n_1, n_2, n_3,  1) & = N(A_1^{\delta} A_1, r, s, n_1, n_2+1, n_3)\\ 
                                                         & +(d-6)N(A_1^{\delta} A_1, r, s, n_1, n_2, n_3+1) \\ 
                                                         & + 2 N(A_1^{\delta} A_1, r, s, n_1+1, n_2, n_3) \\
                                                         &-2\binom{\delta}{2} N(A_1^{\delta-2} \PP D_4, r, s, n_1, n_2, n_3, 0),  
\end{align*}
provided $d \geq 2 \delta +2$. 
\end{thm}

\begin{rem}
To compute $N(A_1^{\delta} \mathcal{P} A_1, r, s, n_1, n_2, n_3,  \theta)$ when $\theta \geq 2$, we use \cref{lm_reduce}.
\end{rem}

\begin{thm}
\label{npa2}
Let $0 \leq \delta \leq 2$ and $\theta$ a non negative integer with the following property:  
if $\delta$ is either $0$ or $1$, then $\theta$ can be anything, but if $\delta=2$, then $\theta=0$.  
Then,
\begin{align*}
N(A_1^{\delta} \mathcal{P} A_2, r, s, n_1, n_2, n_3,  \theta) & = N(A_1^{\delta} \mathcal{P} A_1, r, s, n_1+1, n_2, n_3,  \theta) \\ 
                                                              &+ N(A_1^{\delta} \mathcal{P} A_1, r, s, n_1, n_2+1, n_3,  \theta) \\ 
                                                              &+ (d-3)N(A_1^{\delta} \mathcal{P} A_1, r, s, n_1, n_2, n_3+1,  \theta) \\ 
                                                              &-2\binom{\delta}{1} N(A_1^{\delta-1} \mathcal{P} A_3, r, s, n_1, n_2, n_3,  \theta) \\ 
                                                              &-3 \binom{\delta}{1} N(A_1^{\delta-1} \widehat{D}_4, r, s, n_1, n_2, n_3,  \theta) \\ 
                                                              &-4\binom{\delta}{2} N(A_1^{\delta-2} \mathcal{P} D_4, r, s, n_1, n_2, n_3,  \theta), 
\end{align*}
provided $d \geq 2 \delta +2$.
\end{thm}

\begin{rem}
If $\delta =2$ and $\theta>0$, then the formula given by Theorem \ref{npa2} is not valid; there is a further correction term 
(the interested reader can refer to \cite{BM8} 
to see what the extra correction term is). However, to compute $N(A_1^{2}A_2, r, s, n_1, n_2, n_3)$ 
we only need to know what is $N(A_1^{2} \mathcal{P} A_2, r, s, n_1, n_2, n_3, 0)$ and hence for the purposes of this paper, this 
Theorem is sufficient. We would require $N(A_1^{2} \mathcal{P} A_2, r, s, n_1, n_2, n_3, \theta)$ for $\theta>0$ if we were computing 
any of the codimension five (or higher) numbers; in this paper we are computing numbers till codimension four. 
\end{rem}

\begin{thm}
\label{npa3}
If $0 \leq \delta \leq 1$, then 
\begin{align*}
N(A_1^{\delta} \mathcal{P} A_3, r, s, n_1, n_2, n_3,  \theta) &= N(A_1^{\delta} \mathcal{P} A_2, r, s, n_1, n_2+1, n_3,  \theta) \\ 
                                                              &+ 3 N(A_1^{\delta} \mathcal{P} A_2, r, s, n_1, n_2, n_3,  \theta+1) \\ 
                                                              &+ d N(A_1^{\delta} \mathcal{P} A_2, r, s, n_1, n_2, n_3+1,  \theta) \\ 
                                                              &-2\binom{\delta}{1}N(A_1^{\delta-1} \mathcal{P} A_4, r, s, n_1, n_2, n_3,  \theta),
\end{align*}
provided $d \geq 2 \delta +3$.
\end{thm}

\begin{thm}
\label{npa4}
If $d \geq 4$,  then 
\begin{align*}
N(\mathcal{P} A_4, r, s, n_1, n_2, n_3,  \theta) &= 2N(\mathcal{P} A_3, r, s, n_1, n_2+1, n_3,  \theta) \\ 
                                                 &+2N(\mathcal{P} A_3, r, s, n_1, n_2, n_3,  \theta+1) \\ 
                                                 &+2N(\mathcal{P} A_3, r, s, n_1+1, n_2, n_3,  \theta)
                                                  &+(2d-6)N(\mathcal{P} A_3, r, s, n_1, n_2, n_3+1,  \theta)
\end{align*}
\end{thm}

\begin{thm}
\label{npd4}
If $d \geq 3$,  then 
\begin{align*}
N(\mathcal{P} D_4, r, s, n_1, n_2, n_3,  \theta) &= N(\mathcal{P} A_3, r, s, n_1, n_2+1, n_3,  \theta) \\ 
                                                 &-2N(\mathcal{P} A_3, r, s, n_1, n_2, n_3,  \theta+1) \\ 
                                                 &+2N(\mathcal{P} A_3, r, s, n_1+1, n_2, n_3,  \theta) 
                                                 &+(d-6)N(\mathcal{P} A_3, r, s, n_1, n_2, n_3+1,  \theta)
\end{align*}
\end{thm}


We will  now prove these recursive formulas.

\section{Proof of the recursive formulas}
We are now ready to prove the formulas stated in section \ref{recursive_formulas}. 
We will use a topological method to compute the degenerate contribution to the Euler class. 
Our method is an extension of the method that originates in the paper by A.~Zinger (\cite{Zin}) 
and which is further pursued by S.~Basu and the second author in \cite{R.M}, \cite{BM13_2pt_published} 
and \cite{BM8}. \\ 
\hf \hf When there is no cause for confusion,  
we will sometimes abbreviate 
$N(A_1^{\delta}A_1,r, s, n_1, n_2, n_3)$ and 
\newline $N(A_1^{\delta}\mathcal{P}\mathfrak{X},r, s, n_1, n_2, n_3, \theta)$ 
as $N(A_1^{\delta}A_1)$ and 
$N(A_1^{\delta} \mathcal{P}\mathfrak{X})$ 
(for the sake of notational simplicity).

\subsection{Proof of \Cref{na1_again} and \ref{na1_delta}: computation of $N(A_1^{\delta}A_1)$ when $0\leq \delta \leq 3$}
\label{na1_delta_proof}
\verb+ +\\
\hf  \hf We will justify our formula for 
$N(A_1^{\delta} A_1,r, s, n_1, n_2, n_3)$, when $0 \leq \delta \leq 3$.
Recall that in \Cref{notation}, we have defined  
\begin{align*}
\A^{\delta}_1 \circ \mathcal{S}_{\DD} := \{ ([f], [\eta];  q_1, \ldots, q_{\delta}, q_{\delta+1}) & \in \DD \times (\mathbb{P}^3)^{\delta +1}:  
\eta (q_i) = 0, \qquad \forall i = 1 ~~\textnormal{to} ~~ \delta+1, \\
&\textnormal{$f$ has a singularity of type $\A_1$ at $q_1, \ldots, q_{\delta}$}, \\ 
                   & \textnormal{$q_1, \ldots, q_{\delta+1}$ all distinct}\}. 
\end{align*}
Let $\mu$ be a generic cycle, representing 
the class 
\begin{align*}
[\mu] = \mathcal{H}_L^r \cdot \mathcal{H}_p^s \cdot a^{n_1} \lambda^{n_2} (\pi_{\delta+1}^*H)^{n_3}. 
\end{align*}
Here $\pi_i$ denotes the projection onto the $i^{\textnormal{th}}$ point. 
We will often omit writing down $\pi_{\delta+1}^*$, if there is no cause for confusion. 
We now consider sections of the following two bundles 
that are induced by the evaluation map and the vertical derivative at the last point, namely: 
\begin{align*}
\Psi_{\A_0}: A_1^{\delta} \circ \mathcal{S}_{\DD} \lra \mathcal{L}_{\A_0} & := 
\gamma_{\DD}^* \otimes \pi_{\delta+1}^*\gamma_{\mathbb{P}^3}^{* d},
\qquad \qquad 
\{\Psi_{\A_0}([f], [\eta], q_1, \ldots, q_{\delta+1})\}(f) := f(q_{\delta+1}) \qquad  \textnormal{and} \\
\Psi_{\A_1} : \psi_{\A_0}^{-1}(0) \lra \mathcal{V}_{\A_1} &:= \gamma_{\D}^*\otimes \pi_{\delta+1}^*W^* \otimes \pi_{\delta+1}^*\gamma_{\mathbb{P}^3}^{* d}, 
\qquad 
\{\Psi_{\A_1}([f], [\eta], q_1, \ldots, q_{\delta+1})\}(f) := \nabla f|_{q_{\delta+1}}. 
\end{align*}
We will  show shortly that $\psi_{\A_0}$ and $\psi_{A_1}$ 
are transverse to zero, provided $d \geq 2 \delta +1$. \\
\hf \hf Next, let us define 
\begin{align*}
\mathcal{B} &:= \overline{\A^{\delta}_1 \circ \mathcal{S}}_{\DD}- \A^{\delta}_1 \circ \mathcal{S}_{\DD}. 
\end{align*}
Hence 
\begin{align}
\lan e(\mathcal{L}_{\A_0}) e(\mathcal{V}_{\A_1}), 
~~[\overline{\A^{\delta}_1 \circ \mathcal{S}_{\DD}}] \cap [\mu] \ran & = \N(\A_1^{\delta}\A_1, r, s, n_1, n_2, n_3)  
+ \mathcal{C}_{\mathcal{B} \cap \mu},\label{euler_na1_delta}
\end{align}
where $e$ denote the Euler class and  
$\mathcal{C}_{\mathcal{B} \cap \mu}$ denotes the contribution of the section to the Euler 
class from the points of $\mathcal{B} \cap \mu$.\\
\hf \hf Let us first explain how to compute the left hand side of  \cref{euler_na1_delta}
(i.e. the Euler class). 
From equations \eqref{HL_Hp_class} and \eqref{s_ev}, we note that 
\begin{align*}
\mathcal{H}_L & = \lambda + d a, \qquad \mathcal{H}_p = \lambda a \qquad \textnormal{and} \qquad [\pi_{\delta+1}^*\mathcal{S}_{\DD}] = \lambda + \pi_{\delta+1}^*H. 
\end{align*}
Next, we need to compute the Chern classes of $W$. 
We note that over $\mathcal{S}$, we have the following short exact sequence of vector bundles:  
\begin{align*}
0&\longrightarrow W \longrightarrow T \mathbb{P}^3 \longrightarrow \gamma^{*}_{\widehat{\mathbb{P}}^3} \otimes \gamma^{*}_{\mathbb{P}^3} \longrightarrow 0. 
\end{align*}
Here the first map is the inclusion map and the second map is $\nabla \eta|_q$.
Hence, 
\begin{align}
c(W) c(\gamma^{*}_{\widehat{\mathbb{P}}^3} \otimes \gamma^{*}_{\mathbb{P}^3}) & = c(T\mathbb{P}^3) ~~
\implies c_1(W) = 3H-a \qquad \textnormal{and} \qquad c_2(W) = a^2-2aH + 3H^2. \label{ciW}
\end{align}
Next, using the splitting  principle, we conclude that 
\begin{align}
e(\gamma_{\mathcal{D}}^*\otimes \gamma_{\mathbb{P}^3}^{* d})e(\gamma_{\mathcal{D}}^*\otimes W^*\otimes \gamma_{\mathbb{P}^3}^{* d}) &= 
(\lambda + dH)((\lambda + dH)^2 -c_1(W)(\lambda + dH) + c_2(W)). \label{Eul_ring}
\end{align}
Note that we have made an abuse of notation by omitting to write down $\pi_{\delta+1}^*$; henceforth we will make this abuse of notation. 
Now, suppose $\delta =0$. 
Then, using the ring structure of $\mathcal{D}$ 
(as given by \cref{ring_str})
and by extracting the coefficient of $\lambda^{n-1} a^3 H^3$ from
\begin{align*}
(\lambda + H)(\lambda + dH)((\lambda + dH)^2 -c_1(W)(\lambda + dH) + c_2(W))(\lambda+d a)^r (\lambda a)^s a^{n_1} \lambda^{n_2} H^{n_3},
\end{align*}
we obtain the Euler  class. When $\delta=0$, using \cref{ciW}, we get 
 the formula of Theorem \ref{na1_again}. When 
$\delta>0$, we get $\textnormal{Eul}(\delta, r, s, n_1, n_2, n_3)$ as defined in \cref{eul_num_na1}.\\
\hf \hf Let us now explain how to compute $\mathcal{C}_{\mathcal{B} \cap \mu}$, the degenerate contribution to the Euler class. 
When $\delta=0$, the boundary $\mathcal{B}$ is empty and hence we get the result of Theorem \ref{na1_again}. Let us now consider the 
case when $\delta \geq 1$. Given $k$ distinct integers $i_1, i_2, \ldots, i_k \in [1, \delta+1]$, let us define 
\begin{align*}
\Delta_{i_1, \ldots, i_{k}} & := \{ ([f], [\eta];  q_1, \ldots, q_{\delta}, q_{\delta+1})  \in \mathcal{S}_{\mathcal{D}_{\delta+1}}: 
~~q_{i_1}= q_{i_2}=\ldots = q_{i_k}\} \qquad \textnormal{and} \\
\mathcal{B}(q_{i_1}, \ldots, q_{i_k}) & := \mathcal{B}\cap \Delta_{i_1, \ldots, i_k}.
\end{align*}
Let us now consider $\mathcal{B}(q_i, q_{\delta+1})$. We claim that 
\begin{align}
\mathcal{B}(q_i, q_{\delta+1}) &\approx \overline{A_1^{\delta-1}\circ A}_1 \qquad \forall ~~i=1 ~\textnormal{to} ~\delta, \label{a1_a0_bdry_cmp}
\end{align}
where $\mathcal{B}(q_i, q_{\delta+1})$ is identified as a subset of $\mathcal{S}_{\mathcal{D}_{\delta}}$ in the obvious way 
(namely via the inclusion of $\mathcal{S}_{\mathcal{D}_{\delta}}$ inside 
$\mathcal{S}_{\mathcal{D}_{\delta+1}}$ where the $(\delta+1)^{\textnormal{th}}$ point 
is equal to the $i^{\textnormal{th}}$ point). Next, we claim that 
the contribution from $\mathcal{B}(q_i, q_{\delta+1})\cap \mu$ is given by 
\begin{align}
\langle e(\mathcal{L}_{A_0}), ~[\overline{A_1^{\delta-1}\circ A}_1]\cap [\mu] \rangle  + 3N(A_1^{\delta-1}\mathcal{P}A_2, r,s,n_1, n_2, n_3, 0).  
\label{Eul_deg_node_bdry}
\end{align}
We will explain the reason for both the claims shortly. 
The expression given by \cref{Eul_deg_node_bdry} 
is precisely equal to $\mathrm{B}_1$ as defined in \cref{bd_na1}. Hence, the sum total of the contribution from 
$\mathcal{B}(q_i, q_{\delta+1})$ for $i=1$ to $\delta$ is $\binom{\delta}{1} \mathrm{B}_1$.\\ 
\hf  \hf Next, let us assume $\delta \geq 2$ and consider $\mathcal{B}(q_{i_1}, q_{i_2}, q_{\delta+1})$. We claim that 
\begin{align}
\mathcal{B}(q_{i_1}, q_{i_2}, q_{\delta+1}) &\approx \overline{A_1^{\delta-2}\circ A}_3 \label{a1_a1_is_a3}
\end{align}
for all distinct pairs $(i_1, i_2)$. We also claim that the contribution from each of the points of 
\[\mathcal{B}(q_{i_1}, q_{i_2}, q_{\delta+1})\cap \mu\] 
is $4$. We will justify both these claims  shortly.  
Hence the sum total of the contribution as we vary over all $(i_1, i_2)$ 
is  precisely $\binom{\delta}{2}\mathrm{B_2}$, where $\mathrm{B}_2$ is as defined in \cref{bd_na1}.\\ 
\hf \hf Finally, let us assume $\delta \geq 3$ and consider $\mathcal{B}(q_{i_1}, q_{i_2}, q_{i_3}, q_{\delta+1})$. 
We claim that 
\begin{align}
\mathcal{B}(q_{i_1}, q_{i_2}, q_{i_3}, q_{\delta+1}) &\approx \overline{A_1^{\delta-3}\circ A}_5 \cup 
\overline{A_1^{\delta-3}\circ D}_4\label{three_node_triple_point}
\end{align}
for all distinct triples $(i_1, i_2, i_3)$. Note that $\overline{A_1^{\delta-3}\circ A}_5 \cap \mu$ is empty, since 
the sum of their dimensions is one less than the dimension of the ambient space where we are intersecting them. Hence, 
we get no contribution from $\overline{A_1^{\delta-3}\circ A}_5 \cap \mu$. Finally, 
we claim that the contribution from each of the points of 
$\overline{A_1^{\delta-3}\circ D}_4 \cap \mu$ is $18$.
Hence the sum total of the contribution as we vary over all $(i_1, i_2, i_3)$ 
is  precisely $\binom{\delta}{3}\mathrm{B_3}$, where $\mathrm{B}_3$ is as defined in \cref{bd_na1}. \\ 
\hf \hf Let us now prove the claims regarding transversality and degenerate contributions to the Euler  class.  
We will start by proving transversality. 
Note that we need to prove $A_1^{\delta+1}$ is a smooth complex submanifold 
of $\mathcal{S}_{\mathcal{D}_{\delta+1}}$ (provided $d\geq 2\delta+1$). 
We will prove a stronger statement: we will show 
that 
$\overline{A}_1^{\delta+1}$ is a smooth complex submanifold 
of $\mathcal{S}_{\mathcal{D}_{\delta+1}}$ and the sections $\Psi_{A_0}$ and 
$\Psi_{A_1}$, defined on   
$\overline{A}_1^{\delta}\circ \mathcal{S}_{\delta}$ and $\Psi_{A_0}^{-1}(0)$ 
respectively, are transverse to zero. Our desired claim follows immediately from this statement since  
$A_1^{\delta+1}$ is an open subset of $\overline{A}_1^{\delta+1}$.\\
\hf \hf Let us begin by showing that $\Psi_{A_0}$ is transverse to zero if $d \geq 2\delta +1$. 
Suppose 
\begin{align*}
\Psi_{A_0}([f], [\eta], q_1, \ldots, q_{\delta+1})&= 0. 
\end{align*}
Without loss of generality, let us assume that $[\eta]$ determines the plane where the last coordinate is zero, 
and $q_{\delta+1}$ is the point where only the third coordinate is nonzero and the rest are zero, i.e. 
\begin{align*}
\mathbb{P}^2_{\eta} & \approx \{[X,Y,Z, W] \in \mathbb{P}^3: W=0\} \qquad \textnormal{and} \qquad q_{\delta+1}:= [0,0,1,0].
\end{align*}
Assume that the remaining points 
are given by 
\begin{align*}
q_i&:= [X_i, Y_i, Z_i, 0] \qquad \textnormal{for}~~i=1~~\textnormal{to}~~\delta. 
\end{align*}
For simplicity, we can assume that all $Z_i$ are nonzero. Furthermore, since all the $q_i$ are distinct, we conclude that 
$X_i$ and $Y_i$ can not both be zero; for simplicity let us assume $X_i$ is nonzero for each $i$ (from $1$ to $\delta$). 
Consider the homogeneous degree $d$ polynomial, given by  
\begin{align*}
\rho_{00}&:= (X-X_1)^2(X-X_2)^2\ldots \cdot (X-X_{\delta})^2 Z^{d-2\delta}. 
\end{align*}
We note the following facts about $\rho_{00}$: 
\begin{align}
\rho_{00}(q_i) &= 0 \qquad \forall ~~i=1~~\textnormal{to}~~\delta, \label{rho00}\\
\nabla \rho_{00}|_{q_i} &= 0 \qquad \forall ~~i=1~~\textnormal{to}~~\delta \qquad \textnormal{and} \label{nabla_rho_00}\\ 
\rho_{00}(q_{\delta+1})& \neq 0. \label{rho00_last_pt}
\end{align}
Now consider the curve $\gamma:(-\varepsilon, \varepsilon) \longrightarrow \mathcal{S}_{\mathcal{D}_{\delta+1}}$, given by 
\begin{align*}
\gamma(t)&:= ([f+t \rho_{00}], [\eta], q_1, \ldots, q_{\delta+1}). 
\end{align*}
Because of \cref{rho00} and \cref{nabla_rho_00}, 
we conclude that this curve lies in $\overline{A}_1^{\delta}\circ \mathcal{S}_{\mathcal{D}}$. We now  
note that 
\begin{align}
\{\{\nabla \Psi_{A_0}|_{([f], [\eta], q_1, \ldots, q_{\delta+1})}\}(\gamma^{\prime}(0))\}(f)& = \rho_{00}(q_{\delta+1}). 
\label{nabla_Psi_A0}
\end{align}
Using \cref{rho00_last_pt}, we conclude that the right hand side of \cref{nabla_Psi_A0} is nonzero, whence 
$\Psi_{A_0}$ is transverse to zero. Next, let us prove transversality for the section $\Psi_{A_1}$. Consider the polynomials, 
\begin{align*}
\rho_{10}&:= (X-X_1)^2(X-X_2)^2\ldots \cdot (X-X_{\delta})^2X Z^{d-2\delta-1} \qquad \textnormal{and} \\ 
\rho_{01}&:= (X-X_1)^2(X-X_2)^2\ldots \cdot (X-X_{\delta})^2Y Z^{d-2\delta-1}.
\end{align*}
We note that $\rho_{10}$ and $\rho_{01}$ satisfy \cref{rho00} and \cref{nabla_rho_00} 
(with $\rho_{00}$ replaced with $\rho_{10}$ and $\rho_{01}$ respectively). Furthermore, 
\begin{align}
\rho_{10}(q_{\delta+1})&=0 \qquad \textnormal{and} \qquad  \rho_{01}(q_{\delta+1})= 0. \label{rho_10_and_01}
\end{align}
Construct the curves 
\begin{align*}
\gamma_{10}(t)&:= ([f+t \rho_{10}], [\eta], q_1, \ldots, q_{\delta+1}) \qquad \textnormal{and} 
\qquad \gamma_{01}(t):= ([f+t \rho_{01}], [\eta], q_1, \ldots, q_{\delta+1}).
\end{align*}
Because of \cref{rho00} and \cref{nabla_rho_00} 
(with $\rho_{00}$ replaced with $\rho_{10}$ and $\rho_{01}$ respectively) and \cref{rho_10_and_01}, 
these curves lie inside $\Psi_{A_0}^{-1}(0)$.
We now note that 
\begin{align*}
\{\{\nabla \Psi_{A_1}|_{([f], [\eta], q_1, \ldots, q_{\delta+1})}\}(\gamma_{10}^{\prime}(0))\}(f)& = \lambda Z^{d-2\delta-1} \nabla X|_{[0,0,1,0]} \qquad \textnormal{and} \\ 
\{\{\nabla \Psi_{A_1}|_{([f], [\eta], q_1, \ldots, q_{\delta+1})}\}(\gamma_{01}^{\prime}(0))\}(f)& = \lambda Z^{d-2\delta-1} \nabla Y|_{[0,0,1,0]},
\qquad \textnormal{where} \qquad \lambda := (-X_1)^2 \ldots (-X_{\delta})^2.
\end{align*}
Since $\nabla X|_{[0,0,1,0]}$ and $\nabla Y|_{[0,0,1,0]}$ are two linearly independent vectors of $T\mathbb{P}^2_{\eta}|_{[0,0,1,0]}$, 
we conclude that $\Psi_{A_1}$ is transverse to zero. \\
\hf \hf Let us now justify the closure and multiplicity claims. We will start by giving the reason for \cref{a1_a0_bdry_cmp} 
and \cref{Eul_deg_node_bdry}. This follows from the argument given in the proof 
\cite[Lemma 6.3 (1), Page 685]{BM13_2pt_published} and \cite[Corollary 6.6,  Page 689]{BM13_2pt_published}. 
The proof is the same. \\ 
\hf \hf Next, let us justify \cref{a1_a1_is_a3}. Without loss of generality, it suffices to justify 
it when $i_1:= \delta-1$ and $i_2:= \delta$. Hence, we need to show that 
\begin{align}
\{([f], [\eta], q_1, \ldots,q_{\delta+1}) \in \overline{A_1^{\delta}\circ \mathcal{S}}_{\mathcal{D}}: 
q_{\delta-1} = q_{\delta} = q_{\delta+1}\}&= \overline{A_1^{\delta-2}\circ A}_3. \label{a1_a1_a3_exp}
\end{align}
Before proceeding further, let us make a simple observation. Notice that the left hand side of  
\cref{a1_a1_a3_exp} is the same as 
\begin{align}
\{([f], [\eta], q_1, \ldots,q_{\delta+1}) \in \overline{A_1^{\delta}}: 
q_{\delta-1} = q_{\delta}\}. \label{temp_nat}
\end{align}
Hence, an equivalent way of stating \cref{a1_a1_a3_exp} is 
\begin{align}
\{([f], [\eta], q_1, \ldots,q_{\delta}) \in \overline{A_1^{\delta}}: 
q_{\delta-1} = q_{\delta}\} & = \overline{A_1^{\delta-2}\circ A}_3. \label{a1_a1_a3_exp_ag2}
\end{align}
Following \cite[Equation 6.4, Page 685]{BM13_2pt_published}, we conclude that 
\begin{align}
\Big(\{([f], [\eta], q_1, \ldots, q_{\delta-1}, q_{\delta}, q_{\delta+1}) \in \overline{A_1^{\delta}\circ \mathcal{S}}_{\mathcal{D}}: 
q_{\delta-1} = q_{\delta} = q_{\delta+1}\}\Big)\cap \Big(A_1^{\delta-2}\circ A_2\Big) &= \emptyset. \label{two_nodes_No_cusp}
\end{align}
Equation \cref{two_nodes_No_cusp} is saying that if two nodes come together, then the singularity has to be more degenerate than a cusp. 
Hence, the singularity has to be at least as degenerate as a tacnode 
(since $\overline{A}_2 = A_2 \cup \overline{A}_3$). 
Hence, the left hand side of \cref{a1_a1_a3_exp} is a subset of its right hand side. We will now prove the converse. 
We will simultaneously prove the following four statements: 
\begin{align}
\{([f], [\eta], q_1, \ldots, q_{\delta+1}) &\in \overline{A_1^{\delta}\circ \mathcal{S}}_{\mathcal{D}}: 
q_{\delta-1} = q_{\delta} = q_{\delta+1}\} \supset A_1^{\delta-2}\circ A_3,\label{a1_a1_a3_subset}\\ 
\Big(\{([f], [\eta], q_1, \ldots, q_{\delta+1}) & \in \overline{A_1^{\delta}\circ A}_1: 
q_{\delta-1} = q_{\delta}= q_{\delta+1}\}\Big)\cap \Big(A_1^{\delta-2}\circ A_3\Big) = \emptyset, \label{a1_a1_a1_cap_a3_is_empty} \\ 
\Big(\{([f], [\eta], q_1, \ldots, q_{\delta+1}) & \in \overline{A_1^{\delta}\circ A}_1: 
q_{\delta-1} = q_{\delta}= q_{\delta+1}\}\Big)\cap \Big(A_1^{\delta-2}\circ A_4\Big) = \emptyset \qquad \textnormal{and}\label{three_nodes_canot_be_a4} \\ 
\{([f], [\eta], q_1, \ldots, q_{\delta+1}) &\in \overline{A_1^{\delta}\circ A}_1: 
q_{\delta-1} = q_{\delta} = q_{\delta+1}\} \supset A_1^{\delta-2}\circ A_5. \label{three_nodes_A5}
\end{align}
Since $\overline{A_1^{\delta}\circ \mathcal{S}}_{\mathcal{D}}$ is a closed set, 
\cref{a1_a1_a3_subset} implies that the right hand side of \cref{a1_a1_a3_exp} 
is a subset of its left hand side. Before we prove the above four statements, let us explain intuitively the 
significance of each of the statements.\\ 
\hf \hf The first statement, \cref{a1_a1_a3_subset} is saying that 
every tacnode is in the closure of two nodes (we remind the reader that the left hand side of \cref{a1_a1_a3_subset} 
is same as the expression 
given by \cref{temp_nat}). 
Geometrically, figure \ref{two_nodes_tacnode_pic} explains the meaning  of \cref{a1_a1_a3_subset}.  
\begin{figure}[h!]
\vspace*{0.2cm}
\begin{center}\includegraphics[scale = 0.5]{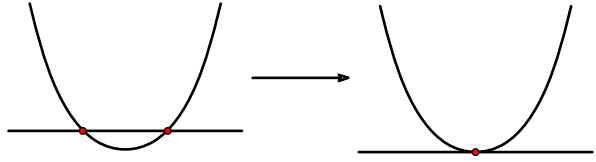}\vspace*{-0.2cm}\end{center}
\caption{Two nodes colliding into a tacnode}
\label{two_nodes_tacnode_pic}
\end{figure}
\newline
\hf\hf The second statement, \cref{a1_a1_a1_cap_a3_is_empty} is saying that in the closure of three nodes, we get a singularity 
more degenerate than a tacnode. The third statement, \cref{three_nodes_canot_be_a4} is saying that in the closure of
three nodes, we get a singularity 
more degenerate than an $A_4$ singularity. Finally, \cref{three_nodes_A5} is saying that every $A_5$ singularity 
is in the closure of three nodes. Geometrically, figure \ref{fig_3_nodes_a5_ag} explains the meaning  of \cref{three_nodes_A5} 
\begin{figure}[h!]
\vspace*{0.2cm}
\begin{center}\includegraphics[scale = 0.5]{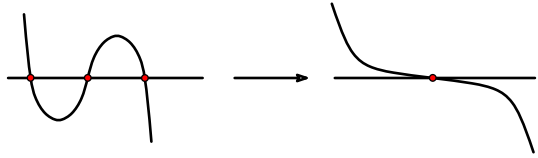}\vspace*{-0.2cm}\end{center}
\caption{Three nodes colliding into an $\A_5$-singularity}
\label{fig_3_nodes_a5_ag}
\end{figure}
\newline 
\hf  \hf We are now ready to prove the above statements. Let us prove the following two claims:
\begin{claim}
\label{two_and_three_node_claim}
Let $([f], [\eta], q_1, \ldots, \ldots, q_{\delta}) \in A_1^{\delta-2}\circ A_3$. Then there exists points 
\begin{align*}
\big([f_t], [\eta_t], q_1(t), \ldots, q_{\delta-2}(t); q_{\delta-1}(t), q_{\delta}(t), q_{\delta+1}(t)\big) \in 
A_1^{\delta-2}\circ \mathcal{S}_{\mathcal{D}}^{3} 
\end{align*}
sufficiently close to $([f], [\eta], q_1, \ldots, \ldots, q_{\delta-2}; q_{\delta}, q_{\delta}, q_{\delta})$, such that 
\begin{align}
f_t(q_{i}(t)) &=0, ~~~\nabla f_t|_{q_{i}(t)} = 0 \qquad \textnormal{for} ~~~i=\delta-1 ~~~\textnormal{and} ~~~\delta.  
\qquad 
\label{two_nodes_eqn}
\end{align}
Furthermore, \textit{every} such solution 
satisfies the condition 
\begin{align}
\Big(f_t(q_{\delta+1}(t)) , ~\nabla f_t|_{q_{\delta+1}(t)}\Big) &\neq (0, 0), \label{three_node_int_tacnode_empty}
\end{align}
i.e. 
$\big([f_t], [\eta_t], q_1(t), \ldots, q_{\delta-2}(t); q_{\delta-1}(t), q_{\delta}(t), q_{\delta+1}(t)\big) \not\in A_1^{\delta} \circ A_1$.
In fact, if 
\[([f], [\eta], q_1, \ldots, \ldots, q_{\delta}) \in A_1^{\delta-2}\circ A_4,\] 
then there does not exist any point 
\begin{align*}
\big([f_t], [\eta_t], q_1(t), \ldots, q_{\delta-2}(t); q_{\delta-1}(t), q_{\delta}(t), q_{\delta+1}(t)\big) \in 
A_1^{\delta-2}\circ \mathcal{S}_{\mathcal{D}}^{3} 
\end{align*}
sufficiently close to $([f], [\eta], q_1, \ldots, \ldots, q_{\delta-2}; q_{\delta}, q_{\delta}, q_{\delta})$, such that 
\begin{align}
f_t(q_{i}(t)) &=0, ~~~\nabla f_t|_{q_{i}(t)} = 0 \qquad \textnormal{for} ~~~i=\delta-1, ~\delta ~~\textnormal{and} ~~\delta+1.  
\qquad 
\end{align}
\end{claim}

\begin{claim}
\label{three_nodes_A5_claim}
Let $([f], [\eta], q_1, \ldots, \ldots, q_{\delta}) \in A_1^{\delta-2}\circ A_5$. Then there exists points 
\begin{align*}
\big([f_t], [\eta_t], q_1(t), \ldots, q_{\delta-2}(t); q_{\delta-1}(t), q_{\delta}(t), q_{\delta+1}(t)\big) \in 
A_1^{\delta-2}\circ A_1^{3} 
\end{align*}
sufficiently close to $([f], [\eta], q_1, \ldots, \ldots, q_{\delta-2}; q_{\delta}, q_{\delta}, q_{\delta})$. 
\end{claim}

\begin{rem}
We note claim \ref{two_and_three_node_claim} proves \cref{a1_a1_a3_subset}, \cref{a1_a1_a1_cap_a3_is_empty} 
and \cref{three_nodes_canot_be_a4} simultaneously. We also note that claim \ref{three_nodes_A5_claim} proves \cref{three_nodes_A5}.
\end{rem}

\begin{rem}
\label{not_big_O}
Before proceeding with the proof, let us make a shorthand notation. We denote 
\begin{align*}
O(|(x_1, x_2, \ldots, x_n)|^k)  
\end{align*}
to be a \textbf{holomorphic} function (in the variables $x_1, \ldots, x_n$), 
defined in a neighborhood of the origin in $\mathbb{C}^n$,
whose order of vanishing is at least $k$ 
(i.e. all the terms of degree lower than $k$ are absent in the Taylor expansion of the function around the origin). 
We say that such an expressions is of order $k$. 
For example, $x_1^4 + x_1 x_2 x_3^2 + x_2^2 x_3^3$ is a term of order $4$ and we will denote it by $O(|(x_1, x_2, x_3)|^4)$. 
Note that we are always dealing with holomorphic functions. Hence, suppose a function (in say one variable) is of type 
$O(|x|^2)$, it means, its Taylor expansion is of the type 
\[ f(x) = a_2 x^2 + a_3 x^3 + \ldots. \] 
It does not mean that there are terms of type $x \overline{x}$ (although the $|x|^2$ in the $O(|x|^2)$ might suggest that).  
Henceforth, it will be understood that $O(|x|^n)$ and $O(x^n)$ mean the same thing in our paper (the latter is the standard 
notation in one variable).  
\end{rem}

\noindent \textbf{Proof of claims \ref{two_and_three_node_claim} and \ref{three_nodes_A5_claim}:} 
Let us define 
\begin{align*}
\mathbb{C}^2_{z}&:= \{(x,y,z) \in \mathbb{C}^3: z=0\}. 
\end{align*}
We will now work in an affine chart where we send the plane 
$\mathbb{P}^2_{\eta_t}$ to $\mathbb{C}^2_{z}$ and the 
point $q_{\delta}(t) \in \mathbb{P}^2_{\eta_t}$ to 
$(0,0,0) \in \mathbb{C}^2_{z}$. 
Using this  chart, 
let us write down the Taylor expansion of $f_t$ around the point $(0,0)$, namely  
\begin{align*}
f_t(x,y)&= \frac{f_{t_{20}}}{2} x^2 + f_{t_{11}}xy + \frac{f_{t_{02}}}{2} y^2 + \ldots  
\end{align*}
Note that since \cref{two_nodes_eqn} holds (for $i=\delta$), we conclude that $f_{t_{00}}, f_{t_{10}}$ and $f_{t_{01}}$ are zero.\\
\hf \hf Next, since $([f], [\eta], q_{\delta}) \in A_3$, we conclude that 
$f_{t_{20}}$ and $f_{t_{02}}$ can not both be zero; let us assume $f_{t_{02}} \neq 0$. Hence, 
$f_t(x,y)$ can be re-written as 
\begin{align*}
f_t(x,y)&= A_0(x) + A_1(x)y + A_2(x) y^2 + \ldots \qquad \textnormal{where} \qquad A_2(0) \neq 0.  
\end{align*}
We will now make a change of coordinates; let us define 
\begin{align*}
\hat{y}&:= y- B(x) 
\end{align*}
where $B(x)$ is a function that is to be determined. We claim that
there exists a unique holomorphic $B(x)$ (vanishing at the origin) such that after 
plugging it in $f_t(x, y)$ we get 
\begin{align*}
f_t (x, y(x, \hat{y})) &= \hat{A}_0(x) + \hat{A}_2 (x) \hat{y}^2 + \hat{A}_3 (x) \hat{y}^3 + \ldots   
\end{align*}
In other words, we want $\hat{A}_1(x) \equiv 0$. This is possible if $B(x)$ satisfies the equation
\begin{align}
A_1(x) + 2A_2(x) B(x) + 3A_3(x) B(x)^2 + \ldots & = 0. \label{psconvgg2} 
\end{align}
Since $A_2(0) = \frac{f_{t_{02}}}{2}  \neq 0$, $B(x)$ exists by the Implicit Function Theorem and 
we can compute $B(x)$ explicitly  
as a power series using \eqref{psconvgg2} and then
compute $\hat{A}_0(x)$. Hence, 
\begin{align*}
f_t (x, y(x, \hat{y})) &=  
\varphi(x,\hat{y})\hat{y}^2 + \frac{\mathcal{B}_2^{f_t}}{2!} x^2 + 
\frac{\mathcal{B}_3^{f_t}}{3!} x^3 + \frac{\mathcal{B}_4^{f_t}}{4!} x^4 + \mathcal{R}(x)x^5,  
\end{align*}
where 
\begin{align*}
\B^{f_t}_2 & :=  f_{t_{20}} - \frac{f_{t_{11}}^2}{f_{t_{02}}}, \qquad \B^{f_t}_3 := \frac{f_{t_{30}}}{6}
+\frac{f_{t_{11}}^2 f_{t_{12}}}{f_{t_{02}}^2}, \ldots \ldots, \qquad   
\varphi(0,0) \neq 0
\end{align*}
and $\mathcal{R}(x)$ is a holomorphic function defined in a neighborhood of the origin. 
Since $([f], [\eta], q_{\delta}) \in A_3$, 
we conclude that $\B^{f_t}_2$ and $\B^{f_t}_3$ are small (close to zero) and $\B^{f_t}_4$ is nonzero.  
Let us make a further change of coordinates and denote 
\begin{align*}
\hat{\hat{y}}&:= \sqrt{\varphi(x, \hat{y})} \hat{y}.
\end{align*}
Note that we can choose a branch of the square root since $\varphi(0,0) \neq 0$. 
Next, for notational convenience, let us now define 
\begin{align}
\hat{f}_t(x, \hat{\hat{y}})  &:= f_t (x, y(x, \hat{y}(\hat{\hat{y}})))), \label{hat_hat_f_defn}
\end{align}
i.e. $\hat{f}_t$ is basically $f_t$ written in the new coordinates (namely $x$ and $\hat{\hat{y}}$). Hence, 
\begin{align*}
\hat{f}_t(x, \hat{\hat{y}}) & = \hat{\hat{y}}^2 + \frac{\mathcal{B}_2^{f_t}}{2!} x^2 + 
\frac{\mathcal{B}_3^{f_t}}{3!} x^3 + \frac{\mathcal{B}_4^{f_t}}{4!} x^4 + \mathcal{R}(x)x^5.
\end{align*}
We will now solve \cref{two_nodes_eqn} for $i=\delta-1$. We note that this is amounts to solving for 
the set of equations 
\begin{align}
\hat{f}_t(u, v) = 0, \qquad \hat{f}_{t_x}(u, v) = 0 \qquad \textnormal{and} \qquad 
\hat{f}_{t_{\hat{\hat{y}}}}(u, v) = 0 \quad (u, v)\neq (0,0) ~~~\textnormal{but small}, \label{hathatf_a2}   
\end{align}
\textbf{and} requiring $\hat{f}_t$ to have $\delta-2$ more  nodes (all distinct from each other and distinct from 
$(0,0)$ and $(u, v)$). 
The solutions to \cref{hathatf_a2} are given by 
\begin{align}
v =0, \qquad \mathcal{B}^{f_t}_2 & = \frac{\mathcal{B}_4^{f_t}}{12} u^2 + 4 u^3 \mathcal{R}(u) + 2 u^4 \mathcal{R}^{\prime}(u) \qquad \textnormal{and} \nonumber \\
\mathcal{B}^{f_t}_3 & = -\frac{\mathcal{B}_4^{f_t}}{2} u -18 u^2 \mathcal{R}(u) -6u^3 \mathcal{R}^{\prime}(u). \label{soln_two_node_tacnode}
\end{align}
To see how, we first use the third equation of \cref{hathatf_a2} to get $v=0$. Then we use the second and first equations 
of \cref{hathatf_a2} to get the value of $\mathcal{B}^{f_t}_2$ and $\mathcal{B}^{f_t}_3$.\\ 
\hf \hf We now require the curve to have $\delta-2$ more nodes. To do that, first construct a degree $4$ curve that satisfies 
\cref{soln_two_node_tacnode}; we can do that since $\B^{f_t}_4$ only depends on the fourth order derivatives of $f_t$. 
Call this degree $4$ curve $g$.
Let us now assume 
that the points $q_1, q_2, \ldots, q_{\delta-2}$ correspond to $(x_1, y_1), \ldots, (x_{\delta-2}, y_{\delta-2})$ under the affine chart 
we are considering. Define 
\begin{align*}
f_t &:= g(x,y)\cdot ((x-x_1)^2+ (y-y_1)^2)\ldots ((x-x_{\delta-2})^2+ (y-y_{\delta-2})^2).  
\end{align*}
This curve $f_t$ satisfies \cref{hathatf_a2} and has $\delta-2$ nodes. 
This argument works provided the degree of the curve is at least $4+2(\delta-2)$.
Hence, solutions to \cref{two_nodes_eqn} exist, if $d \geq 4 + 2(\delta-2)$.\\ 
\hf \hf Next, let us prove \cref{three_node_int_tacnode_empty}, i.e. we have to show 
that in a neighborhood of a tacnode, we can not have a curve with three distinct nodes. More precisely, we need to show that there 
can not be any solutions to the set of equations 
\begin{align}
\hat{f}_t(u_1, v_1) =0, \quad &\hat{f}_{t_x}(u_1, v_1)=0, \quad \hat{f}_{t_{\hat{\hat{y}}}}(u_1, v_1) =0, \label{three_node_uv1} \\ 
\hat{f}_t(u_2, v_2) =0, \quad &\hat{f}_{t_x}(u_2, v_2)=0, \quad \hat{f}_{t_{\hat{\hat{y}}}}(u_2, v_2) =0,  \label{three_node_uv2}\\ 
& (0,0), ~~(u_1, v_1) ~~\textnormal{and} ~~(u_2, v_2) ~~\textnormal{all distinct (but small)}. \nonumber 
\end{align}
Let us try to solve for the above set of equations. Let us first explicitly write down $\hat{f}_t(x, \hat{\hat{y}})$ as 
\begin{align}
\hat{f}_t(x, \hat{\hat{y}})&= \hat{\hat{y}}^2 + \frac{\mathcal{B}_2^{f_t}}{2!} x^2 + 
\frac{\mathcal{B}_3^{f_t}}{3!} x^3 + \frac{\mathcal{B}_4^{f_t}}{4!} x^4 + \frac{\mathcal{B}_5^{f_t}}{5!} x^5 
+\frac{\mathcal{B}_6^{f_t}}{6!} x^6 + \ldots \label{hat_f_defn}
\end{align}
To begin with, we unwind \cref{three_node_uv1} using the expression for $\hat{f}_t$ as given by \cref{hat_f_defn} 
and solve for $\B^{f_t}_2$ and $\B^{f_t}_3$ in terms of $u_1, v_1$, 
$\B^{f_t}_4, \B^{f_t}_5$ and $\B^{f_t}_6$. We then plug in these values 
for $\B^{f_t}_2$ and $\B^{f_t}_3$ in \cref{hat_f_defn} and plug it in \cref{three_node_uv2}. Now we can solve for 
$\B^{f_t}_4$ and $\B^{f_t}_5$ in terms of $\B^{f_t}_6$ and then plugging back those values in the previous expressions 
for $\B^{f_t}_2$ and $\B^{f_t}_3$, gives us their values in terms of $\B^{f_t}_6$. Doing that, we get  
\begin{align}
v_1, v_2 & =0, \nonumber \\  
\B^{f_t}_2 & = \frac{1}{360} \B^{f_t}_6 u_1^2 u_2^2 + O(|(u_1, u_2)|^5),  \qquad 
\B^{f_t}_3 = -\frac{1}{60} \B^{f_t}_6(u_1^2 u_2 + u_1 u_2^2)+ O(|(u_1, u_2)|^4), \nonumber \\
\B^{f_t}_4 &= \frac{1}{30}\B^{f_t}_6(u_1^2+ 4 u_1 u_2 + u_2^2)+ O(|(u_1, u_2)|^3) \qquad \textnormal{and} \qquad 
\B^{f_t}_5 = -\frac{1}{3}\B^{f_t}_6 (u_1+u_2)+ O(|(u_1, u_2)|^2), \label{A6_nhbd}
\end{align}
where $O(|(u_1, u_2)|^n)$ is as defined in \cref{not_big_O}. 
Hence, $\B^{f_t}_4$ is close to zero, which is a contradiction, since $([f], [\eta], q_{\delta}) \in A_3$. 
Since $\B^{f_t}_5$ is also close to zero, we get the last part of the claim \ref{two_and_three_node_claim} 
(i.e. \cref{three_nodes_canot_be_a4}). Finally, we note that the solutions constructed in \cref{A6_nhbd} immediately 
prove claim \ref{three_nodes_A5_claim} (in fact these are the \textit{only} possible solutions). 
This finishes the proof of claims \ref{two_and_three_node_claim} and \ref{three_nodes_A5_claim}. \\ 
\hf \hf Next, 
we claim that each point of $(A_1^{\delta-2}\circ A_3)\cap \mu$ 
contributes $4$ to the Euler  class in \cref{euler_na1_delta}. Using \cref{soln_two_node_tacnode}, 
we conclude that the multiplicity is the number of small solutions 
$(x, \hat{\hat{y}}, u)$
to the following set of equations 
\begin{align*}
\hat{f}_t(x, \hat{\hat{y}}) & := \hat{\hat{y}}^2 + \frac{\mathcal{B}_2^{f_t}}{2!} x^2 + 
\frac{\mathcal{B}_3^{f_t}}{3!} x^3 + \frac{\mathcal{B}_4^{f_t}}{4!} x^4 + \mathcal{R}(x) x^5 =\varepsilon_0, \\ 
\hat{f_t}_x(x, \hat{\hat{y}}) & := \B^{f_t}_2 x + \frac{\B^{f_t}_3}{2} x^2 + \frac{\B^{f_t}_4}{12} x^3 + 5 x^4 \mathcal{R}(x) + 
\mathcal{R}^{\prime}(x)x^5  =\varepsilon_1, \qquad 
\hat{f_t}_{\hat{\hat{y}}}(x, \hat{\hat{y}}):= 2 \hat{\hat{y}} = \varepsilon_2, \\
\B^{f_t}_2& = \frac{\mathcal{B}_4^{f_t}}{12} u^2 + 4 u^3 \mathcal{R}(u) + 2 u^4 \mathcal{R}^{\prime}(u) \qquad 
\textnormal{and} \qquad 
\mathcal{B}^{f_t}_3 = -\frac{\mathcal{B}_4^{f_t}}{2} u -18 u^2 \mathcal{R}(u) -6u^3 \mathcal{R}^{\prime}(u),
\end{align*}
where $(\varepsilon_0, \varepsilon_1, \varepsilon_2) \in \mathbb{C}^3$ is small and generic. 
Let us write $u:= h+x$ and Taylor expand 
$\mathcal{R}(x+h)$ and $\mathcal{R}^{\prime}(x+h)$ around $h=0$, i.e. 
\begin{align}
\mathcal{R}(x+h)& = \mathcal{R}(x) + h \mathcal{R}^{\prime}(x) + \frac{h^2}{2} \mathcal{R}^{\prime \prime}(x) + \ldots   
\qquad \textnormal{and} \qquad 
\mathcal{R}^{\prime}(x+h) = \mathcal{R}^{\prime}(x) + h \mathcal{R}^{\prime \prime}(x) + \ldots
\end{align}
Hence, substituting the values of $\B^{f_t}_2$, $\B^{f_t}_3$,
$\mathcal{R}(x+h)$ and $\mathcal{R}^{\prime}(x+h)$ 
we conclude that we need to 
find the number of small solutions $(x,h)$ to the following set of equations 
\begin{align}
\frac{(x^2 h^2)\Big( \mathcal{B}_4^{f_t} + O(|(x,h)|) \Big)}{24} & = \varepsilon_3 \qquad \textnormal{and} \label{ep3} \\ 
\frac{(xh)\Big(\mathcal{B}_4^{f_t}h-\mathcal{B}_4^{f_t}x + O(|(x,h)|^2)  \Big)}{12} & = \varepsilon_1, \label{ep2}
\end{align}
where $\varepsilon_3:= \varepsilon_0 - \frac{\varepsilon_2^2}{4}$. 
We claim that we can set $\varepsilon_1$ to be $0$; that is justified in \cref{local_degree}. 
Assuming that claim, we use 
\cref{ep2} to solve for $x$ in terms of $h$ and conclude that 
\begin{align}
x&= h + O(h^2). \label{x_h} 
\end{align}
This is because $x=0$ and $h=0$ can not be solutions to \cref{ep2} 
(since if we plug it back in \cref{ep3}, we will get $0$ and 
not $\varepsilon_3$). 
Plugging in the value of $x$ from \cref{x_h} into \cref{ep3}, we get 
\begin{align}
\frac{\mathcal{B}_4^{f_t}}{24} h^4 + O(h^5) &= \varepsilon_3. \label{h4}
\end{align}
Equation \eqref{h4} clearly has $4$ solutions.\\  
\hf \hf Finally, we need to 
justify \cref{three_node_triple_point} and the corresponding contribution to the Euler  class. 
More precisely, we are going to show that  
\begin{align}
\{([f], [\eta], q_1, \ldots,q_{\delta+1}) \in \overline{A_1^{\delta}\circ \mathcal{S}}_{\mathcal{D}}: 
q_{\delta-2}= q_{\delta-1} = q_{\delta} = q_{\delta+1}\}&= \overline{A_1^{\delta-3}\circ A}_5 \cup 
\overline{A_1^{\delta-3}\circ D}_4. \label{three_nodes_collide_D4_yy}
\end{align}
Just like \cref{a1_a1_a3_exp} is equivalent to \cref{a1_a1_a3_exp_ag2}, we similarly conclude  
that \cref{three_nodes_collide_D4_yy} can be equivalently stated as 
\begin{align}
\{([f], [\eta], q_1, \ldots,q_{\delta}) \in \overline{A_1^{\delta}}: 
q_{\delta-2}= q_{\delta-1} = q_{\delta}\}&= \overline{A_1^{\delta-3}\circ A}_5 \cup 
\overline{A_1^{\delta-3}\circ D}_4. \label{three_nodes_collide_D4_ag2}
\end{align}
Let us define 
\begin{align}
W_1&:= \{ ([f], [\eta], q_1, \ldots,q_{\delta+1}) \in \mathcal{S}_{\mathcal{D}_{\delta+1}}: f(q_{\delta+1})=0, ~\nabla f|_{q_{\delta+1}} =0, ~~
\nabla^2f|_{q_{\delta+1}} \neq 0\} \qquad \textnormal{and} \nonumber \\ 
W_2&:= \{ ([f], [\eta], q_1, \ldots,q_{\delta+1}) \in \mathcal{S}_{\mathcal{D}_{\delta+1}}: f(q_{\delta+1})=0, ~\nabla f|_{q_{\delta+1}} =0, ~~
\nabla^2f|_{q_{\delta+1}} = 0\}. \label{W12_defn}
\end{align}
In order to prove \cref{three_nodes_collide_D4_yy}, it suffices to show that 
\begin{align}
\Big(\{([f], [\eta], q_1, \ldots,q_{\delta+1}) \in \overline{A_1^{\delta}\circ \mathcal{S}}_{\mathcal{D}}: 
q_{\delta-2}= q_{\delta-1} = q_{\delta} = q_{\delta+1}\}\Big)\cap W_1 & = \Big(\overline{A_1^{\delta-3}\circ A}_5\Big)\cap W_1 ~~~
\textnormal{and} \label{three_node_cap_w1} \\
\Big(\{([f], [\eta], q_1, \ldots,q_{\delta+1}) \in \overline{A_1^{\delta}\circ \mathcal{S}}_{\mathcal{D}}: 
q_{\delta-2}= q_{\delta-1} = q_{\delta} = q_{\delta+1}\}\Big)\cap W_2 & = \overline{A_1^{\delta-3}\circ D}_4. \label{three_node_cap_w2}
\end{align}
\hf \hf Note that the right hand side of \cref{three_node_cap_w2} is a subset of $W_2$; hence we didn't write down $\cap W_2$ on the right hand side 
of \cref{three_node_cap_w2}.
Let us first justify \cref{three_node_cap_w1}. Equations 
\eqref{a1_a1_a1_cap_a3_is_empty} and \eqref{three_nodes_canot_be_a4}, 
show that the the left hand side of 
\eqref{three_node_cap_w1} is a subset of its right hand side.
Furthermore, \cref{three_nodes_A5} shows that 
the right hand side of 
\eqref{three_node_cap_w1} is a subset of its left hand side; hence \cref{three_node_cap_w1} is true. \\ 
\hf \hf We will now prove \cref{three_node_cap_w2}. 
Equation \eqref{a1_a1_a1_cap_a3_is_empty} shows that the left hand side of 
\cref{three_node_cap_w2} is a subset of its right hand side. Hence, what remains is to show that the right hand side 
of \cref{three_node_cap_w2} is a subset of its left hand side. 
Before we start the proof of that assertion, let us give an intuitive idea about the significance of that statement.
The statement is saying that every triple point is in the closure of three nodes. 
\begin{figure}[h!]
\vspace*{0.2cm}
\begin{center}\includegraphics[scale = 0.5]{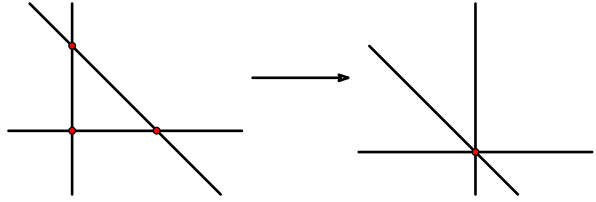}\vspace*{-0.2cm}\end{center}
\caption{Three nodes colliding into a triple point}
\label{fig_3_nodes_triple_pt_ag}
\end{figure}
To summarize, the geometric significance of \cref{three_node_cap_w1} is given by figure \ref{fig_3_nodes_a5_ag} while the geometric significance of 
\cref{three_node_cap_w2} is given by figure \ref{fig_3_nodes_triple_pt_ag}. Equation \eqref{three_nodes_collide_D4_yy} says that these are 
the \textbf{only} two pictures that can occur.\\ 
\hf \hf Let us now prove \cref{three_node_cap_w2}. We will prove the following claim: 
\begin{claim}
\label{three_node_claim_d4}
Let $([f], [\eta], q_1, \ldots, \ldots, q_{\delta}) \in A_1^{\delta-3}\circ D_4$. Then, there exists points 
\begin{align*}
\big([f_t], [\eta_t], q_1(t), \ldots, q_{\delta-3}(t); q_{\delta-2}(t), q_{\delta-1}(t), q_{\delta}(t), q_{\delta+1}(t)\big) \in 
A_1^{\delta-3}\circ \mathcal{S}_{\mathcal{D}}^{4} 
\end{align*}
sufficiently close to $([f], [\eta], q_1, \ldots, \ldots, q_{\delta-3}; q_{\delta}, q_{\delta}, q_{\delta}, q_{\delta})$, such that 
\begin{align}
f_t(q_{i}(t)) &=0, ~~~\nabla f_t|_{q_{i}(t)} = 0 \qquad \textnormal{for} ~~~i=\delta-2, ~~\delta-1 ~~~\textnormal{and} ~~~\delta. 
\label{D4_nhbd_eqn}
\end{align}
\end{claim}

\begin{rem}
We note that claim \ref{three_node_claim_d4} 
implies that the right hand side of \cref{three_node_cap_w2} is a subset of the left hand side. 
\end{rem}

\noindent \textbf{Proof:} Following the setup of the proof of claim \ref{two_and_three_node_claim}, we  
will now work in an affine chart, where we send the plane 
$\mathbb{P}^2_{\eta_t}$ to $\mathbb{C}^2_{z}$ and the 
point $q_{\delta}(t) \in \mathbb{P}^2_{\eta_t}$ to 
$(0,0,0) \in \mathbb{C}^2_{z}$. 
Using this  chart, 
let us write down the Taylor expansion of $f_t$ around the point $(0,0)$, namely  
\begin{align}
f_t(x,y)&= \frac{f_{t_{20}}}{2} x^2 + f_{t_{11}}xy + \frac{f_{t_{02}}}{2} y^2 + \frac{f_{t_{30}}}{6} x^3 + \frac{f_{t_{21}}}{2} x^2y + 
\frac{f_{t_{12}}}{2} x y^2 + \frac{f_{t_{03}}}{6} y^3+\ldots 
\label{eqf_t}
\end{align}
Since $([f], [\eta], q_{\delta}) \in D_4$, we conclude that $f_{t_{20}}, f_{t_{11}}$ and $f_{t_{02}}$
are all small (close to zero). Let us now construct solutions to \cref{D4_nhbd_eqn}. Let us assume that the points
$q_{\delta-1}(t)$ and $q_{\delta-2}(t)$ are sent to $(x_1, y_1, 0)$ and $(x_2, y_2, 0)$ under the affine chart we are considering. 
Hence, constructing solutions to \cref{D4_nhbd_eqn} is same as constructing solutions to the set of equations 
\begin{align}
f_t(x_1, y_1)&=0, \qquad f_{t_x}(x_1, y_1)=0, \qquad f_{t_{y}}(x_1, y_1) =0 \qquad \textnormal{and} \label{eq1} \\ 
f_t(x_2, y_2)&=0, \qquad f_{t_x}(x_2, y_2)=0, \qquad f_{t_{y}}(x_2, y_2) =0, \label{eq2}
\end{align}
where $(0,0), (x_1, y_1)$ and $(x_2, y_2)$ are all distinct (but close to each other).\\
\hf \hf Next, let us define 
\begin{align}
g_t(x,y) & := x f_{t_{x}}(x,y) + y f_{t_{y}}(x,y)- 2 f_t(x,y). \label{g_defn}
\end{align}
The quantity $g(x,y)$ is similarly defined with $f_t$ replaced by $f$.
We note that solving \cref{eq1} and \cref{eq2} is equivalent to solving 
\begin{align}
g_t(x_1, y_1)&=0, \qquad f_{t_x}(x_1, y_1)=0, \qquad f_{t_{y}}(x_1, y_1) =0 \qquad \textnormal{and} \label{eq3} \\ 
g_t(x_2, y_2)&=0, \qquad f_{t_x}(x_2, y_2)=0, \qquad f_{t_{y}}(x_2, y_2) =0, \label{eq4}
\end{align}
where $(0,0), (x_1, y_1)$ and $(x_2, y_2)$ are all distinct (but close to each other). 
We now note that $g_t(x,y)$ and $f_t(x,y)$ have exactly the same cubic term in the Taylor expansion. Furthermore, $g_t(x,y)$ 
has no quadratic term. \\
\hf \hf Let us now study the cubic term of the Taylor expansion of $f$ carefully. 
Let us assume first $f_{{30}} \neq 0$. 
Since $([f], [\eta], q) \in D_4$, we conclude that the cubic term factors into three \textbf{distinct} linear factors. 
Hence, the cubic term can be written as 
\begin{align}
\frac{f_{30}}{6} (x-A_1(0) y) (x-A_2(0)y) (x-A_3(0) y), \label{cubic_defn_0}
\end{align}
where $A_1(0), A_2(0)$ and $A_3(0)$ are all distinct. 
Note that $A_1(0), A_2(0)$ and $A_3(0)$ are explicit expressions involving the coefficients 
$f_{ij}$. If $f_{{30}}=0$, then the cubic term will be of the type 
\begin{align*}
\frac{f_{{21}}}{2}y(x-A_1(0) y)(x-A_2(0) y),
\end{align*}
where $A_1(0)$ and $A_2(0)$ are distinct and $f_{{21}}$ is nonzero. We will assume that 
$f_{{30}} \neq 0$; the case $f_{{30}}=0$ can be dealt with similarly. Hence, $g_t$ 
(or equivalently $f_t$) can be written as 
\begin{align}
g_t(x,y)& = \frac{f_{t_{30}}}{6} (x-A_1 y) (x-A_2 y) (x-A_3 y) + O(|(x,y)|^4), \label{cubic_defn_0}
\end{align}
where $A_i$ are the same as $A_i(0)$, but with the $ f_{ij}$ replaced by $ f_{t_{ij}}$. For notational simplicity, we 
denoted these quantities by the letter $A_i$ and not $A_i(t)$. \\ 
\hf \hf Let us now make a change of coordinates 
\begin{align}
x & := \hat{x} + O(|(\hat{x}, \hat{y})|^2) \qquad \textnormal{and} \qquad y:= \hat{y} + O(|(\hat{x}, \hat{y})|^2),  \label{x_y_hat}
\end{align}
such that 
\begin{align}
g_t& = \frac{f_{t_{30}}}{6} (\hat{x}-A_1 \hat{y}) (\hat{x}-A_2\hat{y}) (\hat{x}-A_3 \hat{y}). \label{g_hat}
\end{align}
Hence, $g_t=0$ has three distinct solutions, given by $\hat{x} = A_i \hat{y}$
for $i=1,2$ and $3$. Converting back in terms of $x$, we conclude that the solutions to $g_t(x,y)=0$ (where $(x,y)$ is small but nonzero) 
are given by 
\begin{align}
y&= u \qquad \textnormal{and} \qquad x = A_i u + E_i(u),  \label{g_soln_2}
\end{align}
where $E_i(u)$ is a second order term in $u$  (and $u$ is small but nonzero).\\  
\hf \hf Next, for notational simplicity 
we will denote $f_{t_{02}}$ by the letter $w$. 
Let us consider the solution $y:= u$ and $x = A_1 u + E_1(u)$ of the equation $g_t(x,y)=0$.  
Plugging this in 
$f_{t_{x}}(x, y) =0$ and $f_{t_{y}}(x, y) =0$ 
and solving for $A_1 f_{t_{11}}$ and $A_1^2 f_{t_{20}}$, we conclude that
\begin{align}
A_1 f_{t_{11}} &= \frac{A_1f_{t_{30}}}{6} (A_1-A_2)(A_1-A_3)u - w + O(|(u,w)|^2) \qquad \textnormal{and} 
\nonumber \\ 
A_1^2 f_{t_{20}} & = -\frac{A_1 f_{t_{30}}}{3} (A_1-A_2)(A_1-A_3)u + w + O(|(u,w)|^2). \label{3_node_sol_1}
\end{align}
Let us now consider a second solution to $g_t(x,y)=0$ (where $(x,y)$ is small but nonzero). 
This will be given by 
$y:= v$ and $x:= A_2 v+ E_2(v)$, where  $v$ is small but nonzero (or the analogous thing with $A_2$ replaced by $A_3$). 
Using \cref{3_node_sol_1} to express the values of 
$f_{t_{11}}$ and $f_{t_{20}}$ in terms of $u$ and $w$ and then  
using 
$f_{t_{x}}(x, y) =0$, we conclude that 
\begin{align}
w & = \frac{f_{t_{30}}}{6} \Big(A_1^3 -2 A_1^2 A_2-A_1^2 A_3 +2 A_1
   A_2 A_3\Big)u + \frac{f_{t_{30}}}{6}\Big(A_1^2 A_3- A_1^2 A_2\Big)v  + O(|(u,w)|^2). \label{w1}
\end{align}
Similarly, using 
\cref{3_node_sol_1} to express the values of 
$f_{t_{11}}$ and $f_{t_{20}}$ in terms of $u$ and $w$ and then  
using 
$f_{t_{y}}(x, y) =0$, we conclude that 
\begin{align}
w & = \frac{f_{t_{30}}}{6}\Big(-A_1^2 A_2 + A_1 A_2 A_3 \Big)u + 
\frac{f_{t_{30}}}{6}\Big(-A_1 A_2^2 + A_1 A_2 A_3\Big)v + O(|(u,w)|^2). \label{w2}
\end{align} 
Equating the right hand sides of \cref{w1} and \cref{w2}, we conclude that 
\begin{align}
\frac{f_{t_{30}}}{6} A_1 (A_1-A_2)(A_1-A_3)u -\frac{f_{t_{30}}}{6} A_1(A_1-A_2)(A_2-A_3)v + O(|(u,v,w)|^2) & =  0. \label{Q}
\end{align}
From \cref{Q}, we can further conclude that 
\begin{align}
A_1 v & = \Big(\frac{A_1-A_3}{A_2-A_3}\Big)(A_1 u) + O(|(u,w)|^2). \label{v_value}
\end{align}
Finally, substituting the value for $v$ from \cref{v_value} into $w$ in \cref{w1}, we get that 
\begin{align}
w &= -\frac{f_{t_{30}}}{3} A_1 A_2 (A_1 - A_3) u + O(|(u,w)|^2)  
~~ \implies ~~w = -\frac{f_{t_{30}}}{3} A_1 A_2 (A_1 - A_3) u + O(|u|^2). \label{w_value_final}
\end{align}
Plugging the value of $w$ from \cref{w_value_final} 
in \cref{3_node_sol_1}, we conclude that 
\begin{align*}
f_{t_{11}} &= \frac{f_{t_{30}}}{6}(A_1 + A_2) (A_1 - A_3) u + O(|u|^2) \qquad \textnormal{and} \qquad 
f_{t_{20}}  = -\frac{f_{t_{30}}}{3} (A_1 - A_3) u + O(|u|^2). 
\end{align*}
Hence, solutions to \cref{eq1} and \cref{eq2} exist, given by 
\begin{align}
(x_1, y_1) & = (A_1 u + E_1(u), u), \qquad (x_2, y_2) = \Big(A_2\frac{(A_1-A_3)}{(A_2-A_3)}u + E_2(u), ~\frac{(A_1-A_3)}{(A_2-A_3)}u + E_4(u)\Big),
\nonumber \\
f_{t_{11}} &= \frac{f_{t_{30}}}{6}(A_1 + A_2) (A_1 - A_3) u + E_5(u), \qquad 
f_{t_{20}} = -\frac{f_{t_{30}}}{3} (A_1 - A_3) u + E_6(u) \qquad \textnormal{and} \nonumber \\ 
f_{t_{02}} & = -\frac{f_{t_{30}}}{3} A_1 A_2 (A_1 - A_3) u + E_7(u), \label{d4_three_nodes_soln}
\end{align}
where $u$ is small and nonzero and the 
$E_i$ are all second order terms. Furthermore, there are \textbf{exactly} $6$ distinct solutions, that corresponds to 
$(A_1, A_2)$ being replaced with $(A_i, A_j)$, where the $(A_i, A_j)$ are ordered 
(or alternatively, we can think of this this way; the $(A_i, A_j)$ is unordered as far as the construction of $f_t$ is concerned, but 
we can permute the values of $(x_1, y_1)$ and $(x_2, y_2)$). 
 This proves \cref{three_node_claim_d4} and hence proves  
\cref{three_node_triple_point}. \qed \\ 
\hf \hf Let us now justify the multiplicity. We claim that each point of 
$(A_1^{\delta-3}\circ D_4)\cap \mu$ 
contributes $18$ to the Euler  class in \cref{euler_na1_delta}. 
As we just explained, there are exactly $6$ distinct solutions to 
\cref{eq1} and \cref{eq2}; we will call each distinct solution of 
\cref{eq1} and \cref{eq2} a 
\textbf{branch} of a neighborhood of $A_1^{\delta-3}\circ D_4$ inside $\overline{A_1^{\delta}}$. 
Since there are $6$ branches, it suffices to show that the multiplicity from each branch is $3$ 
(in which case the total contribution to the Euler  class will be $18$). Let us now compute the multiplicity from each branch.\\  
\hf \hf Let us consider the branch given by \cref{d4_three_nodes_soln}. 
The multiplicity from this branch 
is the number of small solutions 
$(x, y, u)$
to the 
following set of equations 
\begin{align}
f_t(x,y)& = \varepsilon_0, \qquad f_{t_x}(x,y) = \varepsilon_1 \qquad \textnormal{and} \qquad f_{t_y}(x,y) = \varepsilon_2 \label{epp12}
\end{align}
where $(\varepsilon_0, \varepsilon_1, \varepsilon_2) \in \mathbb{C}^3$ is small and generic and $f_{t_{20}}, f_{t_{11}}$ 
and $f_{t_{02}}$ are as given in \cref{d4_three_nodes_soln}. We claim that we can set $\varepsilon_1$ and 
$\varepsilon_2$ to be zero; this is justified in section \ref{local_degree}. Hence, we need to find the number of small 
solutions $(x,y,u)$ to the set of equations 
\begin{align*}
f_t(x,y)& = \varepsilon_0, \qquad f_{t_x}(x,y) = 0 \qquad \textnormal{and} \qquad f_{t_y}(x,y) = 0. 
\end{align*}
This is same as the number of small solutions $(x,y,u)$ to the set of equations 
\begin{align}
g_t(x,y)& = -2\varepsilon_0, \label{g_only}\\ 
f_{t_x}(x,y) & = 0 \qquad \textnormal{and} \qquad f_{t_y}(x,y) = 0, \label{g_eqn_solve}
\end{align}
where $g_t(x,y)$ is as defined in \cref{g_defn}. 
Let us start by solving only the two equations in \cref{g_eqn_solve}. 
Plugging in the values for $f_{t_{20}}, f_{t_{11}}$ 
and $f_{t_{02}}$ as given in \cref{d4_three_nodes_soln}  
and solving the equation $f_{t_x}(x,y)=0$, we conclude that 
\begin{align}
\Big(-2 x + (A_1 + A_2) y\Big) \Big(u+ O(|u|^2)\Big) & = 
\frac{(3 x^2 - 2(A_1+A_2+A_3)x y + (A_1 A_2 + A_1 A_3 + A_2 A_3) y^2)}{(A_3-A_1)} \nonumber \\ 
  & ~~ + O(|(x,y)|^3). \label{ou}
\end{align}
Similarly, plugging in the values for $f_{t_{20}}, f_{t_{11}}$ 
and $f_{t_{02}}$ as given in \cref{d4_three_nodes_soln}  
and solving the equation $f_{t_y}(x,y)=0$, we conclude that 
\begin{align}
\Big( (A_1+A_2)x -2A_1 A_2 y \Big) \Big(u+ O(|u|^2)\Big) & = -\frac{((A_1+A_2+A_3)x^2 -2(A_1 A_2 + A_1 A_3 + A_2 A_3)xy 
+3A_1A_2 A_3 y^2)}{(A_3-A_1)} \nonumber \\ 
& ~~ + O(|(x,y)|^3). \label{ou3}
\end{align}
Multiplying \cref{ou} by $(A_1+A_2)x -2A_1 A_2 y$ and multiplying \cref{ou3} by 
$(-2 x + (A_1 + A_2) y)$, we conclude that
\begin{align}
\Big(x -A_1 y\Big) \Big(x - A_2 y\Big)\Big( (A_1+A_2-2A_3)x -(2A_1 A_2-A_1A_3 - A_2 A_3) y \Big)+
O(|x,y|^4) & = u^2 O(|(x,y)|^2). \label{x_A_y_soln}
\end{align}
Let us now solve \cref{x_A_y_soln}. Let us make a change of coordinates 
\begin{align*}
x & = \hat{x} + O(|(\hat{x}, \hat{y})|^2) \qquad \textnormal{and} \qquad y = \hat{y} + O(|(\hat{x}, \hat{y})|^2) 
\end{align*}
such that \cref{x_A_y_soln} can be rewritten as 
\begin{align}
\Big(\hat{x} -A_1 \hat{y}\Big) \Big(\hat{x} - A_2 \hat{y}\Big)\Big( (A_1+A_2-2A_3)\hat{x} -(2A_1 A_2-A_1A_3 - A_2 A_3) \hat{y} \Big) & =  u^2 O(|(\hat{x}, \hat{y})|^2) \label{x_A_y_soln_hat}
\end{align}
Using \cref{x_A_y_soln_hat}, we solve for $\hat{x}$ in terms for $\hat{y}$ and $u$ and convert back to $x$ and $y$ to conclude that the 
only possible solutions are given by 
\begin{align}
x & = A_1 y + E_8(y,u)\qquad \textnormal{or} \qquad x = A_2 y + E_9(y,u) \qquad \textnormal{or} \nonumber \\ 
(A_1+A_2-2A_3)x &= (2A_1 A_2-A_1A_3 - A_2 A_3) y + E_{10}(y,u), \label{xy_soln}
\end{align} 
such that $E_i(y,0) = O(|y|^2)$, for $i=8,9$ and $10$. 
Plugging the three solutions obtained in \cref{xy_soln} into \cref{ou}, solving for $y$ in terms of $u$ 
and then plugging that back into \cref{xy_soln} to express $x$ in terms of $u$, 
we conclude that the only possible solutions to \cref{g_eqn_solve}
are given by 
\begin{align}
(x, y)& = \Big(A_1u + \widetilde{E}_1(u), u+ \widehat{E}_1(u) \Big) \qquad \textnormal{or} \label{sol1}\\ 
(x,y) & = 
\Big(A_2\Big(\frac{A_1-A_3}{A_2-A_3}\Big)u + \widetilde{E}_2(u), \Big(\frac{A_1-A_3}{A_2-A_3}\Big)u+\widehat{E}_2(u)\Big) 
\qquad \textnormal{or} \label{sol2} \\ 
(x,y) & = \Big( \frac{(2A_1 A_2-A_1 A_3-A_2 A_3)}{3(A_2-A_3)} u + \widetilde{E}_3(u) , \frac{(A_1+A_2-2A_3)}{3(A_2-A_3)} u + 
\widehat{E}_3(u)\Big), \label{sol3}
\end{align}
where $\widetilde{E}_i(u)$ and $\widehat{E}_i(u)$ are second order terms (for $i=1,2$ and $3$). From 
\cref{d4_three_nodes_soln}, we conclude that the solutions  in \cref{sol1} and \cref{sol2} with 
$\widetilde{E}_i(u)$ replaced by $E_i(u)$ and $\widehat{E}_i(u)$ replaced by $0$ (for $i=1$ and $2$) 
is a solution to \cref{g_eqn_solve}. 
Since the solutions in \cref{sol1} and \cref{sol2} are the \textbf{only} solutions to \cref{g_eqn_solve}, we conclude that 
$\widetilde{E}_i(u) = E_i(u)$ and $\widehat{E}_i(u) =0$ (for $i=1$ and $2$). Hence, if we plug the solutions obtained 
from \cref{sol1} and \cref{sol2} into $f_t(x,y)$ (or equivalently $g_t(x,y)$), we will get $0$ and not $\varepsilon_0$. 
Hence, we reject the solutions given by \cref{sol1} and \cref{sol2}.  \\ 
\hf \hf It remains to consider the solution given by \cref{sol3}. 
Plugging in the expression for $x$ and $y$ from \cref{sol3} into $g_t(x,y)$ gives us 
\begin{align}
g_t(x,y) & = \Big(\frac{(A_2-A_1)^2 (A_3-A_1)^2}{162 (A_2 - A_3)}\Big) u^3 + O(u^4). \label{g_3_mult}
\end{align}
From \cref{g_3_mult}, we conclude that 
$g_t(x,y) = -2\varepsilon_0$ has $3$ solutions. This justifies the multiplicity and 
concludes the proof of \cref{na1_delta}. \qed 


\subsubsection{Local degree of a smooth map} 
\label{local_degree}
It remains to show why we could set $\varepsilon_2$ to be $0$ in \cref{ep2} and set 
$(\varepsilon_1, \varepsilon_2)$ to be $(0,0)$ in \cref{epp12}. Let us first recall 
the definition of the local degree of a smooth map around a given point. We will follow the discussion and theory 
developed in \cite{Kes}.\\ 
\hf \hf Let us begin with the proposition 2.1.2  of \cite{Kes}. The statement is as follows:
\begin{prp}\label{local_degree_kes}
Let $f \in C^2(\bar{\Omega}, \mathbb{R}^n)$ where $\Omega$ is an open subset of $\mathbb{R}^n$ and let $b \notin f(\partial \Omega)$. Let $\rho_0$ be the distance between $b$ and $f(\partial \Omega)$ with $\rho_0>0$. Let $b_1,b_2 \in B(b;\rho_0),$ the ball of radius $\rho_0$ with center $b$. If $b_1,b_2$ both are regular values of $f$, then $\textnormal{deg}(f,\Omega,b_1)=\textnormal{deg}(f,\Omega,b_2)$ where $\textnormal{deg}(f,\Omega,y)$ represent the degree of $f$ at $y$ (i.e. the number of solutions to the equation $f(x)=y$ in $\Omega$). 
\end{prp}
\hf \hf Let us first justify the assertion for \cref{ep2}. Let 
$U$ be 
an open ball in $\mathbb{C}^2$ with center $(0,0)$ and radius $r$, where $r$ is sufficiently small and positive real number.
Consider the map 
$\varphi: U \longrightarrow \mathbb{C}^2$, given by 
\begin{align*}
\varphi(x,h)& = (\varphi_{1}(x,h), \varphi_2(x,h))\\ 
&:= \Big(\frac{(x^2 h^2)\Big( \mathcal{B}_4^{f_t} + O(|(x,h)|) \Big)}{24},  
~~\frac{(xh)\Big(\mathcal{B}_4^{f_t}h - \mathcal{B}_4^{f_t}x + O(|(x,h)|^2)\Big)}{12}\Big).
\end{align*}\\
Before proceeding, let us first prove the following claim:
\begin{claim}
If $\varepsilon \neq 0,$ then the point $(\varepsilon,0)$ is a regular value of $\varphi$.
\end{claim}
\begin{proof} Let us assume $\varphi(x, h) = (\varepsilon, 0)$. 
Using the fact that $\varphi_2(x,h) =0$, we conclude that 
\[ x(h) = h + O(h^2).\] 
Plugging in this value of $x$ in $\varphi_1(x,h)$, we conclude that 
\begin{align}
h^4\Big(\frac{\mathcal{B}^{f_t}_4}{24} + O(h) \Big) & = \varepsilon. \label{ze1}
\end{align}
Note that if $h$ is sufficiently small, then $\frac{\mathcal{B}^{f_t}_4}{24} + O(h)$ 
is nonzero, since $\mathcal{B}^{f_t}_4$ is nonzero. We also note that since $\varepsilon$ is nonzero, \cref{ze1} 
implies that $h$ is nonzero. \\  
\hf \hf Next, let us compute the determinant of the differential of $\varphi$ at $(x(h), h)$. It is given by
\begin{align}
M &:=  \textnormal{det}\begin{pmatrix}
\varphi_{1_x}& \varphi_{1_h}  \\ 
\varphi_{2_x}& \varphi_{2_h}  \\ 
\end{pmatrix}\Big|_{(x(h), h)} = h^5\Big(\frac{(\mathcal{B}_4^{f_t})^2}{72}   + O(h)\Big) \label{ze2}
\end{align}
Using \cref{ze1} and  \cref{ze2}, we conclude that 
\begin{align}
M & = h^4 \cdot h\Big(\frac{(\mathcal{B}_4^{f_t})^2}{72}   + O(h) \Big) \nonumber \\ 
  & = \varepsilon \frac{h\Big(\frac{(\mathcal{B}_4^{f_t})^2}{72}   + O(h) \Big)}{\Big(\frac{\mathcal{B}^{f_t}_4}{24} + O(h) \Big)}. \label{ze3} 
\end{align}
Since, $\mathcal{B}^{f_t}_4$ is nonzero,  $h$ is small and nonzero and $\varepsilon$ is nonzero, 
we conclude from \cref{ze3} that $M$ is nonzero. Hence, $(\varepsilon,0)$ is a regular value of $\varphi$.  
\end{proof}
\hf \hf Next, we note that if $S$ is a non empty subset of $\mathbb{C}^2$, then 
the distance function $d_S: \mathbb{C}^2 \longrightarrow \mathbb{R}$ 
is a continuous function. 
Hence, the set \begin{align*}
V & :=(d_{\varphi(\partial U)}-d_{X})^{-1}(0,\infty) \\
 & = \{(\varepsilon_1,\varepsilon_2)\in \mathbb{C}^2\mid d_{\varphi(\partial U)}(\varepsilon_1,\varepsilon_2) > d_{X}(\varepsilon_1,\varepsilon_2)\}
\end{align*}
 is an open subset of $\mathbb{C}^2$, where the function $d_{X}$ denotes the distance from $x$-axis.
Note that $d_{X}(\varepsilon_1,\varepsilon_2)=|\varepsilon_2|$ and this distance is achieved by taking the distance from the point $(\varepsilon_1,\varepsilon_2)$ to the point $(\varepsilon_1,0)$ on $x$-axis. \\
\hf \hf Now, we will show that $V \cap \varphi(U) \neq \emptyset.$ Note that 
$$\partial U=\{ (x,h)\in \mathbb{C}^2 : |x|^2+|h|^2=r^2\}.$$ 
Observe that $\partial U$ is compact; so $\varphi(\partial U)$ is compact and hence closed in $\mathbb{C}^2$. Hence, $d_{\varphi(\partial U)}(\varepsilon,0) = 0$ if and only if $(\varepsilon,0) \in \varphi( \partial U)$. We conclude that $(\varepsilon,0) \in V$ if and only if $(\varepsilon,0) \notin \varphi( \partial U).$ Now, let $(\varepsilon,0) \in \varphi( \partial U)$. Let us assume $\varphi(x,h)=(\varepsilon,0)$ with $|x|^2+|h|^2=r^2$ and $\varepsilon \neq 0$. We conclude from $\varphi_2(x,h)=0$ and \cref{ze1} that 
\begin{align*}
x(h) & = h + O(h^2) \qquad  \textnormal{and} \qquad 
h^4\Big(\frac{\mathcal{B}^{f_t}_4}{24} + O(h) \Big) = \varepsilon.
\end{align*}
Now using the fact $|x|^2+|h|^2=r^2$, we conclude that $|\varepsilon| = \dfrac{|\mathcal{B}^{f_t}_4|}{96}r^4+O(r^5)$. Hence we get either $\varepsilon=0$ or $|\varepsilon| = \dfrac{|\mathcal{B}^{f_t}_4|}{96}r^4+O(r^5)$. 
Note that $\dfrac{|\mathcal{B}^{f_t}_4|}{96}r^4+O(r^5) \neq 0$ as 
$\mathcal{B}^{f_t}_4 \neq 0$ and $r$ is sufficiently small. 
So, $(\varepsilon,0) \in V$ for all nonzero $\varepsilon$ with  
$|\varepsilon| < \dfrac{|\mathcal{B}^{f_t}_4|}{96}r^4+O(r^5)$ (i.e. for all $|\varepsilon|$ sufficiently small). 
From \cref{h4} we concluded that the system $\varphi(x,h)=(\varepsilon,0)$ has solutions in $U$ where $\varepsilon$ is small but nonzero. Hence $(\varepsilon,0) \in V \cap \varphi(U)$ for some nonzero $\varepsilon$ with $|\varepsilon| < \dfrac{\mathcal{B}^{f_t}_4}{96}r^4+O(r^5)$. Hence, $V \cap \varphi(U)$ is non empty. \\ 
\hf \hf Next, we note that since $\varphi:U \longrightarrow \mathbb{C}^2$ is a non constant holomorphic map,
$\varphi(U)$ is an open subset of $\mathbb{C}^2$. Hence, $V \cap \varphi(U)$ is a non empty open subset of 
$\mathbb{C}^2$ and has nonzero measure. 
Using Sard's Theorem (applied to the function $\varphi:U \longrightarrow \mathbb{C}^2$), we conclude that 
$V \cap \varphi(U)$ contains regular values of $\varphi$. Let $(\varepsilon_1,\varepsilon_2) \in V \cap \varphi(U)$ be a regular value of $\varphi$. Therefore by definition of $V$, $d_{\varphi(\partial U)}(\varepsilon_1,\varepsilon_2) > |\varepsilon_2| \geq 0.$ Now, $\varphi(\partial U)$ is a closed subset of $\mathbb{C}^2$ and $d_{\varphi(\partial U)}(\varepsilon_1,\varepsilon_2) >0$ together implies that $(\varepsilon_1,\varepsilon_2) \notin \varphi( \partial U).$ Hence all the hypothesis of \cref{local_degree_kes} are satisfied. We conclude from the proposition that 
$\textnormal{deg}(\varphi,U,(\varepsilon_1,\varepsilon_2))=\textnormal{deg}(\varphi,U,(\varepsilon_1,0))$, 
i.e. the number of solutions in $U$ to both the equations $\varphi(x,h)=(\varepsilon_1,\varepsilon_2)$ and $\varphi(x,h)=(\varepsilon_1,0)$ are same. This justifies our claim in \cref{ep2}. \\
\hf \hf Let us now justify the assertion for \cref{epp12}. The argument is similar to the previous argument. We just need to prove the following claim:
\begin{claim} 
Let $U\subseteq \mathbb{C}^3$ be a small open neighborhood of $(0,0,0)$ and $\varphi: U \longrightarrow \mathbb{C}^3$ be given by
\begin{align*}
\varphi(x,y,u)&:= \Big(f_t(x,y),f_{t_x}(x,y),f_{t_y}(x,y)\Big), 
\end{align*}
where 
$f_t$ 
is as given in \cref{eqf_t} and $f_{t_{20}},f_{t_{11}},f_{t_{02}}$ are as given in \cref{d4_three_nodes_soln}. 
Let $\widehat{U} \subseteq \mathbb{C}$ be a small open neighborhood of $0$. 
If $\varepsilon$ is a generic point of $\widehat{U}$, then $(\varepsilon,0,0)$ is a regular value of $\varphi$.  
\end{claim}
\begin{proof}
Let $(x,y,u) \in U$ such that $\varphi(x,y,u)=(\varepsilon,0,0)$. We note that 
\begin{align}
\det \begin{pmatrix}
f_{t_x}(x,y) & f_{t_{xx}}(x,y) & f_{t_{yx}}(x,y) \\ 
f_{t_y}(x,y) & f_{t_{xy}}(x,y) & f_{t_{yy}}(x,y) \\
f_{t_u}(x,y) & f_{t_{x_u}}(x,y) & f_{t_{y_u}}(x,y)    
\end{pmatrix} & = f_{t_u}(x,y) \cdot \det \begin{pmatrix}
f_{t_{xx}}(x,y) & f_{t_{yx}}(x,y) \\ 
f_{t_{xy}}(x,y) & f_{t_{yy}}(x,y)
\end{pmatrix}. \label{y4}
\end{align}
This is because $f_{t_x}(x,y)$ and $f_{t_y}(x,y)$ 
are both equal to zero. 
We now note that 
$f_t$ has an $A_1$ singularity at $(0,0)$; hence determinant of Hessian of $f_t$ does not vanish at $(0,0)$. 
Since $(x,y)$ is small, we conclude that the determinant of the Hessian of $f_t$ at $(x,y)$ is nonzero. 
Hence, if the right hand side of \cref{y4} is zero, then $f_{t_u}(x,y)$ has to be zero. We claim this is not possible for a generic 
$\varepsilon$. To see why this is so, note that the solution to the equation $\varphi(x,y,u)=(\varepsilon,0,0)$ with $\varepsilon \in \widehat{U}$ is given by \cref{sol3}. 
After plugging the value of $(x,y)$ obtained in \cref{sol3} to the expression  of $f_{t}(x,y)$, we conclude from \cref{g_3_mult} that 
\begin{align*}
f_{t_u}(x,y)& = -\Big(\frac{3(A_2-A_1)^2 (A_3-A_1)^2}{324 (A_2 - A_3)}\Big) u^2 + O(u^3)
\end{align*}
Note that $f_{t_u}(x,y)$ is a power series of $u$ which is not identically zero in a small open subset of $\mathbb{C}$ containing the origin and hence it has only finitely many zeros. 
We conclude that $(\varepsilon,0,0)$ is a regular value of $\varphi$ for all but a finite set of $\varepsilon$; in particular 
for a generic $\varepsilon$, $(\varepsilon,0,0)$ is a regular value of $\varphi$.
\end{proof}


\subsection{Proof of \Cref{npa1}: computation of $N(A_1^{\delta} \mathcal{P} A_1)$ when $0\leq \delta \leq 2$} 
We will now justify our formula for 
$N(A_1^{\delta} \mathcal{P} A_1,r, s, n_1, n_2, n_3, \theta)$, when $0 \leq \delta \leq 2$.
If $\theta =0$, then the formula 
follows from \cref{deg_to_one_up_to_down}.\\ 
\hf \hf Let us now assume $\theta>0$. 
Recall that (as per the definition in section \ref{notation})
\begin{align*}
\A^{\delta}_1 \circ \overline{\hat{\A}}_1 := \{ ([f], [\eta], q_1, \ldots, q_{\delta}, 
l_{q_{\delta+1}}) \in \mathcal{S}_{\mathcal{D}_{\delta}}\times_{\mathcal{D}} \mathbb{P} W_{\mathcal{D}}: 
&\textnormal{$f$ has a singularity of type $\A_1$ at $q_1, \ldots, q_{\delta}$}, \\ 
                   & ([f], [\eta], l_{q_{\delta+1}}) \in \overline{\hat{A}}_1, ~~\textnormal{$q_1, \ldots, q_{\delta+1}$ all distinct}\}. 
\end{align*}
Let $\mu$ be a generic cycle, representing 
the class 
\begin{align*}
[\mu] = \mathcal{H}_L^r \cdot \mathcal{H}_p^s \cdot a^{n_1} \lambda^{n_2} (\pi_{\delta+1}^*H)^{n_3} (\pi_{\delta+1}^*\lambda_{W})^{\theta}. 
\end{align*}
We now define a section of the following bundle 
\begin{align*}
\Psi_{\PP \A_1}:  \A^{\delta}_1 \circ \overline{\hat{\A}}_1 
\lra \mathcal{L}_{\PP \A_1} & := \gamma_{\DD}^\ast\otimes \gamma_W^{\ast 2}\otimes \gamma_{\mathbb{P}^3}^{* d}, \qquad \textnormal{given by}\\
\{\Psi_{\PP \A_1}([f], q_1,\ldots,q_\delta,l_{q_{\delta+1}})\}(f\otimes v^{\otimes 2}) &:= \nabla^2 f|_{q_{\delta+1}}(v, v).  \nonumber
\end{align*}
We will show shortly that this section is transverse to zero. Next, let us define 
\begin{align*}
\mathcal{B} &:= \overline{\A^{\delta}_1 \circ \overline{\hat{\A}}}_1- \A^{\delta}_1 \circ \overline{\hat{\A}}_1. 
\end{align*}
Hence
\begin{align}
\lan e(\mathcal{L}_{\PP \A_1}), 
~~[\overline{\A^{\delta}_1 \circ \overline{\hat{\A}}}_1] \cap [\mu] \ran & = 
\N(\A_1^{\delta}\PP \A_1, r, s, n_1, n_2, n_3, \theta)  
+ \mathcal{C}_{\mathcal{B}\cap \mu},
\label{npa1_euler_bdry_eqn}
\end{align}
where as before, $\mathcal{C}_{\mathcal{B} \cap \mu}$ denotes the contribution of the section 
to the Euler class from $\mathcal{B} \cap \mu$. 
When $\delta=0$, the boundary $\mathcal{B}$ is empty. Hence, plugging in $\mathcal{C}_{\mathcal{B} \cap \mu}=0$ 
and unwinding the left hand side of \cref{npa1_euler_bdry_eqn} 
gives us the formula of \Cref{npa1} for $\delta=0$. \\ 
\hf \hf Let us now assume $\delta>0$. 
Given $k$ distinct integers $i_1, i_2, \ldots, i_k \in [1, \delta+1]$, let 
$\Delta_{i_1, \ldots, i_{k}}$ be as defined in the proof of \Cref{na1_delta}. Let us define 
\begin{align*}
\hat{\Delta}_{i_1, \ldots, i_{k}}&:= \pi^{-1}(\Delta_{i_1, \ldots, i_{k}}), 
\end{align*}
where $\pi:\mathcal{S}_{\mathcal{D}_{\delta}}\times_{\mathcal{D}} \mathbb{P} W_{\mathcal{D}} 
\longrightarrow \mathcal{S}_{\mathcal{D}_{\delta+1}}$ is the projection map. Let us define 
\begin{align*}
\mathcal{B}(q_{i_1}, \ldots, q_{i_{k-1}}, l_{q_{\delta+1}}) & := \mathcal{B}\cap \hat{\Delta}_{i_1, \ldots, i_{k-1}, \delta+1}.
\end{align*}
Let us now consider $\mathcal{B}(q_i, l_{q_{\delta+1}})$. We claim that, 
\begin{align}
\B (q_i, l_{q_{\delta +1 }}) & \approx \overline{\A_1^{\delta-1} \circ \hat{\A}}_3, \label{two_nodes_Hat_A3}
\end{align}
where $\B (q_i, l_{q_{\delta +1 }})$ is identified as a subset of 
$\mathcal{S}_{\mathcal{D}_{\delta-1}}\times_{\widehat{\mathbb{P}}^3} \mathbb{P} W_{\mathcal{D}}$ in the obvious way 
(namely via the inclusion map where the $(\delta+1)^{\textnormal{th}}$ point is equal to the $i^{\textnormal{th}}$ point). 
We will  justify that shortly. 
Let us now intersect $\overline{\A_1^{\delta-1} \circ \hat{\A}}_3$ with $\mu$. This will 
be an isolated set of finite points. Hence, the section $\psi_{\PP \A_1}$ 
will not vanish on $\overline{\A_1^{\delta-1} \circ \hat{\A}}_3 \cap \mu$.  
Hence it does not 
contribute to the Euler class. \\
\hf \hf Next, let us consider $\B (q_{i_1}, q_{i_2}, l_{q_{\delta +1 }})$. 
We claim that
\begin{align}
\B (q_{i_1}, q_{i_2}, l_{q_{\delta +1 }}) & \approx \overline{\A_1^{\delta-2} \circ \hat{A}}_5 \cup 
\overline{\A_1^{\delta-2} \circ \widehat{D}}_4 .  \label{three_nodes_Hat_D4}
\end{align}
The set $\overline{\A_1^{\delta-2} \circ \hat{A}}_5 \cap \mu$ is empty since the sum total of the dimensions of these two varieties 
is one less than the dimension of the ambient space. Next, we note that  
the section $\Psi_{\PP \A_1}$ vanishes everywhere on 
$\overline{\A_1^{\delta-2} \circ \widehat{D}}_4$; hence it also vanishes on 
$\overline{\A_1^{\delta-2} \circ \widehat{D}}_4 \cap \mu$. 
We claim that the 
contribution from each of the points of $\B (q_{i_1}, q_{i_2}, l_{q_{\delta +1 }}) \cap \mu$  
is $6$. Hence the total contribution from all the components 
of type $\B (q_{i_1}, q_{i_2}, l_{q_{\delta +1 }})$ is 
\begin{align*}
6 \binom{\delta}{2} \N(\A_1^{\delta-2} \widehat{D}_4, n_1, n_2,n_3, \theta).   
\end{align*}
Plugging this in \cref{npa1_euler_bdry_eqn} gives us the formula of \cref{npa1}. \\ 
\hf \hf Let us now justify the transversality, closure and multiplicity claims. We will follow the setup of \cref{na1_delta}. 
Suppose 
\begin{align*}
\Psi_{\PP A_1}([f], [\eta], q_1, \ldots, q_{\delta}, l_{q_{\delta+1}}) & =0. 
\end{align*}
As before we assume $\eta$ determines the plane where the last component is zero and $q_{\delta+1}:= [0,0,1,0]$. 
Let us consider $T\mathbb{P}^2_{\eta}|_{[q_{\delta+1}]}$. Let $\partial_x$ and $\partial_y$ be the standard basis 
vectors for $T\mathbb{P}^2_{\eta}|_{[q_{\delta+1}]}$ (corresponding to the first two coordinates). Hence 
\begin{align*}
l_{q_{\delta+1}} & = [a \partial_x + b \partial_y] \in \mathbb{P}T\mathbb{P}^2_{\eta}|_{[q_{\delta+1}]}
\end{align*}
for some complex numbers $a,b$ not both of which are zero. Without loss of generality, we can assume 
$l_{q_{\delta+1}} = [\partial_x]$. Let us now consider the polynomial 
\begin{align*}
\rho_{20}&:= (X-X_1)^2(X-X_2)^2\ldots \cdot (X-X_{\delta})^2 X^2 Z^{d-2\delta-2}
\end{align*}
and consider the corresponding curve $\gamma_{20}(t)$. We now note 
\begin{align*}
\{\{\nabla \Psi_{\PP A_1}|_{([f],[\eta], q_1, \ldots, q_{\delta}, l_{q_{\delta+1}})}\}(\gamma_{20}^{\prime}(0))\}(f \otimes 
\partial_x \otimes \partial_x)&= \lambda Z^{d-2\delta-2} \nabla^2 X|_{[0,0,1,0]}(\partial_x, \partial_x).
\end{align*}
Since $\nabla^2 X|_{[0,0,1,0]}(\partial_x, \partial_x)$ is nonzero, we conclude that the section is transverse to zero. \\ 
\hf \hf Next, let us justify the closure claims. Let us start with \cref{two_nodes_Hat_A3}. This statement is saying that 
when two nodes collide, we get a tacnode. Hence, the proof of \cref{two_nodes_Hat_A3} is same as the proof of \cref{a1_a1_is_a3}. \\ 
\hf \hf Next, let us consider \cref{three_nodes_Hat_D4}. Again, this statement is saying what happens what happens when three nodes 
collide. Hence, the proof of \cref{three_nodes_Hat_D4} is same as the proof of \cref{three_node_triple_point}. \\ 
\hf \hf It remains to justify the contribution from the points of $\overline{\A_1^{\delta-2} \circ \widehat{D}}_4 \cap \mu$. 
We will use the solutions constructed in \cref{d4_three_nodes_soln}. Using the expression for $f_{t_{20}}$, we note that the multiplicity 
from each branch is the number of small solutions $u$ to the equation 
\begin{align*}
-\frac{f_{t_{30}}}{3} (A_1 - A_3) u + E_6(u) & = \varepsilon.  
\end{align*}
This is clearly $1$. Since there are $6$ branches, the total multiplicity is $6$. \qed

\subsection{Proof of \Cref{npa2}: computation of $N(A_1^{\delta} \mathcal{P} A_2)$ when $0\leq \delta \leq 2$} 
We will justify our formula for 
$N(A_1^{\delta} \mathcal{P} A_2,r, s, n_1, n_2, n_3, \theta)$, when $0 \leq \delta \leq 2$.
Recall that 
\begin{align*}
\A^{\delta}_1 \circ \overline{\PP \A}_1 := \{ ([f], [\eta], q_1, \ldots, q_{\delta}, 
l_{q_{\delta+1}}) \in \mathcal{S}_{\mathcal{D}_{\delta}}\times_{\mathcal{D}} \mathbb{P} W_{\mathcal{D}}: 
&\textnormal{$f$ has a singularity of type $\A_1$ at $q_1, \ldots, q_{\delta}$}, \\ 
                   & ([f], [\eta], l_{q_{\delta+1}}) \in \overline{\PP \A}_1, ~~\textnormal{$q_1, \ldots, q_{\delta+1}$ all distinct}\}. 
\end{align*}
Let $\mu$ be a generic cycle, representing 
the class 
\begin{align*}
[\mu] = \mathcal{H}_L^r \cdot \mathcal{H}_p^s \cdot a^{n_1} \lambda^{n_2} (\pi_{\delta+1}^*H)^{n_3} (\pi_{\delta+1}^*\lambda_{W})^{\theta}. 
\end{align*}
Recall that as per the hypothesis of the Theorem, if $\delta=2$ then $\theta=0$. 
We now define a section of the following line bundle 
\begin{align}
\Psi_{\PP \A_2}:  \A^{\delta}_1 \circ \overline{\PP \A}_1 
\lra \mathbb{L}_{\PP \A_2} & := \gamma_{\DD}^\ast\otimes \gamma_W ^\ast \otimes (W/\gamma_{W})^\ast\otimes \gamma_{\mathbb{P}^3}^{* d}, 
\qquad \textnormal{given by} \nonumber \\
\{\Psi_{\PP \A_2}([f], q_1,\ldots,q_\delta,l_{q_{\delta+1}})\}(f\otimes v \otimes w) &:= 
\nabla^2 f|_{q_{\delta+1}}(v,w). \nonumber
\end{align}
We will show shortly that this section is transverse to zero. 
Next, let us define 
\begin{align*}
\mathcal{B} &:= \overline{\A^{\delta}_1 \circ \overline{\PP \A}}_1- \A^{\delta}_1 \circ \overline{\PP\A}_1. 
\end{align*}
Hence 
\begin{align}
\lan e(\mathbb{L}_{\PP \A_2}), 
~~[\overline{\A^{\delta}_1 \circ \overline{\PP\A}}_1] \cap [\mu] \ran & = \N(\A_1^{\delta}\PP \A_2, r, s, n_1, n_2, n_3, \theta)  
+ \mathcal{C}_{\mathcal{B}\cap \mu}.
\label{npa2_Euler_class_formula}
\end{align}
Define $\mathcal{B}(q_{i_1}, \ldots q_{i_k}, l_{q_{\delta+1}})$ as before. 
For simplicity, let us set $(i_1, i_2, \ldots, i_k):= (\delta-k, \ldots, \delta-1, \delta)$. 
Before we 
describe $\B (q_{i_1},q_{i_2}, l_{q_{\delta +1 }})$, let us define a few things. 
Let $v$ be a fixed nonzero vector that belongs to $l_{q_{\delta+1}}$. 
Let us define $W_1, W_2$, $W_3$ and $W_4$ as 
\begin{align}
W_1 &:=  \{ ([f], [\eta], q_1, \ldots,q_{\delta}, l_{q_{\delta+1}}) \in 
\overline{\A^{\delta}_1 \circ \overline{\PP\A}}_1: \nabla^2f|_{q_{\delta+1}} \not\equiv 0\}, \nonumber \\ 
W_2 &:=  \{ ([f], [\eta], q_1, \ldots,q_{\delta}, l_{q_{\delta+1}}) \in \overline{\A^{\delta}_1 \circ \overline{\PP\A}}_1: 
\nabla^2f|_{q_{\delta+1}} \equiv  0 \}, \nonumber  \\  
W_3 &:=  \{ ([f], [\eta], q_1, \ldots,q_{\delta}, l_{q_{\delta+1}}) \in \overline{\A^{\delta}_1 \circ \overline{\PP\A}}_1: 
\nabla^3f|_{q_{\delta+1}}(v,v,v) \neq  0 \} \qquad \textnormal{and} \nonumber \\ 
W_4 &:=  \{ ([f], [\eta], q_1, \ldots,q_{\delta}, l_{q_{\delta+1}}) \in \overline{\A^{\delta}_1 \circ \overline{\PP\A}}_1: 
\nabla^3f|_{q_{\delta+1}}(v,v,v) = 0 \}.
\end{align}

We claim that 
\begin{align}
 \B (q_{\delta}, l_{q_{\delta +1 }}) \cap W_1 & \approx \overline{\A_1^{\delta-1} \circ \PP \A}_3 \cap W_1, \label{hh1} \\ 
\B (q_{\delta}, l_{q_{\delta +1 }}) \cap W_2 & \approx \overline{\A_1^{\delta-1} \circ \widehat{\D}}_4, \label{hh2} \\ 
 \B (q_{\delta-1},  q_{\delta}, l_{q_{\delta +1 }}) \cap W_1 & \subset \overline{\A_1^{\delta-2} \circ \PP \A}_5 \cap W_1, \label{hh3} \\ 
\B (q_{\delta-1},  q_{\delta}, l_{q_{\delta +1 }}) \cap (W_2 \cap W_4)  & \approx \overline{\A_1^{\delta-2} \circ \PP \D}_4 \cap W_1
\qquad \textnormal{and} \label{hh4} \\
\B (q_{\delta-1},  q_{\delta}, l_{q_{\delta +1 }}) \cap (W_2 \cap W_3)  & \subset \overline{\A_1^{\delta-2} \circ \widehat{D}}_5. \label{hh5}
\end{align}
Notice that equations \eqref{hh3} and \eqref{hh5} say that 
the left hand side is a subset of the right hand side (unlike the other three
equations, which assert equality of sets). We now note that equations \eqref{hh1} and \eqref{hh2}, imply that 
\begin{align}
 \B (q_{i_1}, l_{q_{\delta +1 }}) & \approx \overline{\A_1^{\delta-1} \circ \PP \A}_3\cup 
\overline{\A_1^{\delta-1} \circ \widehat{\D}}_4, \label{a1_Pa1_clsr} 
\end{align}
while equations \eqref{hh3}, \eqref{hh4} and \eqref{hh5} imply that 
\begin{align}
\B (q_{i_1},q_{i_2}, l_{q_{\delta +1 }}) & \subset \overline{\A_1^{\delta-2} \circ \PP \A}_5\cup 
\overline{\A_1^{\delta-2} \circ \PP D}_4 \cup \overline{\A_1^{\delta-2} \circ \widehat{\D}}_5. \label{a1_a1_Pa1_clsr}
\end{align}

We claim that the contribution to the Euler class from each of the points of  $\overline{\A_1^{\delta-1} \circ \PP \A}_3 \cap \mu$, $\overline{\A_1^{\delta-1} \circ \widehat{\D}}_4 \cap \mu$ and $\overline{\A_1^{\delta-2} \circ \PP \D}_4 \cap \mu$ 
are 
$2,3$ and $4$ respectively. \\ 
\hf \hf Next, we note that for dimensional reasons, the intersection of $\overline{\A_1^{\delta-1} \circ \PP \A}_5$ 
with $\mu$ is empty. Hence, by \cref{hh3}, the intersection of $\B (q_{\delta-1},  q_{\delta}, l_{q_{\delta +1 }}) \cap W_1$ 
with $\mu$ is also empty and 
hence does not contribute to the Euler class. 
Finally, let us consider the component corresponding to the 
left hand side of \cref{hh5}; this is where we will use $\theta=0$. Let us consider the projection map 
\begin{align*}
\pi:\mathcal{S}_{\mathcal{D}_{\delta}} \times_{\widehat{\mathbb{P}}^3} \mathbb{P}W_D \longrightarrow \mathcal{S}_{\mathcal{D}_{\delta+1}}.
\end{align*}
We recall that 
\begin{align*}
\overline{\A_1^{\delta-2} \circ \widehat{D}}_5 & = \pi^{-1}(\overline{\A_1^{\delta-2} \circ D}_5). 
\end{align*}
Since $\theta=0$, we note that $\mu$ is the pullback of a class $\nu$, i.e. 
\begin{align*}
\mu &= \pi^*(\nu). 
\end{align*}
Hence, the intersection of $\mu$ with $\overline{\A_1^{\delta-2} \circ \widehat{D}}_5$ is in one to one correspondence with the 
intersection of $\nu$ with $\overline{\A_1^{\delta-2} \circ D}_5$. But the degree of the cohomology class $\nu$ is one more than the 
dimension of the cycle $\overline{\A_1^{\delta-2} \circ D}_5$. Hence, the intersection of $\overline{\A_1^{\delta-2} \circ D}_5$ 
with $\nu$ is empty and hence, the intersection of $\mu$ with $\overline{\A_1^{\delta-2} \circ \widehat{D}}_5$ is empty. 
As a result, by \cref{hh5}, the intersection of 
$\B (q_{\delta-1},  q_{\delta}, l_{q_{\delta +1 }}) \cap (W_2 \cap W_3)$ with $\mu$ is also empty. 
Hence the total contribution from all 
the components of type $\B (q_{i_1}, l_{q_{\delta +1 }})$ equals
\begin{align*}
2\binom{\delta}{1}N(\A_1^{\delta-1}\PP\A_3, r, s, n_1, n_2, n_3, \theta)
+3\binom{\delta}{1}N(\A_1^{\delta-1}\widehat{\D}_4, r, s, n_1, n_2, n_3, \theta),
\end{align*}
while the total contribution from all 
the components of type $\B (q_{i_1},q_{i_2}, l_{q_{\delta +1 }})$ equals
\begin{align*}
4\binom{\delta}{2}N(\A_1^{\delta-2}\PP\D_4, r, s, n_1, n_2, n_3, \theta).
\end{align*} 
Plugging this in \cref{npa2_Euler_class_formula} 
gives us the formula of \cref{npa2}. \\ 
\hf \hf Let us now prove the claim about transversality. This follows from following the setup of proof of transversality in 
Theorem  \cref{npa1}. We consider the polynomial 
\begin{align*}
\rho_{11}&:= (X-X_1)^2(X-X_2)^2\ldots \cdot (X-X_{\delta})^2 XY Z^{d-2\delta-2}
\end{align*}
and the corresponding curve $\gamma_{11}(t)$. Transversality follows by computing the derivative of the section $\Psi_{\PP A_2}$ 
along the curve $\gamma_{11}(t)$ as before.\\ 
\hf \hf Next, let us justify the closure and multiplicity claims. We will start by justifying \cref{a1_Pa1_clsr}. 
It suffices to justify \cref{hh1} and \cref{hh2}. Let us rewrite these two equations explicitly, namely  
\begin{align}
\{([f], [\eta], q_1, \ldots,q_{\delta}, l_{q_{\delta+1}}) \in \overline{A_1^{\delta}\circ \PP A}_1: 
q_{\delta} = q_{\delta+1}\} \cap W_1 & = \overline{A_1^{\delta-1}\circ \PP A}_3 \cap W_1 \qquad \textnormal{and} \label{eq1_w1}\\ 
\{([f], [\eta], q_1, \ldots,q_{\delta}, l_{q_{\delta+1}}) \in \overline{A_1^{\delta}\circ \PP A}_1: 
q_{\delta} = q_{\delta+1}\} \cap W_2 & = \overline{A_1^{\delta-1}\circ \widehat{D}}_4.  \label{eq2_w2}
\end{align}
Since $\widehat{D}_4$ is a subset of $W_2$, we did not write $\cap W_2$ on the right hand side of \cref{eq2_w2}.\\ 
\hf \hf Let us now start the proof of \cref{eq1_w1}. Let us first explain why the left hand side of \cref{eq1_w1} 
is a subset of its right hand side. To see that, first we note that $\PP A_1$ is a subset of  $\overline{\hat{A}}_1$. 
Since we have shown while proving \cref{a1_a1_is_a3} and \cref{a1_a1_a3_exp} 
that when two nodes collide we get a tacnode in \cref{a1_a1_a3_exp}, we conclude that 
\begin{align*}
\{([f], [\eta], q_1, \ldots,q_{\delta}, l_{q_{\delta+1}}) \in \overline{A_1^{\delta}\circ \hat{A}}_1: 
q_{\delta} = q_{\delta+1}\} & = \overline{A_1^{\delta-1}\circ \hat{A}}_3. 
\end{align*}
Hence, we conclude that 
\begin{align}
\{([f], [\eta], q_1, \ldots,q_{\delta}, l_{q_{\delta+1}}) \in \overline{A_1^{\delta}\circ \PP A}_1: 
q_{\delta} = q_{\delta+1}\} & \subset \overline{A_1^{\delta-1}\circ \hat{A}}_3 \nonumber \\ 
\implies 
\{([f], [\eta], q_1, \ldots,q_{\delta}, l_{q_{\delta+1}}) \in \overline{A_1^{\delta}\circ \PP A}_1: 
q_{\delta} = q_{\delta+1}\} \cap W_1 & \subset \overline{A_1^{\delta-1}\circ \hat{A}}_3 \cap W_1. \label{a1_PA1_hat_A3}
\end{align}
Suppose 
$([f], [\eta], q_1, \ldots,q_{\delta}, l_{q_{\delta+1}})$ belongs to the left hand side of \cref{a1_PA1_hat_A3}. Since 
$([f], [\eta], l_{q_{\delta+1}})$ belongs to $\overline{\PP A}_1$, we conclude that 
\begin{align*}
\nabla^2f|_{q_{\delta+1}}(v,v) & = 0 \qquad \forall ~~v \in l_{q_{\delta+1}}. 
\end{align*}
Since $([f], [\eta], q_1, \ldots,q_{\delta}, l_{q_{\delta+1}})$ is a subset of the right hand side of 
\cref{a1_PA1_hat_A3}, we conclude that the Hessian $\nabla^2f|_{q_{\delta+1}}$ is not identically zero, but it has a non trivial 
Kernel. We claim that $v$ is in the Kernel of the Hessian. To see why, 
let us assume that the nonzero vector $\tilde{v}$ is in the Kernel of the Hessian, i.e. 
$\nabla^2f|_{q_{\delta+1}}(\tilde{v}, \cdot)=0$. Let $w$ be any other vector, linearly independent from $\tilde{v}$. Since the Hessian 
is not identically zero and the vector space is two dimensional, 
we conclude that $\nabla^2f|_{q_{\delta+1}}(w, w) \neq 0$. Hence, writing the vector $v:= \lambda_1 \tilde{v} + \lambda_2 w$ 
and using $\nabla^2f|_{q_{\delta+1}}(v,v)=0$, we conclude that $\lambda_2 =0$. Hence, $v$ belongs to the Kernel of the Hessian. 
But we also note that if $([f], [\eta], l_q) \in \PP A_3$ and $\nabla^2 f|_q(v, \cdot)=0$, then $\nabla^3 f|_q(v,v,v) =0$. Hence, 
we can improve \cref{a1_PA1_hat_A3} and conclude that 
the left hand side of \cref{eq1_w1} is a subset of its right hand side.\\ 
\hf \hf Let us now prove the converse. We will simultaneously prove the following two statements 
\begin{align}
\{([f], [\eta], q_1, \ldots,q_{\delta}, l_{q_{\delta+1}}) \in \overline{A_1^{\delta}\circ \PP A}_1: 
q_{\delta} = q_{\delta+1}\}  & \supset A_1^{\delta-1}\circ \PP A_3, \label{k1} \\ 
\Big(\{([f], [\eta], q_1, \ldots,q_{\delta}, l_{q_{\delta+1}}) \in \overline{A_1^{\delta}\circ \PP A}_2: 
q_{\delta} = q_{\delta+1}\}\Big) \cap \Big(A_1^{\delta-1}\circ \PP A_3\Big) & = \emptyset \qquad \textnormal{and} \label{k2}  
\end{align}
We will prove the following claim: 
\begin{claim}
\label{one_node_andone_PA1_claim}
Let $([f], [\eta], q_1, \ldots, \ldots, q_{\delta-1}, l_{q_{\delta}}) \in A_1^{\delta-1}\circ \PP A_3$. Then there exists points 
\begin{align}
\big([f_t], [\eta_t], q_1(t), \ldots, q_{\delta-2}(t); q_{\delta-1}(t), q_{\delta}(t), l_{q_{\delta+1}(t)}\big) \in 
A_1^{\delta}\circ \PP A_1 \label{sol_const}
\end{align}
sufficiently close to $([f], [\eta], q_1, \ldots, \ldots, q_{\delta-1}; q_{\delta}, l_{q_{\delta}})$.
Furthermore, \textit{every} such solution 
satisfies the condition 
\begin{align}
\nabla^2f|_{q_{\delta+1}}(v,w) &\neq 0, \label{one_node_one_PA2_int_PA3_empty}
\end{align}
if $v$ is a nonzero vector that belongs to  $l_{q_{\delta+1}(t)}$ and $w$ is a nonzero vector that belongs to 
$T\mathbb{P}^2_{\eta}|_{q_{\delta+1}(t)}/l_{q_{\delta+1}(t)}$.  In other words, 
\begin{align*}
\big([f_t], [\eta_t], q_1(t), \ldots, q_{\delta-1}(t), q_{\delta}(t), l_{q_{\delta+1}(t)}\big) & \not\in A_1^{\delta} \circ \PP A_2. 
\end{align*}
\end{claim}
\begin{rem}
We note that \cref{one_node_andone_PA1_claim} simultaneously proves \cref{k1} and \cref{k2}. 
\end{rem}

\noindent \textbf{Proof:} 
Following the setup of the proofs of claims \ref{two_and_three_node_claim} and 
\ref{three_node_claim_d4}, we  
will now work in an affine chart, where we send the plane 
$\mathbb{P}^2_{\eta_t}$ to $\mathbb{C}^2_{z}$ and the 
point $q_{\delta}(t) \in \mathbb{P}^2_{\eta_t}$ to 
$(0,0,0) \in \mathbb{C}^2_{z}$. We also choose coordinates, such that $ \partial_x \in l_{q_{\delta+1}(t)}$. 
Using this  chart, 
let us write down the Taylor expansion of $f_t$ around the point $(0,0)$, namely  
\begin{align*}
f_t(x,y)&= f_{t_{11}}xy + \frac{f_{t_{02}}}{2} y^2 + \frac{f_{t_{30}}}{6} x^3 + \frac{f_{t_{21}}}{2} x^2y + 
\frac{f_{t_{12}}}{2} x y^2 + \frac{f_{t_{03}}}{6} y^3+\ldots 
\end{align*}
Since $([f_t], [\eta_t], l_{q_{\delta}(t)}) \in \PP A_1$, we conclude that $f_{t_{20}}$ is zero. Next, let us consider the 
Taylor expansion of $f$ (not $f_t$). We note that $([f], [\eta], l_{q_{\delta}}) \in \PP A_3$. This means that 
$f_{11}$ and $f_{02}$ can not both be zero (since that would mean the Hessian is identically zero). If $f_{02} =0$ and 
$f_{11} \neq 0$, then it implies that $([f], [\eta], l_{q_{\delta}}) \in \hat{A}_1$ (and hence does not belong to $\PP A_3$). 
Hence, $f_{02} \neq 0$ and hence we conclude that $f_{t_{02}} \neq 0$. Finally, since 
$([f], [\eta], l_{q_{\delta}}) \in \PP A_3$, we conclude that $f_{{11}}$ and $f_{{30}}$ are zero; hence 
$f_{t_{11}}$ and $f_{t_{30}}$ are small (close to zero). We will mainly follow the Proof of \cref{two_and_three_node_claim}.  
Since $f_{t_{02}} \neq 0$ we can make the same change of coordinates  $ \hat{y}:= y + B(x)$ as 
in the Proof of \cref{two_and_three_node_claim} and write $f_t$ as 
\begin{align*}
f_t (x, y(x, \hat{y})) &=  
\varphi(x,\hat{y})\hat{y}^2 + \frac{\mathcal{B}_2^{f_t}}{2!} x^2 + 
\frac{\mathcal{B}_3^{f_t}}{3!} x^3 + \frac{\mathcal{B}_4^{f_t}}{4!} x^4 + \mathcal{R}(x)x^5,  
\end{align*}
where 
\begin{align}
\B^{f_t}_2 & := - \frac{f_{t_{11}}^2}{f_{t_{02}}}, \qquad \B^{f_t}_3 := f_{t_{30}}
-\frac{3f_{t_{11}} f_{t_{21}}}{f_{t_{02}}}+\frac{3f_{t_{11}}^2 f_{t_{12}}}{f_{t_{02}}^2}
-\frac{3f_{t_{11}}^3 f_{t_{03}}}{f_{t_{02}}^3}, \ldots,  
\qquad   \varphi(0,0) \neq 0 \label{Bf_k_formula}
\end{align}
and $\mathcal{R}(x)$ is a holomorphic function defined in a neighborhood of the origin. 
Since $([f], [\eta], q_{\delta}) \in \PP A_3$, 
we conclude that $\B^{f_t}_2$ and $\B^{f_t}_3$ are small (close to zero) and $\B^{f_t}_4$ is nonzero.  
Let us make a further change of coordinates and denote 
\begin{align*}
\hat{\hat{y}}&:= \sqrt{\varphi(x, \hat{y})} \hat{y}.
\end{align*}
Note that we can choose a branch of the square root since $\varphi(0,0) \neq 0$. 
Next, for notational convenience, let us now define 
\begin{align}
\hat{f}_t(x, \hat{\hat{y}})  &:= f_t (x, y(x, \hat{y}(\hat{\hat{y}})))), 
\end{align}
i.e. $\hat{f}_t$ is basically $f_t$ written in the new coordinates (namely $x$ and $\hat{\hat{y}}$). Hence, 
\begin{align*}
\hat{f}_t(x, \hat{\hat{y}}) & = \hat{\hat{y}}^2 + \frac{\mathcal{B}_2^{f_t}}{2!} x^2 + 
\frac{\mathcal{B}_3^{f_t}}{3!} x^3 + \frac{\mathcal{B}_4^{f_t}}{4!} x^4 + \mathcal{R}(x)x^5.
\end{align*}
We now note that constructing the points 
on the left hand side of \cref{sol_const} amounts to solving the set of equations 
\begin{align}
\hat{f}_t & = 0, \qquad \hat{f}_{t_x} =0 \qquad \textnormal{and} \qquad \hat{f}_{t_{\hat{\hat{y}}}} =0, \label{eq_k4}
\end{align}
where $(x, \hat{\hat{y}})$ is small but not equal to $(0,0)$.\\ 
\hf \hf We will now construct solutions to \cref{eq_k4}. The solutions to \cref{eq_k4} are given by 
\begin{align}
\hat{\hat{y}} =0, \qquad \mathcal{B}^{f_t}_2 & = \frac{\mathcal{B}_4^{f_t}}{12} x^2 + O(|x|^3) \qquad \textnormal{and} \qquad 
\mathcal{B}^{f_t}_3 = -\frac{\mathcal{B}_4^{f_t}}{2} x + O(|x|^2). \label{f11+f30_soln}
\end{align}
Now we use the expression of $\mathcal{B}^{f_t}_2,\mathcal{B}^{f_t}_3$ and conclude from \cref{f11+f30_soln} that 
\begin{align}
\frac{f_{t_{11}}^2}{f_{t_{02}}}& = -\frac{\mathcal{B}_4^{f_t}}{12} x^2 + O(|x|^3) \qquad \textnormal{and} \label{f11} \\
f_{t_{30}} & = -3\mathcal{B}_4^{f_t} x+O(|x|^2). \label{f30_soln}
\end{align}
Hence, there are 
two solutions to \cref{f11}, given by 
\begin{align}
f_{t_{11}} & = \Big(\sqrt{\frac{-f_{t_{02}}\mathcal{B}^{f_t}_4}{12}}\Big) x + O(|x|^2) \qquad \textnormal{or} \qquad  
f_{t_{11}} = -\Big(\sqrt{\frac{-f_{t_{02}}\mathcal{B}^{f_t}_4}{12}}\Big) x + O(|x|^2), \label{f11_mult}
\end{align}
where $\sqrt{}$ denotes a branch of the square root. 
Hence, there are \textbf{exactly} two solutions to 
\cref{eq_k4}, given by 
\begin{align}
x&= u, \qquad f_{t_{11}} = \pm \Big(\sqrt{\frac{-f_{t_{02}}\mathcal{B}^{f_t}_4}{12}}\Big) u + O(|u|^2) \label{x_and_f11}
\end{align}
and $\hat{\hat{y}}=0$ and $f_{t_{30}}$ as given by 
\cref{f30_soln}, where we plug in the expressions for 
$x$ and $f_{t_{11}}$ as given by \cref{x_and_f11} to express them in terms of $u$ 
(the exact expressions in terms of $u$ are not so important, hence we have not written that out explicitly). This proves 
\cref{one_node_andone_PA1_claim}. Since \cref{x_and_f11} are the \textbf{only} solutions and $\mathcal{B}^{f_t}_4 \neq 0$, 
we also conclude that \cref{one_node_one_PA2_int_PA3_empty} is true. \qed \\
\hf \hf It remains to compute the multiplicity. We claim the each point of $(A_1^{\delta-1}\circ \PP A_3)\cap \mu$ 
contributes $2$ to the Euler  class in \cref{npa2_Euler_class_formula}. 
Using \cref{x_and_f11} 
we conclude that the multiplicity 
from each branch 
is the number of small solutions $u$ to the equation 
\begin{align*}
\Big(\sqrt{\frac{-f_{t_{02}}\mathcal{B}^{f_t}_4}{12}}\Big) u + O(|u|^2)& = \varepsilon 
\qquad \textnormal{and} \qquad -\Big(\sqrt{\frac{-f_{t_{02}}\mathcal{B}^{f_t}_4}{12}}\Big) u + O(|u|^2) = \varepsilon.
\end{align*}
This number is $1$ in each case and hence, 
the total multiplicity is $2$. \qed \\ 
\hf \hf Next, let us justify \cref{eq2_w2}. Let us first explain why the left hand side of 
\cref{eq2_w2} is a subset of its right hand side. If 
$([f], [\eta], q_1, \ldots,q_{\delta}, l_{q_{\delta+1}}) \in W_2$, then it means that $\nabla^2f|_{q_{\delta+1}} = 0$. 
Hence, it means that $([f], [\eta], l_{q_{\delta+1}}) \in \overline{\widehat{D}}_4$. Hence, the left hand side of  
\cref{eq2_w2} is a subset of its right hand side.\\ 
\hf \hf Let us now prove \cref{eq2_w2}. Before that, let us introduce a new space. Let us define 
\begin{align*}
\widehat{D}_4^{\#}&:= \{([f], [\eta], l_{q}) \in \widehat{D}_4: \nabla^3f|_q(v,v,v) \neq 0 ~~\textnormal{if} ~~v \in l_q-0\}. 
\end{align*}
Note that $\overline{\widehat{D}^{\#}_4} = \overline{\widehat{D}}_4$. 
We will now simultaneously prove the following two statements: 
\begin{align}
\{([f], [\eta], q_1, \ldots,q_{\delta}, l_{q_{\delta+1}}) \in \overline{A_1^{\delta}\circ \PP A}_1: 
q_{\delta} = q_{\delta+1}\}  & \supset A_1^{\delta-1}\circ \widehat{D}_4^{\#}  \qquad \textnormal{and} \label{kk1} \\ 
\Big(\{([f], [\eta], q_1, \ldots,q_{\delta}, l_{q_{\delta+1}}) \in \overline{A_1^{\delta}\circ \PP A}_2: 
q_{\delta} = q_{\delta+1}\}\Big) \cap \Big(A_1^{\delta-1}\circ \widehat{D}_4^{\#} \Big) & = \emptyset. \label{kk2}
\end{align}
We will prove the following claim:
\begin{claim}
\label{one_node_and_one_PA1_claim_D4}
Let $([f], [\eta], q_1, \ldots, \ldots, q_{\delta-1}, l_{q_{\delta}}) \in A_1^{\delta-1}\circ \widehat{D}_4^{\#}$. Then there exists points 
\begin{align}
\big([f_t], [\eta_t], q_1(t), \ldots, q_{\delta-2}(t); q_{\delta-1}(t), q_{\delta}(t), l_{q_{\delta+1}(t)}\big) \in 
A_1^{\delta}\circ \PP A_1 \label{sol_const_2}
\end{align}
sufficiently close to $([f], [\eta], q_1, \ldots, \ldots, q_{\delta-1}; q_{\delta}, l_{q_{\delta}})$.
Furthermore, \textit{every} such solution 
satisfies the condition 
\begin{align}
\nabla^2f|_{q_{\delta+1}}(v,w) &\neq 0, \label{one_node_one_PA2_int_PA3_empty_2}
\end{align}
if $v$ is a nonzero vector that belongs to  $l_{q_{\delta+1}(t)}$ 
and $w$ is a nonzero vector that belongs to 
$T\mathbb{P}^2_{\eta}|_{q_{\delta+1}(t)}/l_{q_{\delta+1}(t)}$.  
In other words,
\begin{align*}
\big([f_t], [\eta_t], q_1(t), \ldots, q_{\delta-1}(t), q_{\delta}(t), l_{q_{\delta+1}(t)}\big) & \not\in A_1^{\delta} \circ \PP A_2. 
\end{align*}
\end{claim}
\begin{rem}
We note that \cref{one_node_and_one_PA1_claim_D4} proves \cref{kk1} and \cref{kk2} simultaneously 
(since $\overline{\widehat{D}^{\#}_4} = \overline{\widehat{D}}_4$). 
\end{rem}

\noindent \textbf{Proof:} Following the setup of the proofs of claims \ref{two_and_three_node_claim},
\ref{three_node_claim_d4} and \ref{one_node_andone_PA1_claim}, we  
will now work in an affine chart, where we send the plane 
$\mathbb{P}^2_{\eta_t}$ to $\mathbb{C}^2_{z}$ and the 
point $q_{\delta}(t) \in \mathbb{P}^2_{\eta_t}$ to 
$(0,0,0) \in \mathbb{C}^2_{z}$. 
We also choose coordinates, such that $ \partial_x \in l_{q_{\delta+1}(t)}$. 
Using this  chart, 
let us write down the Taylor expansion of $f_t$ around the point $(0,0)$, namely  
\begin{align*}
f_t(x,y)&= f_{t_{11}}xy + \frac{f_{t_{02}}}{2} y^2 + \frac{f_{t_{30}}}{6} x^3 + \frac{f_{t_{21}}}{2} x^2y + 
\frac{f_{t_{12}}}{2} x y^2 + \frac{f_{t_{03}}}{6} y^3+\ldots 
\end{align*}
Since $([f_t], [\eta_t], l_{q_{\delta}(t)}) \in \PP A_1$, we conclude that $f_{t_{20}}$ is zero. Next, since 
$([f], [\eta], l_{q_{\delta}}) \in \widehat{D}_4$, we conclude that $f_{20}, \, f_{11}$ and $f_{02}$ are zero; hence 
$f_{t_{11}}$ and $f_{t_{02}}$ are small (close to zero). Hence, constructing points on the right hand side of 
\cref{sol_const_2} amounts to finding solutions to the set of equations 
\begin{align}
f_t & = 0, \qquad f_{t_x} =0 \qquad \textnormal{and} \qquad f_{t_{y}} =0, \label{eq_s1}
\end{align}
where $(x,y)$ is small but not equal to $(0,0)$. Let us define 
\begin{align*}
g_t(x,y) & = -2f_t(x,y)+x f_{t_x}(x,y) +y f_{t_y}(x,y). 
\end{align*}
We note that $f_t(x,y)$ and $g_t(x,y)$ have the same cubic term in the Taylor expansion. Furthermore, $g_t(x,y)$ does not contain any 
quadratic term. Since $([f], [\eta], l_{q_{\delta}}) \in \widehat{D}_4$, we conclude that 
$f_{t_{30}} \neq 0$. 
Let 
\begin{align*}
x&:= \hat{x}  + E_1(\hat{x}, \hat{y}) \qquad \textnormal{and} \qquad y:= \hat{y}  + E_2(\hat{x}, \hat{y}) 
\end{align*}
be changes change of coordinates (where $E_1$ and $E_2$ are second order terms), such that 
\begin{align*}
g_t &= \frac{f_{t_{30}}}{6}(\hat{x}- A_1 \hat{y})(\hat{x}- A_2 \hat{y}) (\hat{x}- A_3 \hat{y}) 
\end{align*}
There are three solutions to $g_t = 0$, given by $\hat{y}=u$ and $\hat{x}= A_i \hat{u}$, for $i=1,2$ and $3$. 
Converting back in terms of $x$ and $y$, we conclude that the solutions to $g_t =0$ are given by 
\begin{align*}
y& = u \qquad \textnormal{and} \qquad x = A_i u + O(|u|^2). 
\end{align*}
Let us consider the solution $x = A_1 u + O(|u|^2)$; the other two cases can be dealt with similarly. We plug this 
solution into the equations $f_{t_x} =0$ and $f_{t_y} =0$ and solve for $f_{t_{11}}$ and $f_{t_{02}}$ in terms of $u$. 
Doing that, we get the solutions to \cref{eq_s1} are given by  
\begin{align}
y &= u, \qquad x = A_1 u + O(|u|^2), \nonumber \\ 
f_{t_{11}}  & = -\frac{f_{t_{30}}}{6}(A_1-A_2)(A_1-A_3) u + O(|u|^2) \qquad \textnormal{and} \qquad 
f_{t_{02}} = \frac{f_{t_{30}}}{3} A_1 (A_1-A_2) u + O(|u|^2) \label{f11_sol_8}
\end{align}
and two more similar solutions corresponding to $x = A_2 u + O(|u|^2)$ and $x = A_3 u + O(|u|^2)$. This proves 
the first assertion of \cref{one_node_and_one_PA1_claim_D4}. Furthermore, since $f_{t_{30}} \neq 0$ and $A_1, A_2$ and $A_3$ 
are distinct, we conclude using \cref{f11_sol_8}  that $f_{t_{11}} \neq 0$; this proves \cref{one_node_one_PA2_int_PA3_empty_2}. \qed \\
\hf \hf It remains to compute the multiplicity. 
We claim the each point of $(A_1^{\delta-1}\circ \widehat{D}_4^{\#})\cap \mu$ 
contributes $3$ to the Euler  class in \cref{npa2_Euler_class_formula}. 
Using \cref{f11_sol_8} 
we conclude that the multiplicity 
from each branch 
is the number of small solutions $u$ to the equation 
\begin{align*}
-\frac{f_{t_{30}}}{6}(A_1-A_2)(A_1-A_3) u + O(|u|^2) & = \varepsilon. 
\end{align*}
This number is $1$ and hence, 
the total multiplicity is $3$. Finally, we note that since $\mu$ is a generic cycle all points of 
$(A_1^{\delta-1}\circ \widehat{D}_4)\cap \mu$ will actually belong to $(A_1^{\delta-1}\circ \widehat{D}_4^{\#})\cap \mu$.\qed \\ 
\hf \hf Before proceeding further, note that we have proved 
\begin{align}
\Big(\{([f], [\eta], q_1, \ldots,q_{\delta}, l_{q_{\delta+1}}) \in \overline{A_1^{\delta}\circ \PP A}_1: 
q_{\delta-1}=q_{\delta} = q_{\delta+1}\}\Big) \cap \Big(A_1^{\delta-1}\circ \widehat{D}_4^{\#}\Big) & = \emptyset. \label{k4}
\end{align}
To see why that is so, our proof of the claim shows that the family we constructed can not have a third node. \\ 
\hf \hf Next, let us prove equations \eqref{hh3}, \eqref{hh4} and \eqref{hh5} 
(i.e. we  will analyze what happens when three points come together). Let us start with the proof of 
\eqref{hh3}. Let us show that 
\begin{align}
\Big(\{([f], [\eta], q_1, \ldots,q_{\delta}, l_{q_{\delta+1}}) \in \overline{A_1^{\delta}\circ \PP A}_1: 
q_{\delta-1}=q_{\delta} = q_{\delta+1}\}\Big) \cap \Big(A_1^{\delta-1}\circ \PP A_4\Big) & = \emptyset. \label{k6}
\end{align}
We note that \cref{k6} immediately implies \cref{hh3}. In order to prove \cref{k6}, it suffices to prove the following claim: 
\begin{claim}
\label{two_nodes_and_one_PA1_claim_not_PA4}
Let $([f], [\eta], q_1, \ldots, \ldots, q_{\delta-2}, l_{q_{\delta-1}}) \in A_1^{\delta-2}\circ \PP A_4$. Then there 
does not exist any point 
\begin{align}
\big([f_t], [\eta_t], q_1(t), \ldots, q_{\delta-2}(t); q_{\delta-1}(t), q_{\delta}(t), l_{q_{\delta+1}(t)}\big) \in 
A_1^{\delta}\circ \PP A_1 \label{sol_const_ag3}
\end{align}
sufficiently close to $([f], [\eta], q_1, \ldots, \ldots, q_{\delta-1}; q_{\delta}, l_{q_{\delta}})$.
\end{claim}

\noindent \textbf{Proof:} Let us continue with the setup of \cref{one_node_andone_PA1_claim}. 
As before, since $f_{t_{02}} \neq 0$, we can make a change of coordinates $\hat{y} :=y+ B(x)$ and write $f_t$ as 
\begin{align*}
f_t (x, y(x, \hat{y})) &=  
\varphi(x,\hat{y})\hat{y}^2 + \frac{\mathcal{B}_2^{f_t}}{2!} x^2 + 
\frac{\mathcal{B}_3^{f_t}}{3!} x^3 + \frac{\mathcal{B}_4^{f_t}}{4!} x^4 +\frac{\mathcal{B}_5^{f_t}}{5!} x^5+\frac{\mathcal{B}_6^{f_t}}{6!} x^6+ \mathcal{R}(x)x^7  
\end{align*}
where $\mathcal{B}^{f_t}_k$ are as defined in \cref{Bf_k_formula}, $\varphi(0,0) \neq 0$ and 
$\mathcal{R}(x)$ is a holomorphic function defined in a neighborhood of the origin. 
Let us make a further change of coordinates and denote 
\begin{align*}
\hat{\hat{y}}&:= \sqrt{\varphi(x, \hat{y})} \hat{y}.
\end{align*}
as in the Proof of \cref{one_node_andone_PA1_claim}. Let us denote the polynomial $f_t$ by $\hat{f}_t$ which is a polynomial in two variables $x$ and $\hat{\hat{y}}$. Hence, 
\begin{align*}
\hat{f}_t (x, \hat{\hat{y}}) &=  
\hat{\hat{y}}^2 + \frac{\mathcal{B}_2^{f_t}}{2!} x^2 + 
\frac{\mathcal{B}_3^{f_t}}{3!} x^3 + \frac{\mathcal{B}_4^{f_t}}{4!} x^4 +\frac{\mathcal{B}_5^{f_t}}{5!} x^5+\frac{\mathcal{B}_6^{f_t}}{6!} x^6+ \mathcal{R}(x)x^7.  
\end{align*} 
We claim that there does not exist any solutions to the set of equations 
\begin{align}
\hat{f}_t(u_1, v_1) & = 0, ~~\hat{f}_x(u_1, v_1) =0, ~~\hat{f}_{\hat{\hat{y}}}(u_1, v_1) =0 \qquad \textnormal{and} \label{b1}\\ 
\hat{f}_t(u_2, v_2) & = 0, ~~\hat{f}_x(u_2, v_2) =0, ~~\hat{f}_{\hat{\hat{y}}}(u_2, v_2) =0, \label{b2}
\end{align}
where $(u_1, v_1)$ and $(u_2, v_2)$ and $(0,0)$ are all distinct, but close to each other.\\
\hf \hf We now note that the only solutions to the set of equation \cref{b1} and \cref{b2} is given by 
 \begin{align}
v_1, v_2 & =0, \nonumber \\
\B^{f_t}_2  & = \frac{1}{360} \B^{f_t}_6 u_1^2 u_2^2 + O(|(u_1, u_2)|^5),  \quad 
\B^{f_t}_3 = -\frac{1}{60} \B^{f_t}_6(u_1^2 u_2 + u_1 u_2^2)+ O(|(u_1, u_2)|^4), \nonumber \\
\B^{f_t}_4 &= \frac{1}{30}\B^{f_t}_6(u_1^2+ 4 u_1 u_2 + u_2^2)+ O(|(u_1, u_2)|^3) \quad \textnormal{and} \quad 
\B^{f_t}_5 = -\frac{1}{3}\B^{f_t}_6 (u_1+u_2)+ O(|(u_1, u_2)|^2). \label{A6_nhbd_PA1}
\end{align}
To see why this is so, we simply note that \cref{b1} and \cref{b2} are the same as \cref{three_node_uv1} and \cref{three_node_uv2}; 
hence, the argument is exactly the same as how we justified 
\cref{A6_nhbd} is the solution to \cref{b1} and \cref{b2}. \\ 
\hf \hf We now note that $v_1,v_2$ are both zero; hence  $u_1$ and $u_2$ are both nonzero, but small. 
Hence, $\B^{f_t}_5$ is close to zero. This is a contradiction, since $([f], [\eta], l_{q_{\delta}}) \in \PP A_4$.

\hf \hf Next, let us prove \eqref{hh4}. We will prove the following claim:
\begin{claim}
\label{two_nodes_and_one_PA1_claim_PD4}
Let $([f], [\eta], q_1, \ldots, \ldots, q_{\delta-2}, l_{q_{\delta-1}}) \in A_1^{\delta-2}\circ \PP D_4$. Then there exists points 
\begin{align}
\big([f_t], [\eta_t], q_1(t), \ldots, q_{\delta-3}(t); q_{\delta-2}(t), q_{\delta-1}(t), q_{\delta}(t), l_{q_{\delta+1}(t)}\big) \in 
A_1^{\delta}\circ \PP A_1 \label{sol_const_2_ag}
\end{align}
sufficiently close to $([f], [\eta], q_1, \ldots, \ldots, q_{\delta-2};  q_{\delta-1}, q_{\delta-1}, l_{q_{\delta-1}})$.
Furthermore, \textit{every} such solution 
satisfies the condition 
\begin{align}
\nabla^2f|_{q_{\delta+1}}(v,w) &\neq 0, \label{one_node_one_PA2_int_PA3_empty_2_ag}
\end{align}
if $v$ is a nonzero vector that belongs to  $l_{q_{\delta+1}(t)}$ 
and $w$ is a nonzero vector that belongs to 
$T\mathbb{P}^2_{\eta}|_{q_{\delta+1}(t)}/l_{q_{\delta+1}(t)}$.  
In other words,
\begin{align*}
\big([f_t], [\eta_t], q_1(t), \ldots, q_{\delta-3}(t); q_{\delta-2}(t), q_{\delta-1}(t), q_{\delta}(t), l_{q_{\delta+1}(t)}\big) & \not\in A_1^{\delta} \circ \PP A_2. 
\end{align*}
\end{claim}

\noindent \textbf{Proof:} Following the setup of the proof of claim 
\ref{one_node_and_one_PA1_claim_D4}, 
let us write down the Taylor expansion of $f_t$ around the point $(0,0)$, namely  
\begin{align*}
f_t(x,y)&= f_{t_{11}}xy + \frac{f_{t_{02}}}{2} y^2 + \frac{f_{t_{30}}}{6} x^3 + \frac{f_{t_{21}}}{2} x^2y + 
\frac{f_{t_{12}}}{2} x y^2 + \frac{f_{t_{03}}}{6} y^3+\ldots 
\end{align*}
Since $([f_t], [\eta_t], l_{q_{\delta}(t)}) \in \PP A_1$, we conclude that $f_{t_{20}}$ is zero. Next, since 
$([f], [\eta], l_{q_{\delta}}) \in \PP D_4$, we conclude that $f_{11}$, $f_{02}$ and $f_{30}$ are zero; hence 
$f_{t_{11}}$, $f_{t_{02}}$ and $f_{t_{30}}$ are small (close to zero). Constructing points on the right hand side of 
\cref{sol_const_2_ag} amounts to finding solutions to the set of equations 
\begin{align}
f_t(x_1, y_1)&=0, \qquad f_{t_x}(x_1, y_1)=0, \qquad f_{t_{y}}(x_1, y_1) =0 \qquad \textnormal{and} \label{eq1_ag} \\ 
f_t(x_2, y_2)&=0, \qquad f_{t_x}(x_2, y_2)=0, \qquad f_{t_{y}}(x_2, y_2) =0, \label{eq2_ag}
\end{align}
where $(0,0), (x_1, y_1)$ and $(x_2, y_2)$ are all distinct (but close to each other). As before, we define 
\[g_t(x,y):= xf_{t_{x}}(x,y) + yf_{t_{y}}(x,y) - 2f_t(x,y).\] 
We note that $g_t$ has no quadratic term and has the same cubic term 
as $f_t$. The cubic term of $f$ can be written as either 
$\frac{f_{{03}}}{6}(y-A_1(0) x)(y-A_2(0) x)y$ (if $f_{{03}} \neq 0$) or it can be written as 
$\frac{xy}{2}(f_{21}x + f_{12} y)$ (if $f_{{03}} = 0$). We will assume the former case; the latter case can be dealt with 
similarly. Hence, we can write $g_t$ as 
\begin{align*}
g_t(x,y)& = \dfrac{f_{t_{03}}}{6}(y-A_1x)(y-A_2 x)(y-A_3x) + E(x,y), 
\end{align*}
where $E$ is a fourth order term. Let us assume that $A_3$ is close to zero. We also note that since $f_{t_{21}} \neq 0$, 
hence $A_1$ and $A_2$ are both nonzero. Using the equation $g_t = 0$, let us consider the solution 
\begin{align*}
x&= u \qquad  \textnormal{and} \qquad y = A_1 u + O(|u|^2).  
\end{align*}
Let us now use $f_{t_x}(x,y)=0$ and solve for $f_{t_{11}}$ in terms of $u$. Doing that, we get 
\begin{align*}
f_{t_{11}} & = \dfrac{f_{t_{03}}}{6}(A_1^2-A_1A_2-A_1A_3+ A_2A_3)u + O(|u|^2). 
\end{align*}
Plugging in this value of $f_{t_{11}}$ into the equation $f_{t_{y}}$ and solving for $f_{t_{02}}$, we get that 
\begin{align*}
f_{t_{02}} & = \dfrac{f_{t_{03}}}{6}\Big(-2 A_1 + 2 A_2 + 2 A_3 - \frac{2 A_2 A_3}{A_1}\Big) u + O(|u|^2). 
\end{align*}
Let us now try to produce a second node. We will justify shortly that $x=v$ and $y=A_2v + O(|v|^2)$ is a not a possible solution. 
Hence, let us consider $x=v$ and $y=A_3v + O(|v|^2)$. Plugging this into $f_{t_y}(x,y)=0$ and solving for $u$ in terms of $v$, 
we conclude that 
\begin{align*}
u & = \Big(\frac{A_1(A_3-A_2)}{(A_1-A_2)(A_1-2A_3)}\Big)v + O(|v|^2). 
\end{align*}
Plugging in this value for $u$ into $f_{t_x}(x,y)=0$ and solving for $A_3$, we conclude that 
\begin{align*}
A_3 & = O(|v|). 
\end{align*}
Plugging in the value of $A_3$ into $u$ and then plugging that back into $f_{t_{11}}$ and $f_{t_{02}}$, we conclude that 
\begin{align*}
u& = \frac{A_2}{A_2-A_1} v + O(|v|^2), \qquad  
f_{t_{11}} = -\dfrac{f_{t_{03}}}{6}A_1A_2 v + O(|v|^2) \qquad \textnormal{and} \qquad  
f_{t_{02}} = \dfrac{f_{t_{03}}}{3}A_2 v + O(|v|^2).
\end{align*}
There are four ways to construct such solutions (interchange $(A_1, A_3)$, with $(A_2, A_3)$). Furthermore, we can permute the nodal points. 
From the expression for $f_{t_{11}}$ we see that the order of vanishing is $1$; hence the total multiplicity is $4$. \\ 
\hf \hf It remains to show why we reject the solution $x=v$ and $y = A_2 v+ O(|v|^2)$. If we take that solution, then we plug it in 
$f_{t_x}=0$, then solving for $u$ (in terms of $v$), we conclude that 
\begin{align*}
u & = \Big(\frac{A_3-A_2}{A_1-A_3}\Big)v + O(|v|^2) 
\end{align*}
Plugging this into $f_{t_{y}}$, we conclude that 
\begin{align*}
f_{t_y} & = \dfrac{f_{t_{03}}}{3} \Big(  \frac{(A_1-A_2)^2(A_3-A_2)}{A_1}\Big) v^2 + O(|v|^2). 
\end{align*}
This is clearly nonzero, if $v$ is small and nonzero. Hence, we reject the solution corresponding to $x=v$ and $y = A_2 v + O(|v|^2)$. 
This completes the proof. \\ 
\hf \hf Finally, let us justify \cref{hh5}. This follows from \cref{k4}. This completes the proof of \Cref{npa2}. \qed


\subsection{Proof of \Cref{npa3}: computation of $N(A_1^{\delta} \mathcal{P} A_3)$ when $0\leq \delta \leq 1$} 
We will justify our formula for 
$N(A_1^{\delta} \mathcal{P} A_3,r, s, n_1, n_2, n_3, \theta)$, when $0 \leq \delta \leq 1$.
Recall that 
\begin{align*}
\A^{\delta}_1 \circ \overline{\PP \A}_2 := \{ ([f], [\eta], q_1, \ldots, q_{\delta}, 
l_{q_{\delta+1}}) \in \mathcal{S}_{\mathcal{D}_{\delta}}\times_{\mathcal{D}} \mathbb{P} W_{\mathcal{D}}: 
&\textnormal{$f$ has a singularity of type $\A_1$ at $q_1, \ldots, q_{\delta}$}, \\ 
                   & ([f], [\eta], l_{q_{\delta+1}}) \in \overline{\PP \A}_2, ~~\textnormal{$q_1, \ldots, q_{\delta+1}$ all distinct}\}. 
\end{align*}
Let $\mu$ be a generic cycle, representing 
the class 
\begin{align*}
[\mu] = \mathcal{H}_L^r \cdot \mathcal{H}_p^s \cdot a^{n_1} \lambda^{n_2} (\pi_{\delta+1}^*H)^{n_3} (\pi_{\delta+1}^*\lambda_{W})^{\theta}. 
\end{align*}
We now define a section of the following bundle 
\begin{align}
\Psi_{\PP \A_3}:  \A^{\delta}_1 \circ \overline{\PP \A}_2 
\lra \mathbb{L}_{\PP \A_3} & := \gamma_{\DD}^\ast\otimes\gamma_{W}^{\ast 3}\otimes \gamma_{\mathbb{P}^3}^{* d}, \qquad \textnormal{given by} 
\nonumber \\
\{\Psi_{\PP \A_3}([f], [\eta], q_1,\ldots,q_\delta,l_{q_{\delta+1}})\}(f\otimes v^{\otimes 3}) &:= 
\nabla^3f|_{q_{\delta+1}}(v,v,v).  \nonumber
\end{align}
Analogous to \cite[Lemma 6.1]{BM13_2pt_published}, 
we conclude that for $d \geq 4$, 
\begin{align}
\overline{\PP \A}_2 &= \PP \A_2 \cup \overline{\PP \A}_3 \cup \overline{\widehat{D}}_4. \label{pa2_closure}
\end{align}
Furthermore, analogous to \cite[Lemma 6.3]{BM13_2pt_published} 
we conclude that for $d \geq 4$, 
\begin{align}
\overline{\overline{\A^{\delta}_1} \circ \PP \A}_2 &= (\overline{\A^{\delta}_1} \circ \PP \A_2) \cup \overline{\A^{\delta}_1} \circ (\overline{\PP \A}_2- \PP \A_2) \cup \overline{\A^{\delta-1}_1} \circ ( \Delta \overline{\PP \A}_4 \cup \Delta \overline{\widehat{D}}_5). \label{pa2_closure_again}
\end{align}
Let us define 
\begin{align*}
\mathcal{B} &:= \overline{A_1^{\delta} \circ \overline{\PP A}}_2 - \A^{\delta}_1 \circ (\PP\A_2\cup \overline{\PP A}_3). 
\end{align*}
We will show shortly that the section 
$\Psi_{\PP \A_3}$ vanishes on the points of  
$\A_1^{\delta}\circ \PP A_3$ transversally. 
Hence,  
\begin{align}
\lan e(\mathbb{L}_{\PP \A_3}), 
~~[\overline{\A^{\delta}_1 \circ \overline{\PP\A}}_2] \cap [\mu] \ran & = \N(\A_1^{\delta}\PP \A_3, n_1, n_2,n_3, \theta)  
+ \mathcal{C}_{\mathcal{B}\cap \mu}. 
\label{npa3_Euler_class_formula}
\end{align}
We now give an explicit description of $\mathcal{B}$. 
Let us first define 
\begin{align*}
\mathcal{B}_0 &:= \{ ([f], [\eta], q_1, \ldots q_{\delta}, l_{q_{\delta+1}}) \in \mathcal{B}: q_1, q_2 \ldots q_{\delta+1} 
~~\textnormal{are all distinct}\}. 
\end{align*}
In other words, $\mathcal{B}_0$ is that component of the boundary, where all the points are still distinct.   
By \cref{pa2_closure}, we conclude that 
\begin{align*}
\mathcal{B}_0 &= \overline{A_1^{\delta}} \circ \overline{\widehat{D}}_4. 
\end{align*}
If we intersect $\mathcal{B}_0$ with $\mu$ then we will get a finite set of points. 
Since the representative $\mu$ is generic, we conclude that the third derivative along 
$v$ will not vanish, i.e. the section $\Psi_{\PP A_3}$ will not vanish on those points. 
Hence, $\mathcal{B}_0\cap \mu$ does not contribute to the Euler class. \\
\hf\hf Next, let us 
consider the components of 
$\mathcal{B}$ where one (or more) of the $q_i$ become equal to the last point $q_{\delta+1}$. 
Define $\mathcal{B}(q_{i_1}, \ldots q_{i_k}, l_{q_{\delta}})$ as before. 
Analogous to the proof of \cite[Lemma 6.3]{BM13_2pt_published}, 
we conclude that 
\begin{align*}
\B (q_1, l_{q_{\delta +1 }}) & \approx \overline{\A_1^{\delta-1} \circ \PP \A}_4 \cup 
\overline{\A_1^{\delta-1} \circ \widehat{D}}_5.
\end{align*}
Furthermore, analogous to the proof of \cite[Corollary 6.13, Page 700]{BM13_2pt_published}, 
we conclude that 
the contribution to the Euler class from each of the points of 
$\overline{\A_1^{\delta-1} \circ \PP \A}_4 \cap \mu$ is $2$. 
Finally, we note that the section $\Psi_{\PP A_3}$ does not vanish on $\overline{\A_1^{\delta-1} \circ \widehat{D}}_5 \cap \mu$, since $\mu$ is generic.  
Hence, the total contribution from all the components of type $\B (q_{i_1}, l_{q_{\delta +1 }})$ equals
\bgd
2\binom{\delta}{1}N(\A_1^{\delta-1}\PP\A_4,n_1,n_2,n_3, \theta).
\edd
Plugging in this in \cref{npa3_Euler_class_formula}
gives us the formula of \cref{npa3}. \\ 
\hf \hf It just remains to prove the transversality claim. 
This follows from following the setup of proof of transversality in 
Theorem  \cref{npa2}. We consider the polynomial 
\begin{align*}
\rho_{30}&:= (X-X_1)^2(X-X_2)^2\ldots \cdot (X-X_{\delta})^2 X^3 Z^{d-2\delta-3}
\end{align*}
and the corresponding curve $\gamma_{30}(t)$. Transversality follows by computing the derivative of the section $\Psi_{\PP A_3}$ 
along the curve $\gamma_{30}(t)$ as before. \qed 

\subsection{Proof of \Cref{npa4}: computation of $N(\mathcal{P} A_4)$} 
We will now justify our formula for 
\newline $N(\mathcal{P} A_4,r, s, n_1, n_2, n_3, \theta)$. 
Let $\mu$ be a generic cycle, representing 
the class 
\begin{align*}
[\mu] = \mathcal{H}_L^r \cdot \mathcal{H}_p^s \cdot a^{n_1} \lambda^{n_2} (\pi^*H)^{n_3} (\pi^*\lambda_{W})^{\theta}. 
\end{align*}
Let $v \in \gamma_{W}$ and $w \in \pi^*W/\gamma_{W}$ be two fixed nonzero vectors. 
Let us introduce the following abbreviation:  
\begin{align*}
f_{ij} & := \nabla^{i+j} f|_{q}
(\underbrace{v,\cdots v}_{\textnormal{$i$ times}}, \underbrace{w,\cdots w}_{\textnormal{$j$ times}}).
\end{align*}
We now define a section of the following bundle 
\begin{align}
\Psi_{\PP \A_4}:  \overline{\PP \A}_3 
\lra \mathbb{L}_{\PP \A_4} & := 
\gamma_{\DD}^{\ast 2}\otimes\gamma_W^{\ast 4}\otimes(W/\gamma_{W})^{\ast 2}\otimes \gamma_{\mathbb{P}^3}^{\ast 2d}, \label{psiPA_4} \\
\{\Psi_{\PP \A_4}([f], l_{q})\}(f^{\otimes 2}
\otimes v^{\otimes 4}\otimes w^{\otimes 2}) &:= f_{02}A_4^f, \qquad \textnormal{where} \qquad 
\A^{f}_4 := f_{40}-\frac{3 f_{21}^2}{f_{02}}. \label{A^f_4}
\end{align}
Analogous to \cite[Lemma 6.1]{BM13_2pt_published},  we conclude that 
\begin{align}
\overline{\PP A}_3 & = \PP A_3 \cup \overline{\PP A}_4 \cup \overline{\PP D}_4.  \label{pa3_closure}
\end{align}
Hence, let us define 
\begin{align*}
\mathcal{B} &:= \overline{\PP A}_3 - \PP A_3 \cup \overline{\PP A}_4. 
\end{align*}
We will show shortly that 
the section 
$\Psi_{\PP \A_4}$ vanishes on the points of  
$\PP A_4$ transversally. 
Hence,  
\begin{align}
\lan e(\mathbb{L}_{\PP \A_4}), 
~~[\overline{\PP\A}_3] \cap [\mu] \ran & = \N(\A_1^{\delta}\PP \A_4, r, s, n_1, n_2, n_3, \theta)  
+ \mathcal{C}_{\mathcal{B}\cap \mu}.
\label{npa4_Euler_class_formula}
\end{align}
Let us now study the boundary $\mathcal{B}$. 
By \cref{pa3_closure}, we conclude that 
\begin{align*}
\mathcal{B} \cap \mu &= \overline{\PP D}_4 \cap \mu. 
\end{align*}
Since the representative $\mu$ is generic, we conclude that the directional derivative 
$f_{21}$ will not vanish on those points. Since $f_{02} =0$ on $\mathcal{B}$, 
we conclude that 
\begin{align*}
f_{02}A^f_4 &= f_{02}f_{40} - 3f_{21}^2 \neq 0 
\end{align*}
if $f_{21} \neq 0$.
Hence, the section $\Psi_{\PP A_4}$ will not vanish on 
$\mathcal{B}\cap \mu$. Hence, the total boundary contribution is zero and \cref{npa4_Euler_class_formula} 
gives us the formula of \cref{npa4}. \\ 
\hf \hf It remains to prove the claim regarding transversality. This 
follows from following the setup of proof of transversality in 
 \Cref{npa3}. We consider the polynomial 
\begin{align*}
\rho_{40}&:=  X^4 Z^{d-4}
\end{align*}
and the corresponding curve $\gamma_{40}(t)$. Transversality follows by computing the derivative of the section $\Psi_{\PP A_4}$ 
along the curve $\gamma_{40}(t)$ as before. \qed

\subsection{Proof of \Cref{npd4}: computation of $N(\mathcal{P} D_4)$} 
We will now justify our formula for 
$N(\mathcal{P} D_4,r, s, n_1, n_2, n_3, \theta)$. 
Let $\mu$ be a generic cycle, representing 
the class 
\begin{align*}
[\mu] & = \mathcal{H}_L^r \cdot \mathcal{H}_p^s \cdot a^{n_1} \lambda^{n_2} (\pi^*H)^{n_3} (\pi^*\lambda_{W})^{\theta}. 
\end{align*}
As before, let $v \in \gamma_{W}$ and $w \in \pi^*W/\gamma_{W}$ be two fixed nonzero vectors. 
Define a section of the following bundle 
\begin{align}
\Psi_{\PP \D_4}:  \overline{\PP \A}_3 
\lra \mathbb{L}_{\PP \D_4} & := \gamma_{\DD}^{\ast}\otimes(W/\gamma_W)^{\ast 2}\otimes \gamma_{\mathbb{P}^3}^{*d}, \qquad 
\textnormal{given by} \nonumber \\
\{\Psi_{\PP \D_4}([f], l_{q})\}(f\otimes w^{\otimes 2}) &:= \nabla^2f|_q(w,w). \label{psiD_4}
\end{align}
We recall \cref{pa3_closure}, namely 
\begin{align}
\overline{\PP A}_3 & = \PP A_3 \cup \overline{\PP A}_4 \cup \overline{\PP D}_4.  \label{pa3_closure_again}
\end{align}
We now define 
\begin{align*}
\mathcal{B} &:= \overline{\PP A}_3 - (\PP\A_3\cup \overline{\PP D}_4). 
\end{align*}
We will show that  
the section 
$\Psi_{\PP \D_4}$ vanishes on the points of  
$\PP D_4$ transversally.   
Hence, 
\begin{align}
\lan e(\mathbb{L}_{\PP \D_4}), 
~~[\overline{\PP\A}_3] \cap [\mu] \ran & = \N(\PP \D_4, r,s, n_1, n_2, n_3, \theta)  
+ \mathcal{C}_{\mathcal{B}\cap \mu}. \label{npd4_Euler_class_formula}
\end{align}
By definitions, the section 
$\Psi_{\PP D_4}$ does not vanish on $\PP A_4 \cap \mu$. Hence, the total boundary contribution is zero 
and \cref{npa4_Euler_class_formula} 
gives us the formula of \cref{npd4}. \\ 
\hf \hf It remains to prove the claim regarding transversality. This 
follows from following the setup of proof of transversality in 
 \Cref{npa4}. We consider the polynomial 
\begin{align*}
\rho_{02}&:=  Y^2 Z^{d-2}
\end{align*}
and the corresponding curve $\gamma_{02}(t)$. Transversality follows by computing the derivative of the section $\Psi_{\PP D_4}$ 
along the curve $\gamma_{02}(t)$ as before. \qed


\section{Verification with other results and low degree checks}
\label{low_degree_checks}
Let us make a few low degree checks. We will abbreviate $N(A_1^{\delta+1}, r,s,0,0)$ as 
$N(A_1^{\delta+1}, r,s)$.
\subsection{Verification with S.~Kleiman and R.~Piene's result} 
\label{KP_check}
Let us start by verifying the numbers predicted by the algorithm of S.~Kleiman and R.~Piene in \cite{KP2}. 
Let us explain how to obtain the formula for $N(A_1^{\delta+1}, r,s)$ using \cite[Algorithm 2.3, Page 5]{KP2}. 
Let us first define four polynomials (called Bell polynomials), given by 
\begin{align*}
P_1(a_1):= a_1, \qquad P_2(a_1, a_2)&:= a_1^2 + a_2, \qquad 
P_3(a_1, a_2, a_3):= a_1^3+3a_1 a_2 + a_3 \qquad \textnormal{and} \\ 
P_4(a_1, a_2, a_3, a_4)&:= a_1^4 + 6 a_1^2 a_2 + 3 a_2^2 + 4 a_1 a_3 + a_4.    
\end{align*}
We define the following cycles in $\mathcal{S}_{\mathcal{D}_{1}}$, namely 
\begin{align}
v&:= \lambda + dH, \qquad w_1:= a-3H \qquad \textnormal{and} \qquad w_2:= a^2-2aH+3aH^2. \label{vw}
\end{align}
Note that $v= c_1(\mathcal{L}_{A_0})$ and $w_i= c_i(T^*W)$, where $\mathcal{L}_{A_0}$ and $W$ are the 
bundles defined in section \ref{na1_delta_proof}. The algorithm \cite[Algorithm 2.3, Page 5]{KP2} 
produces polynomials $b_i(v,w_1, w_2)$ of degree $i+2$ (from $i=1$ to $8$). Let us write down the expressions explicitly, 
\begin{align}
b_1(v, w_1, w_2)&= v^3 + v^2w_1 + vw_2, \qquad 
b_2(v, w_1, w_2) = -7 v^4-13 v^3 w_1-6 v^2 w_1^2-7 v^2 w_2-6 v w_1 w_2, \nonumber \\ 
b_3(v, w_1, w_2)& = 138 v^5+394 v^4 w_1+376 v^3 w_1^2+138 v^3 w_2 +120 v^2 w_1^3+256 v^2 w_1 w_2+120 v w_1^2 w_2 \qquad \textnormal{and} \nonumber \\ 
b_4(v, w_1, w_2)& = -4824 v^6-19134 v^5 w_1-28842 v^4 w_1^2-3888 v^4 w_2-19572 v^3 w_1^3 \nonumber \\ 
                & ~~    -12438 v^3 w_1 w_2  
                -5040 v^2 w_1^4-13596 v^2 w_1^2 w_2 \nonumber \\ 
                & ~~ +936 v^2 w_2^2-5040 v w_1^3 
   w_2+936 v w_1 w_2^2.   \label{b1_defn} 
\end{align}
The numbers $N(A_1^{\delta+1}, r,s)$ will be computed from the polynomials $P_{\delta+1}$ 
by intersecting cycles in $\mathcal{S}_{\mathcal{D}_{\delta+1}}$. Let 
$\pi_i:\mathcal{S}_{\mathcal{D}_{\delta+1}} \longrightarrow \mathcal{S}_{\mathcal{D}_{1}}$ denote 
the $i^{\textnormal{th}}$ projection map. Then 
\begin{align*}
N(A_1, r,s)&=  [b_1]\cdot \mathcal{H}_L^r \cdot \mathcal{H}_p^s,
\end{align*}
where the right hand side is an intersection number on $\mathcal{S}_{\mathcal{D}_{1}}$. Note that we plug in the values for 
$v, w_1$ and $w_2$ from \cref{vw} in \cref{b1_defn}, use \cref{HL_Hp_class} for $\mathcal{H}_L$ and $\mathcal{H}_p$ 
and the ring structure  as given by \cref{ring_str} to compute the intersection number. 
Next, let us explain how to compute $N(A_1^2, r,s)$. This is given by 
\begin{align}
N(A_1^2, r,s)&= (\pi_1^*b_1)\cdot (\pi_2^*b_1) \cdot \mathcal{H}_L^r \cdot \mathcal{H}_p^s + b_2 \cdot \mathcal{H}_L^r \cdot \mathcal{H}_p^s.  
\label{two_nodes_kp}
\end{align}
The first number on the right hand side of \cref{two_nodes_kp} is an intersection number on 
$\mathcal{S}_{\mathcal{D}_{2}}$, while the second one is an intersection number on $\mathcal{S}_{\mathcal{D}_{1}}$. Similarly, 
\begin{align*}
N(A_1^3, r,s)&= 
\Big((\pi_1^*b_1)\cdot (\pi_2^*b_1) \cdot (\pi_3^*b_1)  + 3 (\pi_1^*b_1)\cdot (\pi_2^*b_1) + b_3\Big) \cdot 
\mathcal{H}_L^r \cdot \mathcal{H}_p^s \qquad \textnormal{and} \\ 
N(A_1^4, r,s)&= \Big((\pi_1^*b_1)\cdot (\pi_2^*b_1)\cdot (\pi_3^*b_1) \cdot (\pi_4^*b_1) + 
6 (\pi_1^*b_1) \cdot (\pi_2^*b_1) \cdot (\pi_3^*b_1)  \\ 
& \qquad + 3 (\pi_1^*b_2)\cdot (\pi_2^*b_2)  + 4 (\pi_1^*b_1)\cdot (\pi_2^*b_3) + b_4 \Big)\cdot \mathcal{H}_L^r \cdot \mathcal{H}_p^s. 
\end{align*}
We have written a mathematica program to implement this formula and verified that the answers agree with our formula.

\subsection{Verification with T.~Laraakker's result} 
Next we note that in \cite[Appendix A,  Page 32]{TL}, T.~Laraakker has explicitly written down the formulas for $N(A_1^{\delta+1},0,0)$. 
We have verified that our formulas agree with his. 


\subsection{Verification with the second author and R.~Singh's result}
We now verify some  of the numbers obtained by R.~Mukherjee and R.~Singh in \cite{RS}.   
In \cite{RS}, the authors compute $C_d^{\textnormal{Planar}, \mathbb{P}^3}(r,s)$, the number of planar genus zero degree $d$ curves in 
$\mathbb{P}^3$ intersecting $r$ lines and passing through $s$ points having a cusp (where $r+2s = 3d+1$). Let us compare this with 
$N_d(A_1^{\delta} A_2, r, s)$, the number 
of planar degree $d$ curves in $\mathbb{P}^3$, passing through $r$ lines and 
passing through $s$ points,
that have $\delta$ (ordered) nodes and one cusp
(where $r +  2s =\dfrac{d(d+3)}{2} +1-\delta$).
For 
$d =3$, and $\delta =0$, this number should be the same as the characteristic number of genus zero planar cubics in $\mathbb{P}^3$ with a cusp, 
i.e. $C_d(r, s)$.  We have verified that is indeed the case. We tabulate the numbers for the readers convenience: 
\begin{align*}
C_3(10, 0)& = 17760, \quad C_3(8, 1)  = 2064, \quad C_3(6, 2) = 240 \quad \textnormal{and} \qquad C_3(4, 3) = 24. 
\end{align*}
These numbers are the same as $N_d(A_1^{\delta} A_2, r, s)$ for $d=3$ and $\delta =0$. \\ 
\hf \hf Next, we note that when $d=4$ and $\delta =2$, the number $\frac{1}{\delta !}N_d(A_1^{\delta} A_2, r, s)$ 
is same as the characteristic number of genus zero planar quartics in $\mathbb{P}^3$ with a cusp, i.e.   
$C_d(r, s)$.  We have verified that fact. The numbers are  
\begin{align*}
C_4(13, 0)& = 10613184, \quad C_4(11, 1) = 760368, \quad C_4(9, 2)  = 49152 \quad \textnormal{and} \quad C_4(7, 3) = 2304.
\end{align*}
These numbers are the same as $\frac{1}{2!}N_d(A_1^{\delta} A_2, r, s)$ for $d=4$ and $\delta =2$. 
We have to divide out by a factor of $\delta!$ because in the definition of 
$N_d(A_1^{\delta} A_2, r, s)$, the nodes are ordered.

\subsection{Enumerativity of BPS numbers computed by R.~Pandharipande}
We will now verify some of the numbers predicted by the conjecture made by Pandharipande in \cite{RPDeg}, 
regarding the enumerativity of the BPS numbers for $\mathbb{P}^3$.  
Let $N^d_g(r,s)$ denote the genus $g$ Gromov-Witten 
invariant of $\mathbb{P}^3$ (corresponding to the insertion of $r$ lines and $s$ points) 
and let $E^d_g(r,s)$ denote the corresponding BPS invariant as given by 
\cite[Equations 5 and 9, Pages 493 and 494]{RPDeg}. 
The  numbers $E^d_g(r,s)$ are conjectured to be integers. Even if the conjecture is true, 
it is not always clear if the the BPS numbers have an enumerative significance. We will now give some evidence 
for the enumerativity of some of the BPS number. \\ 
\hf \hf Let us consider the case $g=2$ and $d=4$. It is far from clear that $E^d_2(r,s)$ is enumerative when $d=4$, 
because the moduli space of curves has  more than the expected dimension 
(see the remark in \cite{RPDeg} just after Theorem 3, Page 494). We claim that $E^d_2(r,s)$ is enumerative when $d=4$.  
To see how, 
we first note that every degree $4$, genus $2$ curve lies inside some $\mathbb{P}^2$ 
(this follows from the Castelnuovo bound, \cite[Page 527]{GH}). Since the genus of a smooth degree $4$ curve is $3$, 
we conclude that the corresponding enumerative invariant is equal to the characteristic number of planar degree $4$ 
curves in $\mathbb{P}^3$ with one node. We have verified that $E^d_2(r,s)$ is indeed equal to $N_d(A_1, r,s)$ for all $r$ and $s$ when $d=4$. 
We tabulate the numbers for the readers convenience 
\begin{align}
N_4(A_1, 16, 0)&= 258300, \quad N_4(A_1, 14, 1)= 15498, \quad 
N_4(A_1, 12, 2)= 792 \quad \textnormal{and} \quad N_4(A_1, 10, 3) = 27. 
\label{BPS_g_2}
\end{align}
The degree four, genus two BPS numbers are directly tabulated in \cite[Page 43]{AGath} and are seen to be equal to the above numbers
listed in \cref{BPS_g_2}.  



\section{Explicit Formulas} 
\label{expfor}
\noindent For the convenience of the reader, we write down some explicit formulas. 
\begin{align*}
N(r,s, 0, 0) &= \begin{cases}
\frac{1}{324} d(d^2-1)(d+2)\left(d^2 +4d +6 \right) \left(2d^3 +6d^2+13d+3 \right) & \mbox{if} ~~ s =0,  \\
\frac{1}{36}d(d^2-1)(d+2) \left( 2d^2 +8d +3 \right) & \mbox{if} ~~s=1,\\
\frac{1}{3}d(d-1)(d+4) & \mbox{if} ~~s= 2, \\
1 & \mbox{if} ~~s= 3.
\end{cases} \\
N(A_1, r,s, 0, 0) &= \begin{cases}
 \frac{1}{108}d(d^2 -1)^2 (d+2)(d+3) \left(2d^4 +4d^3 +d^2 -10d -6 \right) & \mbox{if} ~~ s =0,  \\
 \frac{1}{12} d(d-1)^2(d+3) \left( 2d^4 +6d^3 -9d^2 -3d -2 \right) & \mbox{if} ~~s=1,\\
d(d-1)^2 \left( d^2 +3d-6 \right)  & \mbox{if} ~~s= 2, \\
 3(d-1)^2 & \mbox{if} ~~s= 3.
\end{cases}\\
N(A_2, r,s, 0, 0) &= \begin{cases}
 \frac{1}{27} d(d^2 -1)(d^2 -4) \left( 2 d^6+12 d^5+11 d^4 -30 d^3-49 d^2-18 \right) & \mbox{if} ~~ s =0,  \\
 \frac{1}{3} d(d-1)(d-2) \left(2 d^5+12 d^4+d^3-54 d^2+9 d+6 \right) & \mbox{if} ~~s=1,\\
4d(d-1)(d-2) \left(d^2+3 d-8\right) & \mbox{if} ~~s= 2, \\
 12(d-1)(d-2) & \mbox{if} ~~s= 3.
\end{cases} \\
N(A_3, r,s, 0, 0) &= \begin{cases}
 \frac{1}{162} d(d-1)(d-2)\big( 50 d^8+408 d^7+539 d^6-2556 d^5-6625 d^4 \\
 \qquad \qquad \qquad \qquad +762 d^3+10050 d^2-11232 d+8208\big) & \mbox{if} ~~ s =0,  \\
 \frac{1}{18} (d-2) (d-1) \big(50 d^6+258 d^5-485 d^4-2241 d^3 \\
 \qquad \qquad \qquad \qquad \qquad \qquad \qquad +2172 d^2+1512 d-648\big) & \mbox{if} ~~s=1,\\
\frac{2}{3} d(d-2)(d+5) \left(25 d^2-96 d+84\right) & \mbox{if} ~~s= 2, \\
2 \left(25 d^2-96 d+84\right) & \mbox{if} ~~s= 3.
\end{cases}\\
N(A_4, r,s, 0, 0) &= \begin{cases}
 \frac{5}{27} (d-1)(d-3) \big(6 d^9+50 d^8+41 d^7-445 d^6-715 d^5 \\
 \qquad \qquad \qquad  +1529 d^4+2720 d^3-7902 d^2+7164 d-2160\big) & \mbox{if} ~~ s =0,  \\
 \frac{5}{3} (d-3) \big(6 d^7+26 d^6-105 d^5-231 d^4 \\
 \qquad \qquad \qquad \qquad +765 d^3-107 d^2-762 d+360\big) & \mbox{if} ~~s=1,\\
 20 d(d-3) (3 d-5) \left(d^2+3 d-12\right) & \mbox{if} ~~s= 2, \\
 60(d-3)(3d -5) & \mbox{if} ~~s= 3.
\end{cases}\\
N(D_4, r,s, 0, 0) &= \begin{cases}
\frac{5}{36} (d-1)(d-2)^2 (d+4) \big(2 d^7+12 d^6-d^5-66 d^4 -91 d^3 \\
\qquad \qquad \qquad \qquad \qquad +234 d^2-270 d+108\big) & \mbox{if} ~~ s =0,  \\
\frac{5}{4} (d-2)^2 \left(2 d^6+12 d^5-15 d^4-102 d^3+85 d^2+90 d-48\right) & \mbox{if} ~~s=1,\\
15 d(d-2)^2 \left(d^2+3 d-12\right) & \mbox{if} ~~s= 2, \\
 45 (d-2)^2 & \mbox{if} ~~s= 3.
\end{cases}\\
N(A_1^2, r,s, 0, 0) &= \begin{cases}
\frac{1}{108} d(d^2-1)(d^2-4) \big(6 d^8+30 d^7-25 d^6-255 d^5-142 d^4 \\
\qquad \qquad \qquad+333 d^3+629 d^2+18 d+198\big) & \mbox{if} ~~ s =0,  \\
\frac{1}{12} d(d-1)(d-2) \big(6 d^7+30 d^6-55 d^5-297 d^4+190 d^3 \\
\qquad \qquad \qquad \qquad +537 d^2-69 d-78\big) & \mbox{if} ~~s=1,\\
 d(d-1)(d-2) \left(d^2+3 d-8\right) \left(3 d^2-3 d-11\right) & \mbox{if} ~~s= 2, \\
3 (d-1)(d-2) \left(3 d^2-3 d-11\right) & \mbox{if} ~~s= 3.
\end{cases}\\
\end{align*}

\begin{align*}   
N(A_1 A_2, r,s, 0, 0) &= \begin{cases}
 \frac{1}{27} d(d-1)(d-2)(d-3) \big(6 d^9+60 d^8+155 d^7-186 d^6 \\
  \qquad -1288 d^5 -1422 d^4+641 d^3+1512 d^2-2034 d+1836\big) & \mbox{if} ~~ s =0,  \\
 \frac{1}{3} (d^2 -1)(d-2)(d-3) \big(6 d^6+36 d^5-37 d^4-338 d^3 \\
 \qquad \qquad \qquad  +123 d^2+438 d-144\big) & \mbox{if} ~~s=1,\\
 4 d(d-2)(d-3)(d+5) \left(3 d^3-6 d^2-11 d+18\right) & \mbox{if} ~~s= 2, \\
 12 (d-3) \left(3 d^3-6 d^2-11 d+18\right) & \mbox{if} ~~s= 3.
\end{cases}\\
N(A_1 A_3, r,s, 0, 0) &= \begin{cases}
 \frac{1}{54} (d-1)(d-3) \Big(50 d^{11}+358 d^{10}-489 d^9-6967 d^8 \\
-3139 d^7 +40955 d^6+40482 d^5-112250 d^4-131080 d^3 \\
 \qquad \qquad +436176 d^2-402480 d+120960\Big) & \mbox{if} ~~ s =0,  \\
\frac{1}{6} (d-3) \big(50 d^9+158 d^8-1471 d^7-2389 d^6+14857 d^5 \\
\qquad \qquad +2359 d^4 -41156 d^3+7912 d^2+41808 d-19440\big) & \mbox{if} ~~s=1,\\
 2 d(d-3) \left(d^2+3 d-12\right) \left(25 d^3-71 d^2-122 d+280\right) & \mbox{if} ~~s= 2, \\
 6 (d-3) \left(25 d^3-71 d^2-122 d+280\right) & \mbox{if} ~~s= 3.
\end{cases}\\
N(A_1^3, r,s, 0, 0) &= \begin{cases}
\frac{1}{108} d(d-1)(d-2) \Big(18 d^{12}+108 d^{11}-315 d^{10}-2664 d^9 \\
+470 d^8+21919 d^7+19103 d^6-58136 d^5-106948 d^4 \\
\qquad \qquad +7039 d^3+129360 d^2-165798 d+110700\Big) & \mbox{if} ~~ s =0,  \\
\frac{1}{12} (d-1)(d-2) \Big(18 d^{10}+54 d^9-567 d^8-1179 d^7+6383 d^6 \\
 +7774 d^5-25775 d^4-20197 d^3+26955 d^2+20802 d-8640\Big) & \mbox{if} ~~s=1,\\
 d(d-2)(d+5) \Big(9 d^6-54 d^5+9 d^4+423 d^3 \\
 \qquad \qquad \qquad \qquad \qquad \qquad -458 d^2-829 d+1050\Big) & \mbox{if} ~~s= 2, \\
 3 \left(9 d^6-54 d^5+9 d^4+423 d^3-458 d^2-829 d+1050\right) & \mbox{if} ~~s= 3.
\end{cases}\\ 
N(A_1^2 A_2, r,s, 0, 0) &= \begin{cases}
 \frac{1}{9} (d-1)(d-3) \Big(6 d^{13}+36 d^{12}-159 d^{11}-1124 d^{10}+1209 d^9 \\
 \qquad +12169 d^8+664 d^7-52991 d^6-39896 d^5+127254 d^4 \\
 \qquad \qquad +129112 d^3-452904 d^2+413280 d-120960\Big) & \mbox{if} ~~ s =0,  \\
 (d-3) \Big(6 d^{11}+12 d^{10}-249 d^9-236 d^8+3653 d^7+367 d^6 \\
  -20186 d^5+6389 d^4+38600 d^3-7828 d^2-42896 d+19680\Big) & \mbox{if} ~~s=1,\\
 12 d(d-3) \left(d^2+3 d-12\right) \big(3 d^5-12 d^4-30 d^3 \\
 \qquad \qquad \qquad \qquad \qquad \qquad \qquad +125 d^2+82 d-280\big) & \mbox{if} ~~s= 2, \\
 36 (d-3) \left(3 d^5-12 d^4-30 d^3+125 d^2+82 d-280\right) & \mbox{if} ~~s= 3.
\end{cases}
\end{align*}
\begin{align*} 
N(A_1^4, r,s, 0, 0) &= \begin{cases}
 \frac{1}{36} (d-1)(d-3) \Big(18 d^{15}+90 d^{14}-747 d^{13}-3843 d^{12}+11660 d^{11} \\
 +63140 d^{10}-75352 d^9-486678 d^8+73143 d^7+1773729 d^6+1150606 d^5 \\
 \qquad -4123550 d^4-3282032 d^3+12893256 d^2-11795040 d+3404160\Big) & \mbox{if} ~~ s =0,  \\
 \frac{1}{4} (d-3) \Big(18 d^{13}+18 d^{12}-945 d^{11}-261 d^{10}+18590 d^9-4254 d^8 \\
 \qquad -164328 d^7+80206 d^6+653953 d^5-362481 d^4-1051128 d^3 \\
 \qquad \qquad \qquad +245636 d^2+1215312 d-554880\Big) & \mbox{if} ~~s=1,\\
 3d(d-3) \left(d^2+3 d-12\right) \Big(9 d^7-45 d^6-135 d^5+801 d^4 \\
 \qquad \qquad \qquad +691 d^3-4671 d^2-1386 d+7880\Big) & \mbox{if} ~~s= 2, \\
 9 (d-3) \big(9 d^7-45 d^6-135 d^5+801 d^4+691 d^3 \\
 \qquad \qquad \qquad -4671 d^2-1386 d+7880\big) & \mbox{if} ~~s= 3.
\end{cases}
\end{align*} 

\section{Acknowledgment} 
The ideas of this paper originated while the second author had discussions with 
Martijn Kool and Ties Laarakker regarding the papers \cite{BM8} and \cite{KST}. During the discussion 
we wondered if one can count planar curves in $\mathbb{P}^3$ with singularities. 
As shown by Ties Laarakker in \cite{TL}, one can adapt the techniques in \cite{KST} 
to count $\delta$-nodal planar curves in $\mathbb{P}^3$. On the other hand, the 
discussions also led us to conclude that by adapting the methods of \cite{R.M}, 
\cite{BM13_2pt_published} and \cite{BM8}, we can enumerate planar curves in  
$\mathbb{P}^3$ with singularities that are more degenerate than nodes. The result is this 
present paper. The second author is therefore very grateful to Martijn Kool and Ties Laarakker 
for the discussions and fruitful exchange of ideas that resulted in this paper. 
The second author would  also like to thank Steven Kleiman for his comments 
and pointing out that the ideas of \cite{KP1} and \cite{KP2} combined can be used 
to compute the numbers obtained in this paper (see \Cref{KP_remark}).  
The second author 
would also like to thank Center for Quantum Geometry of Moduli Space at Aarhus, Denmark (DNRF95) for giving him a chance to spend six 
weeks 
there, when the author  got the initial idea for this project; the visit was 
was mainly paid by the grant ``EU-IRSES Fellowship within FP7/2007-2013 under grant 
agreement number 612534, project MODULI - Indo European Collaboration on Moduli Spaces." 
Finally, the second author would like to acknowledge the External Grant 
he has obtained, namely 
MATRICS (File number: MTR/2017/000439) that has been sanctioned by the Science and Research Board (SERB). 
Both the authors are grateful to Anantadulal Paul and Rahul Singh for several fruitful discussions.   

\bibliographystyle{siam}

\begin{thebibliography}{10} 
\bibitem{R.M}
{\sc S.Basu and R.Mukherjee },
       {\em Enumeration of curves with one singular point},  Journal of Geometry and Physics,104 (2016), 175--203.\\
\bibitem{BM13_2pt_published}
{\sc S.Basu and R.Mukherjee },
       {\em  Enumeration of curves with two singular point}, Bull.Sci.math, 139(2015), 667--735.\\
\bibitem{BM8}
       {\sc S.Basu and R.Mukherjee },
       {\em  Counting curves in a linear system with upto eight singular points},\\
available at \url{https://arxiv.org/abs/1909.00772}. \\       
\bibitem{Berczi}
{\sc G.~B\'{e}rczi}, Tautological integrals on curvilinear Hilbert schemes , Geometry and 
Topology, (2017), pp.~2897--2944. \\
\bibitem{FB} {F.~Block}, \newblock {\em Relative node polynomials for plane curves}, J. Algebraic Combin., 36 (2012)
no. 2, pp.~ 279--308.\\
\bibitem{CH}
{\sc L.~Caporaso and J.~Harris}, {\em Counting plane curves of any genus},
  Invent. Math., 131 (1998), pp.~345--392.\\ 
\bibitem{FoMi}
{\sc S. Fomin and G. Mikhalkin} {\em Labelled floor diagrams for plane curves}, J. Eur. Math. Soc., 12 (2010), no. 6, pp.~ 1453--1496.\\
\bibitem{AGath}
{\sc A. Gathmann} {\em Gromov-Witten invariants of hypersurfaces}, 
Habilitation Thesis, TU Kaiserslautern (2003), available at 
\url{https://www.mathematik.uni-kl.de/~gathmann/en/publications.php}\\ 
\bibitem{GH}
{\sc P. Griffiths and J. Harris} {\em Principles of Algberaic Geometry}, Wiley Classics Library, 1994. \\ 
\bibitem{Kaz}
{\sc M.~{\`E}. Kazarian}, {\em Multisingularities, cobordisms, and enumerative
  geometry}, Uspekhi Mat. Nauk, 58 (2003), pp.~29--88. \\
\bibitem{Ker1}
{\sc D.~Kerner}, {\em Enumeration of singular algebraic curves}, Israel J.
  Math., 155 (2006), pp.~1--56.\\ 
\bibitem{Ker2}
{\sc D.~Kerner}, {\em On the enumeration of complex plane curves with two singular points},
IMRN, (23):4498--4543, 2010.\\
\bibitem{Kes} 
{\sc S.~Kesavan}, {\em Non Linear Functional Analysis a First Course}, 
Texts and Readings in Mathematics 28, Hindustan Book Agency, 2004.\\ 
\bibitem{KP1}
{\sc S.~Kleiman and R.~Piene}, {\em Enumerating singular curves on surfaces},
  in Algebraic geometry: {H}irzebruch 70 ({W}arsaw, 1998), vol.~241 of Contemp.
  Math., Amer. Math. Soc., Providence, RI, 1999, pp.~209--238. \\ 
\bibitem{KP2}
{\sc S.~Kleiman and R.~Piene}, {\em Node polynomial for families: methods and applications},
  Math. Nachr. 271 (2004), pp. ~69-90. \\ 
  \bibitem{KP3} 
  {\sc S.~Kleiman}, {\em Personal communication}. \\
\bibitem{KST}
{\sc M.~Kool, V.~Shende, and R.~P. Thomas}, {\em A short proof of the
  {G}\"ottsche conjecture}, Geom. Topol., 15 (2011), pp.~397--406. \\ 
\bibitem{TL}  {\sc T.~Laarakker}, {\em The Kleiman-Piene Conjecture and node polynomials for plane curves in $\mathbb{P}^3$}, 	
Sel. Math. New Ser., 24 (2018), pp.~4917--4959. \\ 
\bibitem{MPS}
{\sc R.~Mukherjee, A.~Paul and R.~Singh},
       {\em  Enumeration of Rational Curves in a Moving family of $\mathbb{P}^2$}, Bull.Sci.math, 150 (2019), 1--11.\\
\bibitem{RS}
{\sc R.~Mukherjee and R.~Singh}, {\em Enumeration of Rational Cuspidal Curves in a Moving family of $\mathbb{P}^2$}, \\ 
available at \url{https://arxiv.org/abs/2005.10664}. \\ 
\bibitem{RPDeg}
{\sc R.~Pandharipande}, 
{\em Hodge Integrals and Degenerate Contributions}, Comm. Math. Phys, Vol 208, (1999), Issue 2, pp 489--506.  \\ 
\bibitem{Ran1}
{\sc Z.~Ran}, {\em Enumerative geometry of singular plane curves}, Invent.
  Math., 97 (1989), pp.~447--465. \\ 
\bibitem{Ran}
{\sc Z.Ran}, \newblock {\em Enumerative geometry of divisorial families of rational curves}, 
Annali della Scuola Normale Superiore di Pisa - Classe di Scienze, Série 5, Volume 3 (2004) no. 1, pp. 67-85.\\ 
\bibitem{Tz}
{\sc Y.-J. Tzeng}, {\em A proof of the {G}\"ottsche-{Y}au-{Z}aslow formula}, J.
  Differential Geom., 90 (2012), pp.~439--472.\\
\bibitem{Tzeng_Li}
{\sc Y.-J. Tzeng and J.~Li}, {\em Universal polynoials for singular curves on
  surfaces}, Compos. Math., 150 (2014), pp.~1169--1182. \\ 
\bibitem{Van}
{\sc I.~Vainsencher}, {\em Enumeration of {$n$}-fold tangent hyperplanes to a
  surface}, J. Algebraic Geom., 4 (1995), pp.~503--526. \\
\bibitem{Zu}
{\sc H.~Zeuthen}, {\em Almindelige egenskaber ved systemer af plane kurver},
  Kongelige Danske Videnskabernes Selskabs Skrifter, 10 (1873), pp.~285--393. \\
\bibitem{Zin}{\sc A.~Zinger}, {\em Counting plane rational curves: old and new approaches},
available at \\ \url{http://arxiv.org/abs/math/0507105}. \\ 
\bibitem{Zing_Notes}
{\sc A.~Zinger}, {\em Notes on enumerative geometry}, \newblock available at \\ 
\url{http://www.math.stonybrook.edu/~azinger/mat620/EGnotes.pdf}. 
  
  


\end{thebibliography}

\end{document}